%% file: main.tex
\documentclass[11pt, a4paper]{amsart}

\input{preamble}

\usepackage[alphabetic]{amsrefs}

\title{Perturbed Traceless SU(2) Character Varieties of Tangle Sums}
\author{Kai Smith}
\email{kaismith@amss.ac.cn}

\begin{document}

\begin{abstract}

If a link $L$ can be decomposed into the union of two tangles $T\cup_{S^2} S$ along a 2-sphere intersecting $L$ in 4 points, then the intersections of perturbed traceless $\SU$ character varieties of tangles in a space called the pillowcase form a set of generators for Kronheimer and Mrowka's reduced singular instanton homology, $I^\natural$. It is conjectured by Cazassus, Herald, Kirk, and Kotelskiy that with the addition of bounding cochains, the differential of $I^\natural$ can be recovered from these Lagrangians as well.
This article gives a method to compute the perturbed character variety for a large class of tangles using cut-and-paste methods. In particular, given two tangles, $T$ and $S$, Conway defines the tangle sum $T+S$. Given the character varieties of $T$ and $S$, we show how to construct the perturbed character variety of $T+S$. This is done by first studying the perturbed character variety of a certain tangle $C_3$ properly embedded in $S^3$ with 3 balls removed. Using these results, we prove a nontriviality result for the bounding cochains in the conjecture of Cazassus, Herald, Kirk, and Kotelskiy.

\end{abstract}

\maketitle

\tableofcontents

\section{Introduction}

Given a link $L$ in a 3-manifold $X$, one can define an associated space, $\RR(X, L)$, the traceless $\SU$ character variety. For certain perturbation data, one can also define a perturbed traceless $\SU$ character variety, $\RR_\pi(X, L)$. One of the main reasons why the perturbed character variety is studied is its relation to instanton gauge theory \cites{Taubes, Herald, KMinst, KM}. In particular, there is a slight modification of the character variety, $\RR^\natural_\pi(X, L)$, whose points naturally correspond to generators of Kronheimer and Mrowka's reduced singular instanton homology $I^\natural(L)$. By decomposing the link $L$ into tangles and understanding the perturbed traceless $\SU$ character varieties of these tangles, it is conjectured that the differentials of $CI^\natural$ can be computed as well and therefore $I^\natural$ could be recovered by Lagrangian Floer methods \cites{WW, PCI, PCII, Earring}. This is a case of the Atiyah-Floer conjecture which posits that every instanton gauge theory has an isomorphic Lagrangian Floer theory \cite{AtiyahFloer}. The strategy put forth in \cites{PCI, PCII, Earring} is particularly appealing because the Lagrangian Floer homology is computed between two curves in a 2-dimensional symplectic manifold, $P$, called the pillowcase. As a result, the usual symplectic methods used to count differentials can be replaced by simpler methods from differential topology.

Traceless $\SU$ character varieties have been computed for several classes of tangles including rational tangles and specific families of tangles that come from certain decompositions of torus knots and pretzel knots \cites{PCI, PCII, FKP}.
One strategy for computing character varieties for a wider range of tangles is to use cut-and-paste operations to decompose the tangles into simpler pieces and combine the character varieties of each piece to get the character variety of the original tangle. The simplest such gluing operation for tangles is provided by Conway \cite{Conway} in the form of the \emph{tangle sum}, shown in Figure \ref{fig:TangleSum}. In particular, given two tangles, $T_1$ and $T_2$, the tangle sum $T_1+T_2$ is the result of gluing $T_1$ and $T_2$ into the tangle shown in Figure \ref{fig:C3Intro}, which we label $C_3$. The first goal of this article is to compute, given the perturbed traceless character varieties $\RR_{\pi_1}(T_1)$ and $\RR_{\pi_2}(T_2)$, the character variety $\RR_{\pi_1\cup\pi_2}(T_1+T_2)$. In addition, we compute the image of this character variety in the pillowcase.

For a traceless $\SU$ representation $\rho$, let $\widetilde{Stab}(\rho) := (\Stab(\rho), \Stab(\rho|_{\partial}))$. Composing the natural maps $p:\RR_{\pi_i}(T_i)\to P$ with a projection onto one of the coordinates of the pillowcase gives maps $\RR_{\pi_i}(T_i)\to [0,\pi]$. Let $\RR_{\pi_1}(T_1)\times_{[0,\pi]} \RR_{\pi_2}(T_2)$ be the fiber product of these maps.

\begin{mythm}{\ref{thm:notnice}}
There is a surjective map $\phi: \RR_{\pi_1\cup\pi_2}(T_1+T_2) \to \RR_{\pi_1}(T_1)\times_{[0,\pi]} \RR_{\pi_2}(T_2)$ and the fiber over $(\rho_1, \rho_2)$ is homeomorphic to:

\begin{itemize}
    \item $S^1$ if $\gamma(\rho_1) \in \{0,\pi\}$ and $\wStab(\rho_1) = \wStab(\rho_2) = (\Z/2,\Z/2)$. Parameterize the fiber by $[\rho_1+\rho_2]_\psi$ for $\psi\in S^1$. Then \[p([\rho_1+\rho_2]_\psi) = (\gamma(\rho_1),\arccos(\cos\theta(\rho_1)\cos\theta(\rho_2) +\sin\theta(\rho_1)\sin\theta(\rho_2)\cos\psi)).\] 
    
    \item $S^1$ if $\wStab(\rho_i) = (\Z/2,\U)$ and $\wStab(\rho_j)$  is  $(\Z/2,\U)$ or $(\Z/2, \Z/2)$. For any $\rho$ in this fiber over $(\rho_1,\rho_2)$, \[p(\rho) = (\gamma(\rho_1), \theta(\rho_1)+\theta(\rho_2)).\]
    
    \item $\pt$ otherwise. If $\rho$ is the unique representation in this fiber over $(\rho_1,\rho_2)$, then \[p(\rho) = (\gamma(\rho_1), \theta(\rho_1)+\theta(\rho_2)).\]
\end{itemize}
\end{mythm}

\begin{figure}
    \centering
    \includegraphics[height=1.2in]{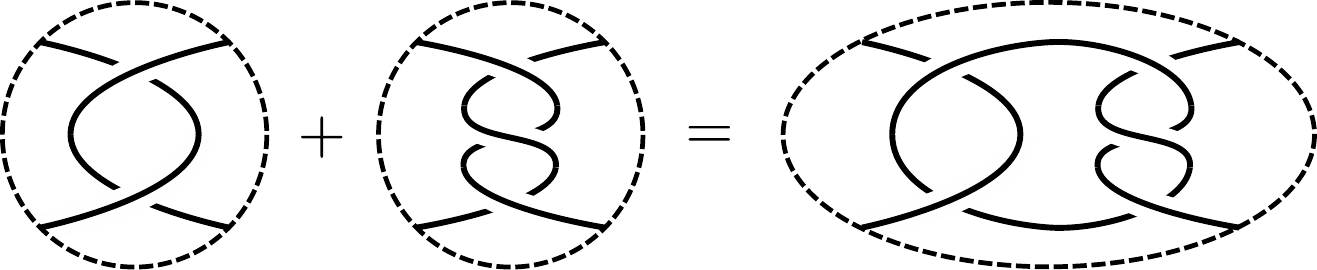}
    \caption{The tangles $T$, $S$ and the tangle sum $T+S$.}
    \label{fig:TangleSum}
\end{figure}

\begin{figure}
    \centering
    \begin{subfigure}[b]{0.3\textwidth}
        \centering
        \includegraphics[width=\textwidth]{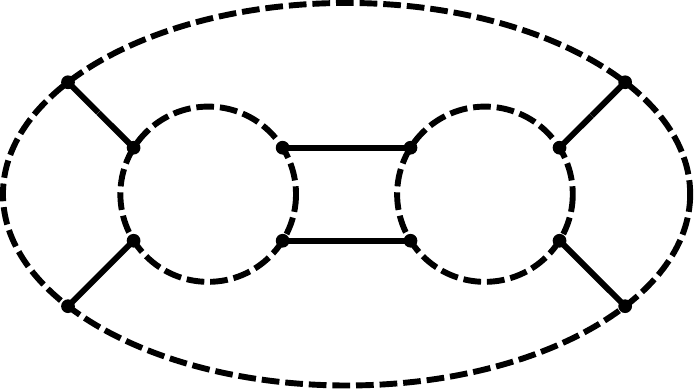}
        \caption{}
        \label{fig:C3Intro}
    \end{subfigure}
    \hspace{1cm}
    \begin{subfigure}[b]{0.3\textwidth}
        \centering
        \includegraphics[width=\textwidth]{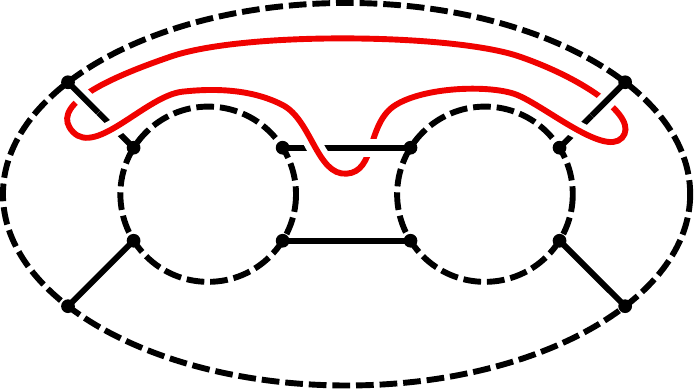}
        \caption{}
        \label{fig:C3PertIntro}
    \end{subfigure}
    \caption{A diagram for the tangle $C_3$. On the right is $C_3$ along with curve $\Pert$ which defines a perturbation of $\RR(C_3)$.}
\end{figure}

Even if the perturbations $\pi_1$ and $\pi_2$ are nontrivial, in almost all cases, the resulting character variety $\RR_{\pi_1\cup\pi_2}(T+S)$ is not a manifold. The next goal is to apply suitable perturbations so that the perturbed character variety is manifold and maps into the pillowcase as a Lagrangian.

The first step in this process is to show that arbitrarily small perturbations $\pi_1$ and $\pi_2$ can be applied to the tangles $T$ and $S$ respectively such that the character variety $\RR(T+S)$ obtained from Theorem \ref{thm:notnice} has a particularly simple structure.

\begin{mythm}{\ref{thm:gpsum}}
    Let $T$ and $S$ be tangles. Then there exists arbitrarily small perturbation data $\pi_1$ in $T$ and $\pi_2$ in $S$ such that $\RR_{\pi_1\cup \pi_2}(T+S)$ has the following structure:

    There exists $N'\subset \RR_{\pi_1\cup \pi_2}(T+S)$ which is a finite set of points such that $\RR_{\pi_1\cup \pi_2}\setminus N'$ is a 1-manifold with boundary. For every point $x\in N'$, there is a neighborhood of $x$ in $\RR_{\pi_1\cup \pi_2}(T+S)$ that is homeomorphic to a cone on four points. The points in $N'$ come in pairs such that if $x$ and $y$ are points in $N'$ which form a pair, then there exists a circle embedded in $\RR_{\pi_1\cup \pi_2}$ containing $x$ and $y$ and no other points from $N'$. The image of this circle in the pillowcase lies entirely within the edges $\{(\gamma, \theta)\in P \mid \gamma \in \{0,\pi\}, \theta \in(0,\pi)\}$. Note that this image does not intersect any corner of the pillowcase. One of the points $x$ or $y$ will attain the maximum $\theta$-coordinate of this image while the other will attain the minimum $\theta$-coordinate.
\end{mythm}

\begin{figure}
    \centering
    \begin{subfigure}[b]{0.25\textwidth}
        \centering
        \includegraphics[width=\textwidth, trim={0 .1cm 0 .1cm}, clip]{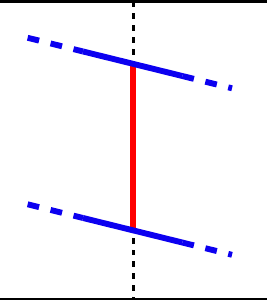}
        \caption{}
        \label{fig:Typical}
    \end{subfigure}
    \hspace{1.5cm}
    \begin{subfigure}[b]{0.25\textwidth}
        \centering
        \includegraphics[width=\textwidth, trim={0 .1cm 0 .1cm}, clip]{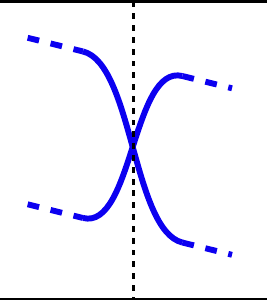}
        \caption{}
        \label{fig:TypicalResolve}
    \end{subfigure}
    \caption{Theorem \ref{thm:gpsum} tells us that a perturbed character variety $\RR_\pi(T+S)$ is a 1-manifold except for a collection of subspaces which map to the pillowcase to look like (A). Theorem \ref{thm:cross} says that by applying an additional perturbation $\pi'$, the character variety $\RR_{\pi+\pi'}(T+S)$ can be computed up to regular homotopy by replacing the regions that look like (A) by regions that look like (B).}
\end{figure}

Figure \ref{fig:Typical} depicts the neighborhood of one of these circles after being mapped to the pillowcase.
We identify a specific perturbation $D$ on the tangle $C_3$ shown in Figure \ref{fig:C3PertIntro} and study the structure of the perturbed character variety $\RR_D(C_3)$. Then we can define a perturbation $D'$ arbitrarily close to $D$ such that $\RR_{D'\cup \pi_1\cup\pi_2}(T_1+T_2)$ is a manifold. The regular homotopy type of the image of this character variety is identified, which in turn gives identifies the manifold $\RR_{D'\cup \pi_1\cup\pi_2}(T_1+T_2)$ itself.

\begin{idea}{\ref{thm:cross}}

For tangles $T_1$ and $T_2$, take perturbations $\pi_1$ and $\pi_2$ such that the conclusion of Theorem \ref{thm:gpsum} holds. Then for each circle described in the theorem, take a neighborhood of the circle in the pillowcase (which looks like Figure \ref{fig:Typical}), remove it and replace it by Figure \ref{fig:TypicalResolve}. It is possible that there are additional circle components, but as the perturbation goes to zero, each must shrink to a point. The result of this procedure is regularly homotopic to the image of $\RR_{D'\cup \pi_1\cup\pi_2}(T_1+T_2)$ in the pillowcase.

\end{idea}

We apply these results to a program initiated by Hedden, Herald, and Kirk which aims to define a link invariant by decomposing the link into two tangles and computing the Lagrangian Floer homology of their traceless $\SU$ character varieties in the pillowcase \cites{PCI, PCII, FKP, Earring}. The Lagrangian Floer homology itself is not a link invariant, but a conjecture of Cazassus, Herald, Kirk, and Kotelskiy posits that by considering extra data in the form of \emph{bounding cochains} an invariant can be recovered. Furthermore, this invariant is conjectured to be an Atiyah-Floer counterpart to $I^\natural$.

\begin{myconjecture}{\ref{con:boundingcc}}[\cite{Earring}*{Conjecture D}]
There exists an assignment that associates to every 2-tangle $T$ and its holonomy perturbed traceless character variety $\RR_\pi(T)\looparrowright P^*$ a bounding cochain $b\in CF(\RR_\pi(T), \RR_\pi(T))$, satisfying

\begin{itemize}
    \item $(R_\pi(T), b)$ is a well defined tangle invariant, as an object of the wrapped Fukaya category of $P^*$;
    \item This assignment of bounding cochains extends to tangles modified by the earring, which results in a tangle invariant $(\RR^\natural_\pi(T), b)$;
    \item Given a decomposition of a link $(S^3, L) = (D^3, T_1)\cup_{(S^2, 4)}(D^3, T_2)$, the corresponding Lagrangian Floer homology recovers the reduced singular instanton homology: \[HF((\RR_\pi(T_1), b_1), (\RR^\natural_\pi(T_2), b_2))\cong I^\natural(L).\]
\end{itemize}
\end{myconjecture}

\begin{figure}
    \centering
    \begin{subfigure}[b]{0.35\textwidth}
        \centering
        \includegraphics[width=\textwidth]{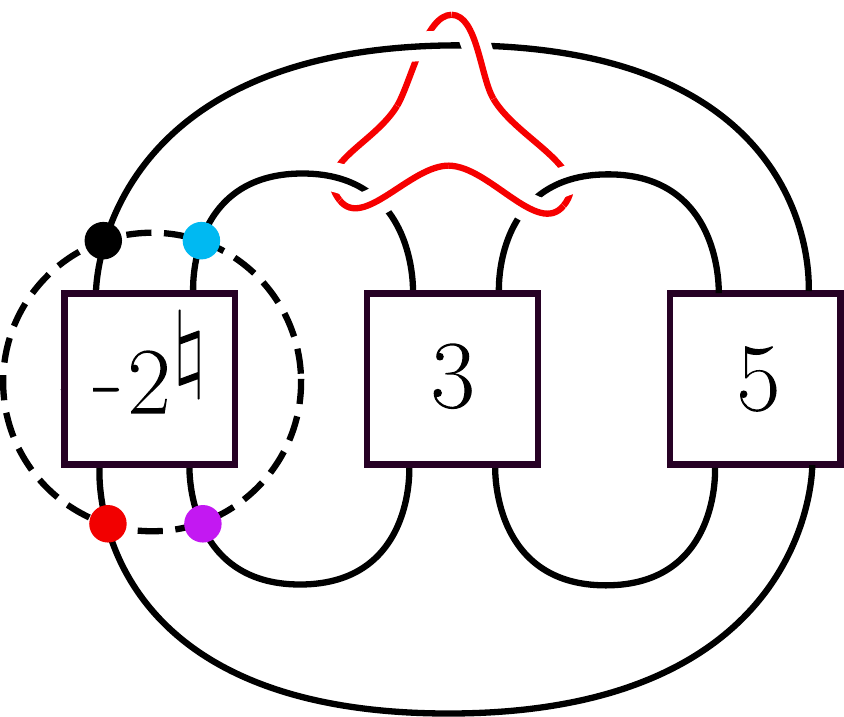}
        \caption{}
        \label{fig:ExDecomp}
    \end{subfigure}
    \hspace{1cm}
    \begin{subfigure}[b]{0.4\textwidth}
        \centering
        \includegraphics[height=2.2in]{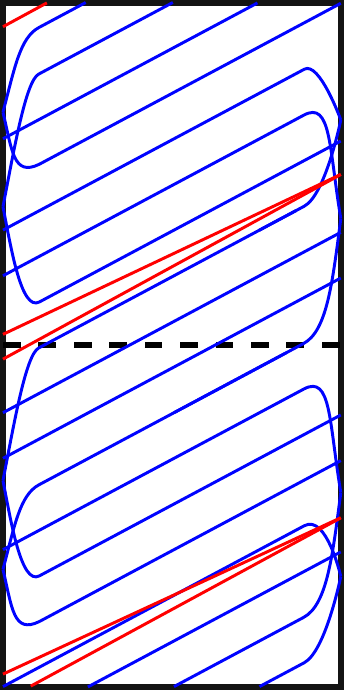}
        \caption{}
        \label{fig:comp}
    \end{subfigure}
    \caption{A decomposition of $P(-2,3,5)$ and the corresponding immersed Lagrangians of pillowcase.}
    \label{fig:compf}
\end{figure}

There is an example in \cite{Earring} which shows that there exist tangles $T$ such that $b^\natural_T\in CF(\RR^\natural_\pi(T), \RR^\natural_\pi(T))$ is nontrivial. In Section 11.1 the authors speculate that there should exist tangles $T$ such that $b_T\in CF(\RR_\pi(T), \RR_\pi(T))$ is nontrivial. To confirm this, we examine the pretzel knot $P(-2,3,5)$, decomposed into two links as shown in Figure \ref{fig:ExDecomp}. Using Theorems \ref{thm:notnice} and \ref{thm:cross}, the corresponding Lagrangians in the pillowcase can be computed, which is shown in Figure \ref{fig:comp}. By calculating the Lagrangian Floer homology of this pair (without bounding cochains) and comparing the result to $I^\natural$, a tangle $T$ is identified for which $b_T$ must be nontrivial.

\begin{mythm}{\ref{thm:main}}
If Conjecture \ref{con:boundingcc} holds, then there exist tangles $T$ such that $b_T$ is nontrivial. 
\end{mythm}

\subsection*{Acknowledgments}
The author would like to thank Paul Kirk for suggesting the initial problem which grew into this article and for the countless hours of helpful discussion. Additionally, the author would like to thank Chris Herald, Mark Ronnenberg, and Zuyi Zhang for their constructive comments.

\section{Background}
\label{sec:bg}
\subsection{The Lie Group \texorpdfstring{$\SU$}{SU(2)}}

$\SU$ is the Lie group of complex valued 2 by 2 matrices with determinant 1. Any such matrix can be written in the form $A = \bmat{a+b\i&c+d\i\\-c+d\i&a-b\i}$. The map $\phi:\SU\to \mathbb{H}$ given by $A\mapsto a+b\i+c\j+d\k$ is a Lie group isomorphism from $\SU$ to the unit quaternions. In this context, it is easy to see that $\SU$ is homeomorphic to $S^3$ embedded in the space of all quaternions, which can be identified with $\R^4$. 

For a matrix $A\in\SU$, the trace is given by $\tr(A) = 2a = 2\Re(\phi(A))$. Conjugating by an element of $\SU$ preserves the trace. Furthermore, this action is transitive on subsets with the same trace, and so the conjugacy classes of $\SU$ are completely determined by the trace. For $x\in \SU$, let $C(x)$ denote the conjugacy class of $x$. One conjugacy class of particular interest is $C(\i)$, the \textit{traceless} unit quaternions. Taking $\{\i,\j,\k\}$ as the basis for $\R^3$,$C(\i)$ is precisely the unit 2-sphere. Note that for $x\in C(\i)$, $x\inv = -x$. Going forward, $\ol{x}$ will be used as more compact notation for $x\inv$.

\begin{definition}
\label{def:coeq}
Call a set of points in $C(\i)$ \textit{coequatorial} if all of the points in the set lie on a single great circle in $C(\i)$.
\end{definition}

Since conjugation preserves trace, $C(\i)$ is preserved by conjugation by $x\in\SU$. More explicitly, if $Q\in C(\i)$ and $t\in \R$, conjugation by $x = e^{tQ}$ acts on $C(\i)$ by a rotation of angle $2t$ about the axis through $Q$.

\begin{definition}
Let $\Sk$ denote the great circle of $C(\i)$ which is perpendicular to $\k$, \[\{e^{\theta\k}\i \in \SU\mid \theta\in S^1\}.\] 
\end{definition}

\begin{definition}
\label{def:ang}
Given $x, y\in \SU$ let the \emph{angle between $x$ and $y$} be the angle of the corresponding vectors in $\R^4$, $\arccos(\mathbf{x}\cdot\mathbf{y})\in[0,\pi]$. Denote the angle between $x$ and $y$ by $\ang{x}{y}$.
\end{definition}

\begin{definition}
If $x$, $y$, and $z$ lie on the same great circle of $C(\i)$, define the \emph{angle of $x$ and $y$ relative to $z$} to be 
\[\rang{x}{y}{z} := \begin{cases}\ang{x}{y}& \text{if }\ang{x}{z}+\ang{z}{y} = \ang{x}{y} \text{ or }\ang{x}{z}-\ang{z}{y} = \ang{x}{y}\\
2\pi - \ang{x}{y} & \text{otherwise}\end{cases}\]
\end{definition}

Intuitively, for any two points $x$ and $y$ on a circle such that $y\notin\{\pm x\}$, there are two different angles between $x$ and $y$ depending on which path one takes around the circle to get from $x$ to $y$. By Definition \ref{def:ang}, $\ang{x}{y}$ is always the smaller of these two angles. But sometimes we will want to specify which of these two angles we are talking about, which is what the relative angle allows us to do. Given a third point on the circle $z$ which is also not in $\{\pm x\}$, taking the shortest path from $x$ to $z$ defines an orientation on the circle. Then $\rang{x}{y}{z}$ is defined to be the angle spanned by the path from $x$ to $y$ on the circle following this orientation.
If $y$ or $z$ is in $\{\pm x\}$, then $\rang{x}{y}{z}$ is defined to be $\ang{x}{y}$.

\begin{lemma}
\label{lem:relang}
Assume $x$, $y$, and $z$ lie on the same great circle in $C(\i)$. Then there exists $w\in\SU$ such that $wx\ol{w} = \i$, $wz\ol{w} = e^{\gamma\k}\i$ for $\gamma\in[0,\pi]$, and $wy\ol{w} = e^{\theta\k}\i$ for $\theta\in[0,2\pi)$. If $\sin(\gamma)=0$, then $w$ can be chosen to ensure that $\theta\in[0,\pi]$. 
Then $\ang{x}{z} = \gamma$ and $\rang{x}{y}{z} = \theta$.
\end{lemma}

\begin{proof}
Because conjugation is transitive on $C(\i)$, there exists some $w'\in \SU$ such that $w'x\ol{w'} = \i$. Further conjugating by a quaternion of the form $e^{\delta\i}$ can guarantee the hypothesis. So $w = e^{\delta\i}w'e^{-\delta\i}$.

Because angles are invariant under rotation, $\ang{x}{z} = \ang{(wx\ol{w})}{(wz\ol{w})} = \gamma$. The invariance of angles under rotation also implies that 
\begin{align*}
\rang{x}{y}{z} &= \rang{(wx\ol{w})}{(wy\ol{w})}{wz\ol{w}} \\&= \begin{cases}\arccos(\cos(\theta))& \text{if }\gamma\pm\arccos(\cos(\gamma-\theta)) = \arccos(\cos(\theta)) \\
2\pi - \arccos(\cos(\theta)) & \text{otherwise}\end{cases}
\end{align*}

This splits into four cases depending on the signs of $\cos(\theta)$ and $\cos(\gamma-\theta)$. Checking these cases shows that this expression is always equal to $\theta$.
\end{proof}

\subsection{Tangles}

\begin{definition}
A \emph{generalized tangle} is a pair $(X, T)$ where $X$ is a compact manifold with (possibly empty) boundary and $T$ is a codimension-2, properly embedded submanifold with (possibly empty) boundary.
\end{definition}

If $X$ is clear from context, we will often write $T$ as shorthand for $(X,T)$. Additionally, $\partial(X, T)$ is used to denote $(\partial X, \partial T)$.

One important generalized tangle to consider is the case where $X = S^2$ and $T = \{a,b,c,d\}$ is a set of four points. For brevity, we denote this generalized tangle $(S^2, 4)$. Let $\mathbf{S}$ be a specific model for $(S^2, 4)$ embedded in $\R^3$ given by the unit 2-sphere centered at the origin with $a = (-\frac{1}{\sqrt{2}}, \frac{1}{\sqrt{2}},0)$, $b = (-\frac{1}{\sqrt{2}}, -\frac{1}{\sqrt{2}},0)$, $c = (\frac{1}{\sqrt{2}}, \frac{1}{\sqrt{2}},0)$, and $d = (\frac{1}{\sqrt{2}}, -\frac{1}{\sqrt{2}},0)$.

\begin{definition}
\label{def:tangle}
Let $(X, T)$ be a generalized tangle where $X$ is a 3-manifold and $\partial(X, T) = \bigsqcup (S^2, 4)$. Label each boundary component $S^T_i$. Let $\{\psi^T_i\}$ be a collection of homeomorphisms $\psi^T_i:S^T_i\to \mathbf{S}$.
The data $((X, T), \{\psi^T_i\})$ is called a \emph{tangle}.
\end{definition}

\begin{definition}
\label{def:equiv}
Let $(X_1, T_1)$ and $(X_2, T_2)$ be tangles such that $\partial(X_i, T_i) \cong \bigsqcup_{i=1}^n (S^2, 4)$. Let $\phi: (X_1, T_1) \to (X_2, T_2)$ be a diffeomorphism and let $\sigma$ be the permutation on $\{1, 2, \dots, n\}$ satisfying $\psi(S^{T_1}_i) = S^{T_2}_{\sigma(i)}$. Then $\phi$ is called a \emph{tangle equivalence} if $\psi^{T_1}_i$ is isotopic to $\psi^{T_2}_{\sigma(i)}\circ \phi$ relative to $\{a, b, c, d\}$.
Two tangles are called \emph{equivalent} if there exists a tangle equivalence between them.
\end{definition}

Given two tangles $(X_1, T_1)$ and $(X_2, T_2)$, a new tangle can be constructed by gluing $S^{T_1}_i$ to $S^{T_2}_j$, which is given by taking the pushout of the diagram:

\begin{center}
\begin{tikzcd}
\mathbf{S} \ar[r,"(\psi^{T_1}_i)\inv"]\ar[d, "(\psi^{T_2}_j)\inv"']& S^{T_1}_i \ar[r, hookrightarrow] & (X_1, T_1)\\
S^{T_2}_j \ar[d,hookrightarrow]&&\\
(X_2, T_2)&&
\end{tikzcd}
\end{center}

For any diagram of a tangle in this article, the boundary $(S^2,4)$ components are drawn as dashed circles with four colored marked points ($\fcirc{black}$, $\fcirc{red}$, $\fcirc{cyan}$, $\fcirc{violet}$). To visualize the tangle, one should imagine 2-spheres intersecting the page at each such circle with marked points. The diagram implicitly defines a map $\psi: (S^2, 4)\to \mathbf{S}$ for each boundary component (up to homotopy relative to the four marked points) satisfying the following conditions:

\begin{itemize}
    \item The circle in the tangle diagram is mapped to the unit circle $\{(x, y, z)\in \mathbf{S} \mid z=0\}$ such that $\psi(\fcirc{black}) =a$, $\psi(\fcirc{red}) = b$, $\psi(\fcirc{cyan}) = c$, and $\psi(\fcirc{violet}) = d$.

    \item The open hemisphere coming out of the page, toward the reader, is mapped to $\{(x, y, z)\in \mathbf{S} \mid z>0\}$.

    \item The other open hemisphere is mapped to $\{(x, y, z)\in \mathbf{S} \mid z<0\}$.
\end{itemize}

\begin{definition}
If $(X, T)$ is a tangle such that $T\cong S^0\times I$, then $(X, T)$ is called a \emph{2-stranded tangle}.
\end{definition}

\subsection{Traceless \texorpdfstring{$\SU$}{SU(2)} Character Varieties}
\label{sec:charvar}

\begin{definition}
Given a generalized tangle, $(X, T)$, define the \emph{traceless $\SU$ representation variety} of $(X,T)$ to be \[\Ro(X, T):= \{\rho\in \Hom(\pi_1(X\setminus T), \SU)\mid \tr(\rho(m))=0 \text{ for each } m \text{ a meridian of } T\}.\]
\end{definition}

\begin{definition}
The \emph{traceless $\SU$ character variety} of the generalized tangle $(X,T)$ is \[\RR(X,T) := \Ro(X, T)/{\sim}\] %
where $\rho_1\sim\rho_2$ if  there exists an $x\in \SU$ such that $\rho_2(\gamma) = x\rho_1(\gamma)\ol{x}$ for every $\gamma\in \pi_1(X\setminus T)$.
\end{definition}

\begin{remark}
For a topological space $X$, $\Ro(X)$ and $\RR(X)$ are used as shorthand for $\Ro(X,\emptyset)$ and $\RR(X,\emptyset)$ respectively. Since no traceless condition is being imposed in these cases, these are simply known as the \emph{$\SU$ representation variety} of $X$ and the \emph{$\SU$ character variety} of $X$ respectively.
\end{remark}

\begin{remark}
If the generalized tangle $(X, T)$ is equivalent to a disjoint sum $\bigsqcup_i (X_i, T_i)$, then $\Ro(X, T)$ is defined to be the product $\prod_i \Ro(X_i, T_i)$.
\end{remark}

Working directly with conjugacy classes can be tedious. To avoid this issue, one can use a process called gauge fixing which identifies each conjugacy class $[\rho]\in\RR(X, T)$ with a unique representation in $\Ro(X, T)$.
Assume $X$ is simply connected. Then the meridians of $T$ generate $\pi_1(X\setminus T)$. Let $\boldsymbol{\mu} = (\mu_1, \dots, \mu_n)$ be an ordered set of meridians of $T$ which generate $\pi_1(X\setminus T)$.
An element $[\rho]\in \RR(X,T)$ is determined by the image of $\boldsymbol{\mu}$ under a representative $\rho$. Thus, an element of $\RR(X, T)$ is determined by a set of marked points on $C(\i)$, given by the images of each $\mu_i$ under a representative $\rho$.

\begin{definition}
A traceless representation $\rho\in\Ro(X, T)$ is in \emph{standard gauge} with respect to $\boldsymbol{\mu}=(\mu_1, \dots, \mu_n)$ if

\begin{itemize}
    \item $\rho(\mu_1) = \i$;
    \item $\rho(\mu_i) \in \{\pm\i\}$ for $1<i\le n$
\end{itemize}
or if there exists $m\in\Z$, $1<m<n$ such that:
\begin{itemize}
    \item $\rho(\mu_1) = \i$;
    \item $\rho(\mu_i) \in \{\pm\i\}$ for $1<i<m$;
    \item $\rho(\mu_m) = e^{\phi \k}\i$ for $\phi \in (0,\pi)$.
\end{itemize}
\end{definition}

\begin{definition}
Let \[\Ro_{\boldsymbol{\mu}}(X, T) := \{\rho\in\Ro(X, T) \mid \rho \text{ is in standard gauge with respect to } \boldsymbol{\mu}\}\]
\end{definition}

\begin{lemma}
\label{lem:gaugefixing}
If $X$ is simply connected and $T\ne \emptyset$, for any ordering $\boldsymbol{\mu}$, there is a bijection $f:\Ro_{\boldsymbol{\mu}}(X, T)\to \RR(X,T)$.%
\end{lemma}

\begin{proof}
Let $f(\rho) = [\rho]$. To define $f\inv([\rho])$, start by picking a representative $\rho\in[\rho]$.
Conjugation by $\SU$ acts by rotation on $C(\i)$ and each rotation is achievable by some conjugation. Thus for any representation $\rho$, there is a conjugate representation $\rho'$ such that $\rho'(\mu_1) = \i$. Let $m$ be the lowest number, if it exists, such that $\rho'(\mu_m)\ne \pm\i$. Then $\rho'$ can be conjugated about the $\i$ axis to obtain a new representation $\rho''$ such that $\rho''(\mu_m) = e^{\phi\k}\i$ for $\phi\in(0,\pi)$. It is clear that any further rotation either fixes $\rho''$ or brings it to a new representation which is not in standard gauge, thus $p''$ is uniquely defined.

\end{proof}

\begin{lemma}
$\RR$ is an invariant of tangle equivalence.
\end{lemma}

\begin{proof}
Clear from Definition \ref{def:equiv}.
\end{proof}

\begin{remark}
If $(Y, S)$ is a submanifold of $(X, T)$, the inclusion $\pi_1(Y\setminus S)\to \pi_1(X\setminus T)$ induces maps $\RR(X, T)\to \RR(Y, S)$ and $\Ro(X, T)\to \Ro(Y, S)$. Denote the image of $\rho$ under either of these maps by $\rho|_{(Y, S)}$.
\end{remark}

\begin{remark}
The definition of standard gauge can be modified to work with tangles $(X, T)$ even when $X$ is not simply connected or $T=\emptyset$, but this is not necessary for the present article.
\end{remark}

\begin{definition}
\label{def:bd}
A representation $\rho\in \RR(X, T)$ is called \textit{binary dihedral} if $\rho$ sends all of the meridians of $T$ to a coequatorial set. 
Let $\RR\bd(X, T) := \{[\rho]\in\RR(X, T)\mid \rho \text{ is binary dihedral}\}$ and $\RR\nbd(X,T) := \RR(X, T)\setminus \RR\bd(X, T)$.
\end{definition}

\subsubsection{Tangent Spaces}

For a conjugacy class $[\rho]\in\RR(X)$, Weil \cite{Weil} shows that the \emph{formal tangent space} at $[\rho]$ is identified with the first cohomology:
\[T_{[\rho]}\RR(X) = H^1(\pi_1(X); \su_{\operatorname{ad}\rho}) = H^1(X; \su_{\operatorname{ad}\rho}).\]

For a generalized tangle, the formal tangent space of $[\rho]\in\RR(X,T)$ is given by the following kernel (see \cite{CHK}):
\[T_{[\rho]}\RR(X) = \ker\left[H^1(X; \mathfrak{su}(2)_{\operatorname{ad}\rho})\to H^1(T; \mathfrak{su}(2)_{\operatorname{ad}\rho})\right]\]

\begin{definition}
\label{def:reg}
A point $[\rho]$ in $\RR(X, T)$ is called \emph{regular} if $\dim T_{[\rho]}\RR(X)$ is locally constant in a neighborhood of $[\rho]$. %
\end{definition}

\subsection{The Pillowcase}
\label{sec:pillowcase}

\begin{figure}
    \centering
    \includegraphics[height=2in]{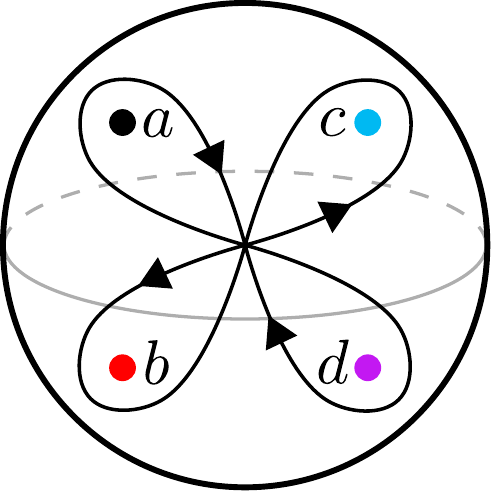}
    \caption{A slightly distorted picture of $\mathbf{S}$ and the loops that generate its fundamental group.}
    \label{fig:4PS}
\end{figure}

A particularly important example is the traceless character variety of $\mathbf{S}\cong (S^2, 4)$. Abusing notation, let $a$, $b$, $c$, and $d$ denote the meridians around the punctures of the same name, oriented as in Figure \ref{fig:4PS}. The fundamental group is generated by $a$, $b$, $c$, and $d$ and there is one relation called the \textit{pillowcase relation}, $a\ol{c}=b\ol{d}$. Thus, $\pi_1(\mathbf{S} \setminus 4) = \{a,b,c,d\mid a\ol{c} = b\ol{d}\}$. 

Let $\rho\in\Ro(\mathbf{S})$. $\rho$ can be conjugated so that $\rho(a) = \i$ and $\rho(b) = e^{\gamma \k}\i$ for $\gamma\in S^1$. Then to see where $c$ can be sent, note that $\rho(d) = \rho(c\ol{a}b) = \rho(c)e^{-\gamma\k}$ must be traceless. Thus, $\rho(c)$ cannot have a real component or a $\k$ component and so $\rho(c) = e^{\theta \k}\i$ for $\theta\in S^1$. Then $\rho(d) = (e^{\theta \k}\i)(-\i) (e^{\gamma \k}\i) =e^{(\gamma+\theta)\k}\i$. The above argument proves the following useful lemma:

\begin{lemma}
\label{lem:boundcond}
For $\rho\in \Ro(\mathbf{S})$, $\{\rho(a), \rho(b), \rho(c), \rho(d)\}$ is a coequatorial set. 
\end{lemma}

The space of representations of the type described above can be parameterized by $(\gamma,\theta)\in\R^2$. Clearly $(\gamma,\theta)$ and $(\gamma+2n\pi,\theta+2m\pi)$ denote the same representation. In addition, $(\gamma,\theta)$ and $(-\gamma,-\theta)$ are related by conjugation by $\i$. It is easy to check that there are no other conjugates.

\begin{lemma}
\label{lem:R2lift}
$\RR(\mathbf{S})$ is a quotient of $\R^2$ by the group of orientation-preserving isometries generated by the maps
\[(\gamma, \theta) \mapsto (\gamma + 2\pi, \theta),~ (\gamma, \theta)\mapsto (\gamma, \theta+2\pi), (\gamma,\theta)\mapsto (-\gamma,-\theta).\]
\end{lemma}

$\RR(\mathbf{S})$ is parameterized by the coordinates $(\gamma, \theta)$ induced by this quotient map. For $\rho\in\RR(\mathbf{S})$, let $\gamma(\rho)$ and $\theta(\rho)$ be the projection onto these coordinates. 
When $\rho\in\Ro(\mathbf{S})$ is in standard gauge, applying Lemma \ref{lem:relang} gives
\begin{gather}
\label{eq:gammaS}\gamma(\rho) = \ang{\rho(a)}{\rho(b)} = \ang{\rho(c)}{\rho(d)}\\
\label{eq:thetaS}\theta(\rho) = \rang{\rho(a)}{\rho(c)}{\rho(b)} = \rang{\rho(b)}{\rho(d)}{\rho(\ol{a})}
\end{gather}
Since these angles are invariant under conjugation of $\rho$ by an element of $\SU$, it is possible to take these as the definition for the coordinates of $\RR(\mathbf{S})$.

\begin{definition}
The space $\RR(\mathbf{S})$ is homeomorphic to a space called the \emph{pillowcase} which is denoted by $P$.
\end{definition}

\begin{figure}
    \centering
    \begin{subfigure}[b]{0.22\textwidth}
        \centering
        \includegraphics[width=\textwidth]{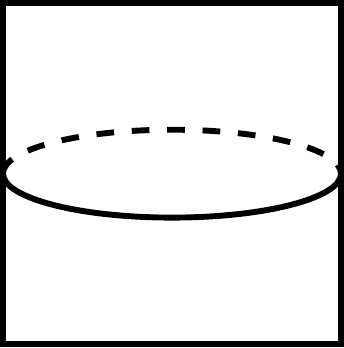}
        \caption{}
        \label{fig:PCase1}
    \end{subfigure}
    \hspace{3cm}
    \begin{subfigure}[b]{0.22\textwidth}
        \centering
        \includegraphics[width=\textwidth]{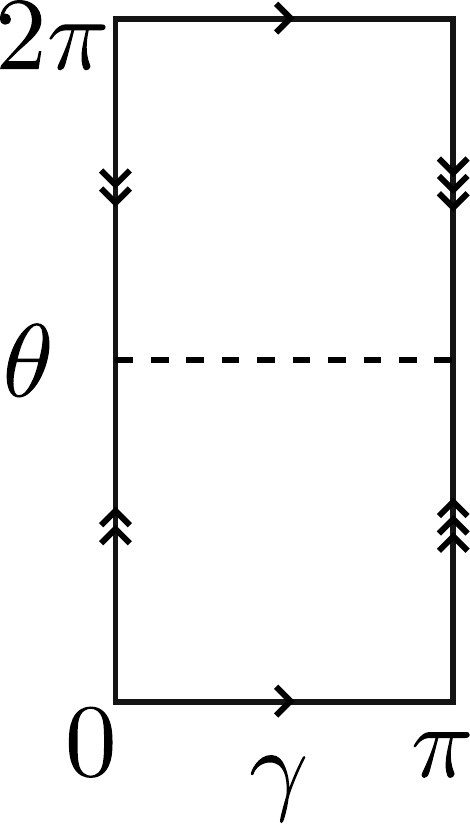}
        \caption{}
        \label{fig:Pcase2}
    \end{subfigure}
    \caption{The pillowcase and its parametrization.}
\end{figure}

A fundamental domain for $\RR(\mathbf{S})$ in $\R^2$ is shown in Figure \ref{fig:Pcase2}. $\RR(\mathbf{S})$ is homeomorphic to a 2-sphere (see Figure \ref{fig:PCase1}). The character variety is regular in the sense of Definition \ref{def:reg} except at the four points $(0,0)$, $(0,\pi)$, $(\pi,0)$, and $(\pi,\pi)$. These points are called the \textit{corners} of the pillowcase and will be denoted by $P^c$. $P^c$ is the 0-dimensional stratum of the pillowcase. Let $P^*$ denote the smooth 2-dimensional stratum of the pillowcase, $P\setminus P^c$. 
$P^*$ has a symplectic structure inherited from the standard symplectic structure $d\gamma\wedge d\theta$ on $\R^2$ via the quotient map.

Now suppose $(X, T)$ is a tangle. Recall that there is a homeomorphism $\psi_i$ from each boundary component $S_i$ to $\mathbf{S}$. Thus, $P_i:=\RR(S_i) \cong P$ and $P_i$ can be parameterized by coordinates $(\gamma_i, \theta_i)$ induced by pulling back the coordinates $(\gamma, \theta)$ on $\RR(\mathbf{S})$ along $\psi_i$.
Let $a_i$, $b_i$, $c_i$, and $d_i$ be the meridians of $S_i$ which map to $a$, $b$, $c$, and $d$ under $\psi$. Then for $\rho\in \RR(S_i)$, Equations \ref{eq:gammaS} and \ref{eq:thetaS} give
\begin{gather}
\label{eq:gamma}\gamma_i(\rho) = \ang{\rho(a_i)}{\rho(b_i)} = \ang{\rho(c_i)}{\rho(d_i)}\\
\label{eq:theta}\theta_i(\rho) = \rang{\rho(a_i)}{\rho(c_i)}{\rho(b_i)} = \rang{\rho(b_i)}{\rho(d_i)}{\rho(\ol{a_i})}
\end{gather}

\begin{definition}
\label{def:pcasecoords}
Let $(X, T)$ be a tangle and let $\rho\in\RR(X, T)$; Then let
\[\gamma_i(\rho):= \gamma_i(\rho|_{S_i})\in[0,\pi],~ \theta_i(\rho) := \theta_i(\rho|_{S_i})\in[0,2\pi]\].

If $\partial(X, T)$ only has one component, $S_1$, then let \[\gamma(\rho):= \gamma_1(\rho|_{S_1})\in[0,\pi],~ \theta(\rho) := \theta_1(\rho|_{S_1})\in[0,2\pi]\]
\end{definition}

\begin{remark}
The conventions chosen for the $\pi_1(\mathbf{S})$ differ from some of the other literature, but are nice since they allow several nice properties:

\begin{itemize}
    \item The character varieties of rational tangles interact nicely with the slope of the tangle (Proposition \ref{prop:rational}).
    \item Tangle addition (in the sense of Conway) corresponds to a natural addition on the pillowcase with this parameterization (Proposition \ref{thm:notniceA}).
\end{itemize}
\end{remark}

\subsection{Holonomy Perturbations}
\label{sec:holpert}

In order to extract information out of character varieties, it is often necessary to perturb them. The types of perturbations considered in this article are called \emph{holonomy perturbations} and are chosen in this context because they agree with the types of perturbations used for instanton gauge theory.

\begin{definition}
\emph{Perturbation data} for the tangle $(X, T)$ is a set of tuples $\{(\Pert'_i, f_i, t_i)\}$ where:
\begin{itemize}
\item $\{\Pert'_i\}$ is a family of pairwise disjoint embedding $S^1\times D^2 \to X\setminus T$;
\item $f_i$ is a smooth, odd, $2\pi$-periodic function;
\item $t_i\in \R$.
\end{itemize}

The resulting perturbed character variety will only depend on the embeddings $\Pert'_i$ up to isotopy class. Thus, it is often convenient to specify the embedding $\Pert'_i$ with an embedding of a circle $\Pert_i: S^1\to X\setminus T$ along with a framing (a nowhere zero choice of normal section).
Each $\Pert_i$ is called a \emph{perturbation curve}. Let $\mu_i$ and $\lambda_i$ be a choice of meridian and longitude of $\Pert_i$ with respect to the specified framing. $f_i$ is called the \emph{perturbation function} for $\Pert_i$ and $t_i$ is called the \emph{perturbation parameter}.
\end{definition}

\begin{definition}
Let $\pi = \{(\Pert_i, f_i, t_i)\}$ be the perturbation data for a tangle $(X, T)$.
Following \cite{Herald}*{Lemma 61}, the \emph{perturbed character variety} $\RR_\pi(X,T)$ is given by the subspace of %
$\RR(X\setminus(\bigcup_i \Pert_i), T)$ satisfying the following \emph{perturbation condition}:

\begin{equation}
\label{eq:holpert}
\text{If } \rho(\lambda_i) = e^{u_iQ_i} \text{ for some } u_i\in \R \text{ and } Q_i\in C(\i) \text{, then }\rho(\mu_i) = e^{t_if_i(u_i)Q_i} 
\end{equation}
\end{definition}

\begin{remark}
For any representation satisfying this condition, $\rho(\mu_i)$ and $\rho(\lambda_i)$ commute.
\end{remark}

Let $\Im: \SU\to \su$ be the projection which sends $w+x\i+y\j+z\k$ to $x\i+y\j+z\k$. Note that $\Im(\rho(\lambda_i)) = \sin(u_i)Q_i$. Therefore, by choosing $f_i = \sin$, the perturbation condition above simplifies to 
\begin{equation}
\label{eq:pert}
\rho(\mu_i) = e^{t\Im(\rho(\lambda_i))}.
\end{equation}
For this article, unless otherwise specified, each $f_i$ is assumed to be $\sin$, and each diagram of a perturbation curve is assumed to have the blackboard framing.

\subsubsection{Shearing Perturbations}
\label{sec:shearing}

\begin{figure}
    \centering
    \begin{subfigure}[b]{0.35\textwidth}
        \centering
        \includegraphics[width=\textwidth]{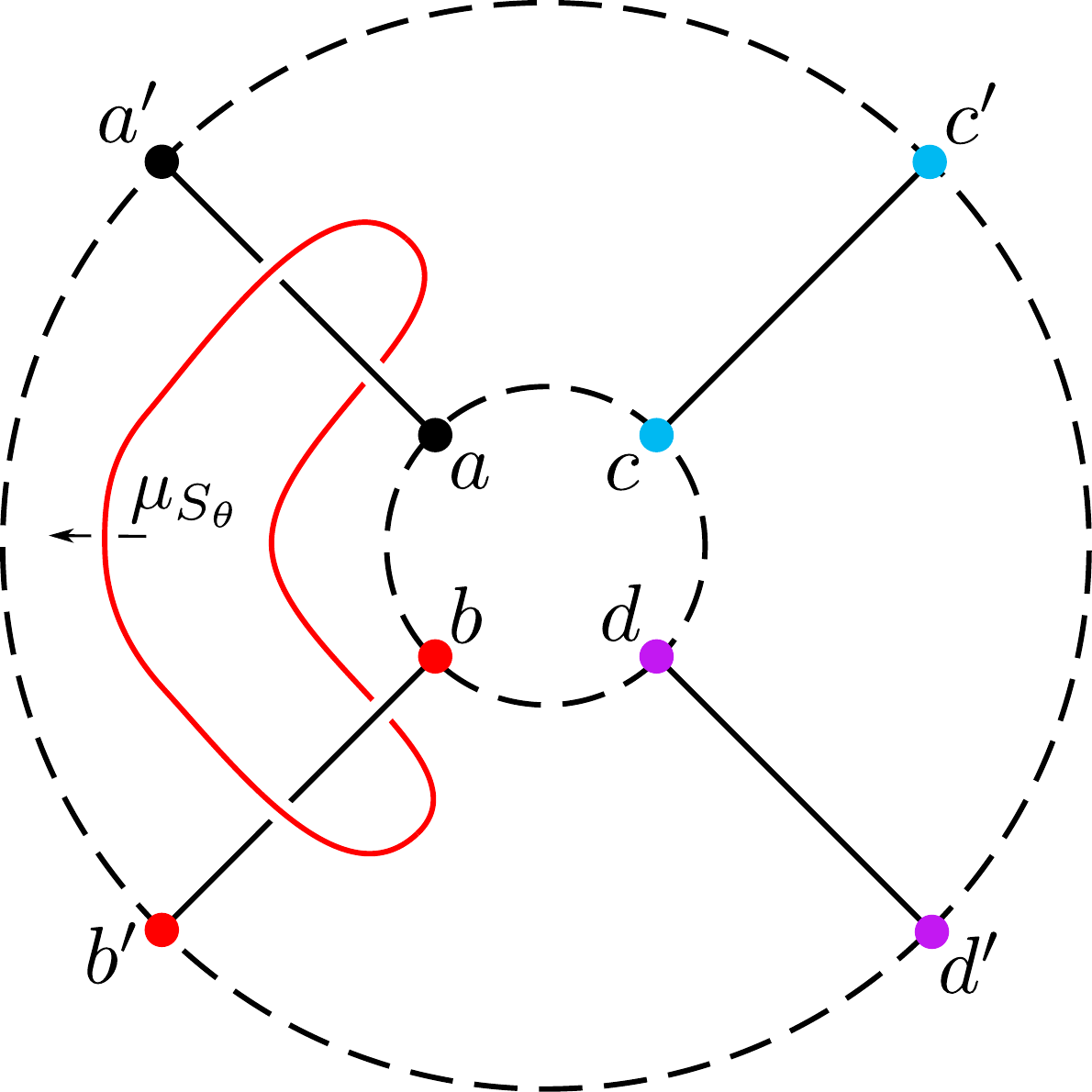}
        \caption{}
        \label{fig:ShearingA}
    \end{subfigure}
    \hspace{2cm}
    \begin{subfigure}[b]{0.35\textwidth}
        \centering
        \includegraphics[width=\textwidth]{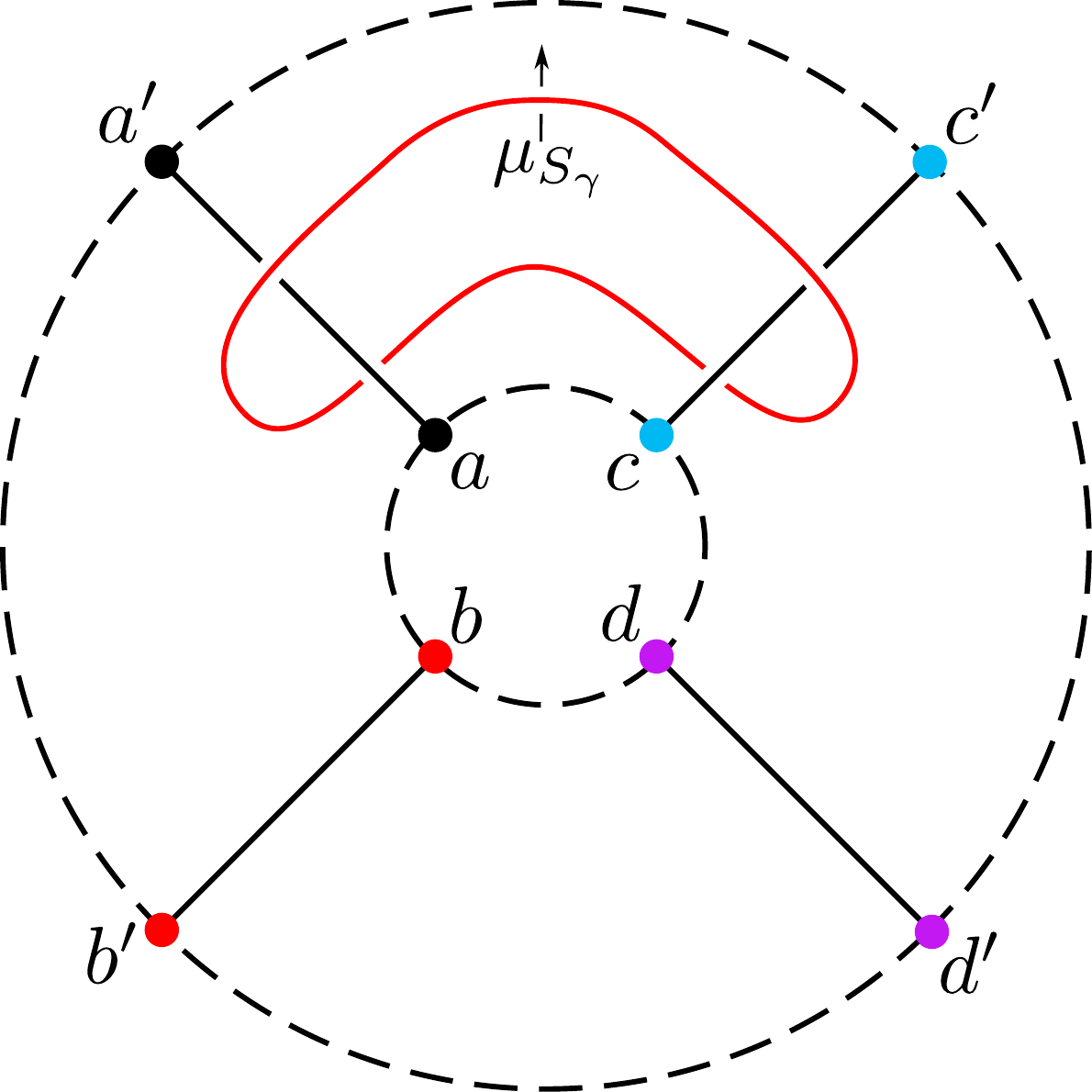}
        \caption{}
        \label{fig:ShearingB}
    \end{subfigure}
    \caption{The perturbation curves which induce the shearing perturbations.}
    \label{fig:Shearing}
\end{figure}

Shearing perturbations are particularly simple and easy to control holonomy perturbations.
Consider the tangle in Figure \ref{fig:ShearingA}, which is $(S^2,4)\times I$ along with the perturbation curve $S_\theta$. \[\rho(\lambda_{S_\theta}) = \rho(\ol{a}b) = e^{-\gamma\k}.\] So the perturbation condition, Equation \ref{eq:holpert}, tells us that \[\rho(\mu_{S_\theta}) = e^{tf(-\gamma)\k}.\] Thus it follows that 
\[\rho(a') = \rho(pa\ol{p}) = e^{2tf(-\gamma)\k}\i\]
\[\rho(b') = \rho(pb\ol{p}) = e^{(2tf(-\gamma)+\gamma)\k}\i.\] Then by Equations \ref{eq:gamma} and \ref{eq:theta}, 
\[\gamma' = \ang{\rho(a')}{\rho(b')} =  \gamma\]
\begin{equation}
\label{eq:sheartheta}
\theta' = \rang{\rho(a')}{\rho(c')}{\rho(b')} = \theta-2tf(\gamma).
\end{equation}
Suppose $(X, T)$ is a tangle such that $\partial(X, T)=(S^2, 4)$. Call the result of gluing $(X, T)$ to the inner $(S^2,4)$ of Figure \ref{fig:ShearingA} \emph{shearing $(X, T)$ in the $\theta$ direction}.

The same calculation shows that for the perturbation curve $S_\gamma$ in Figure \ref{fig:ShearingB}, $\gamma' = \gamma-2tf(\theta)$ and $\theta' = \theta$.
Call the result of gluing $(X, T)$ to the inner $(S^2,4)$ of Figure \ref{fig:ShearingA} \emph{shearing $(X, T)$ in the $\gamma$ direction}.

\subsection{Character Varieties of 2-stranded Tangles}
\label{sec:pcslope}

If $(X, T)$ is a 2-stranded tangle, then $\partial(X, T) = (S^2, 4)$. %
Let $p$ be the map $\RR(X, T)\to \RR(\partial(X, T))\cong P$ and let $\gamma(\rho)$ and $\theta(\rho)$ be the coordinates of $p(\rho)$ as in Definition \ref{def:pcasecoords}.

\begin{lemma}[\cite{FKP}*{Theorem 3.2}]
\label{lem:fkpc}
If $T$ is a 2-stranded tangle in a $\Z/2$ homology 3-ball, then $\RR\bd(T)\cong I \sqcup(\bigsqcup S^1)$. Call the $I$ component of $\RR\bd(T)$ the \textit{arc of binary dihedrals}, or $I\bd$. $p: \RR(T)\to P$ maps $I\bd$ linearly into the pillowcase, meaning its lift from $P$ to $\R^2$ (induced by the quotient map in Lemma \ref{lem:R2lift}) is a line segment.
\end{lemma}

\begin{definition}
\label{def:pcs}
The endpoints of the interval $I\bd$ are central representations, and therefore $I\bd$ can be lifted $\R^2$ in a natural way so that one end is sent to $(0,0)$ and the other to $(\gamma',\theta')\in \Z^2$. Then let $\pcs(X, T):= \frac{\theta'}{\gamma'}$. Call this the \textit{pillowcase slope} of $(X,T)$. 
\end{definition}

\begin{figure}
    \centering
    \includegraphics[height=1.5in]{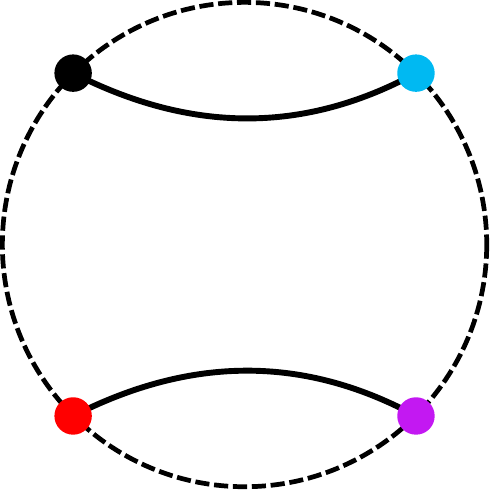}
    \caption{Rational tangle of slope 0, $Q_0$.}
    \label{fig:slope0}
\end{figure}

\begin{definition}
A tangle $(X, T)$ is called a \emph{rational tangle} if it is homeomorphic (but not necessarily equivalent in the sense of Definition \ref{def:equiv}) to the tangle in Figure \ref{fig:slope0}.
\end{definition}

Conway showed that rational tangles are classified by an element of $\Q\cup\{\infty\}$ called the slope \cites{Conway}. Let $Q_{\frac{p}{q}}$ denote the rational tangle with slope $\frac{p}{q}$. 

\begin{proposition}
\label{prop:rational}
The character variety $\RR(Q_{\frac{p}{q}})\cong I$. Furthermore, $\pcs(Q_{\frac{p}{q}}) = \frac{p}{q}$.
\end{proposition}

\begin{remark}
A version of this theorem can be found in \cite{PCI}*{Lemma 10.1}. Note that the conventions in \cite{PCI} are chosen differently so that $\RR(Q)_\frac{p}{q}$ maps to the pillowcase with slope $\frac{q-p}{q}$. 
\end{remark}

\begin{remark}
$\operatorname{PSL}(2,\Z)$ acts on the pillowcase by homeomorphisms \cite{CHK}*{Theorem B}. For $M:=\bmat{a&b\\c&d}\in \operatorname{SL}(2,\Z)$, %
$M(\gamma,\theta) = (a\gamma+b\theta, c\gamma+d\theta)$. Then to check that this action is well defined, note that if $(\gamma,\theta)$ is a point on the pillowcase,$(\gamma,\theta) \sim (\gamma+2m\pi,\theta+2n\pi) \sim (-\gamma,-\theta)$. It is easy to check that $M(\gamma,\theta)\sim M(\gamma+2m\pi,\theta+2n\pi) \sim M(-\gamma,-\theta)$.
\end{remark}

\begin{figure}
    \centering
    \begin{subfigure}[b]{0.3\textwidth}
        \centering
        \includegraphics[width=\textwidth]{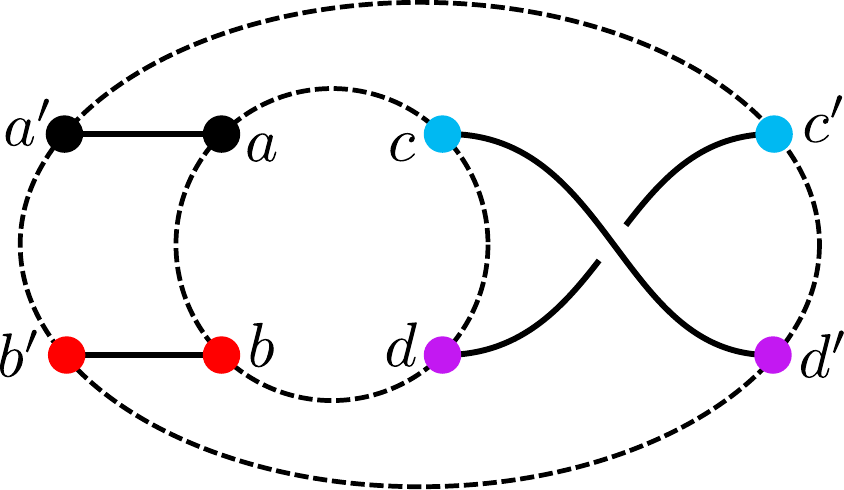}
        \caption{}
        \label{fig:Twist}
    \end{subfigure}
    \hspace{2cm}
    \begin{subfigure}[b]{0.3\textwidth}
        \centering
        \includegraphics[width=\textwidth]{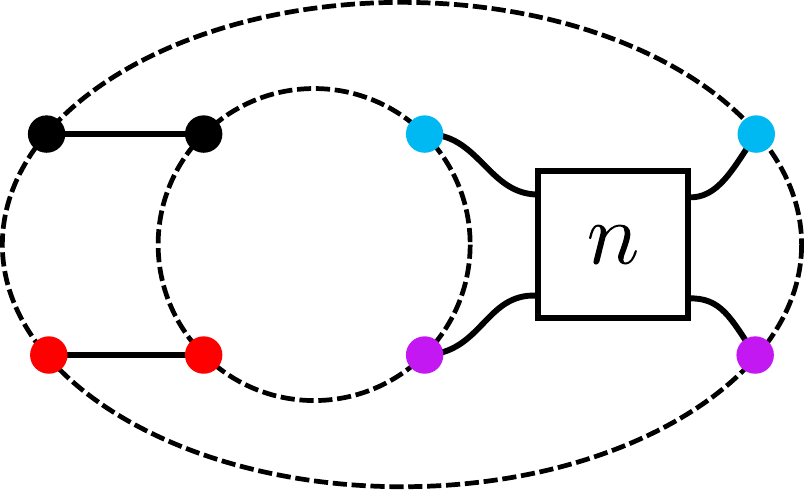}
        \caption{}
        \label{fig:nTwist}
    \end{subfigure}
    \caption{The tangles $T_1$ and $T_n$ as described in Lemma \ref{lem:nTwist}.}
\end{figure}

\begin{lemma}
\label{lem:nTwist}
Let $T_n$ be the result of gluing a tangle $T$ with one boundary component into the inner boundary of the tangle in \ref{fig:nTwist}. Then $\RR(T_n)\cong \RR(T)$ and $p(\RR(T_n)) = \bmat{1&0\\n&1}p(\RR(T))$. As a result, $\pcs(T_n) = \pcs(T)+n$.
\end{lemma}

\begin{proof}
Note that $T_1$ is the result of plugging $T$ into the inner boundary of the tangle in Figure \ref{fig:Twist}. Using the labels in the figure, $a'=a$ and $b'=b$. Furthermore, $c' = d$. In standard gauge with respect to the inner boundary, $a = \i$, $b = e^{\gamma \k}\i$ and $c = e^{\theta \k}\i$ forcing $d = e^{\theta+\gamma \k}\i$, then we get $a' = \i$, $b' = e^{\gamma \k}\i$ and $c' = e^{\gamma+\theta \k}\i$ and so $\gamma' = \gamma$ while $\theta' = \theta+\gamma$. So $(\gamma',\theta') = \bmat{1&0\\1&1}(\gamma,\theta)$. In addition, $\pcs(T_1) = \frac{\theta'}{\gamma'} = \frac{\theta+\gamma}{\gamma} = \frac{\theta}{\gamma}+1 = \pcs(T)+1$. 

It is clear that composing with the tangle in Figure \ref{fig:nTwist}, can be recast as $n$ iterations of the previous tangle for $n>0$, showing that it acts on $p(\RR(T))$ by $\bmat{1&0\\1&1}^n=\bmat{1&0\\n&1}$ and $\pcs(T_n) = \pcs(T)+n$. A similar calculation shows that it works for $n\le 0$ as well.
\end{proof}

\begin{figure}
    \centering
    \includegraphics[height=1.5in]{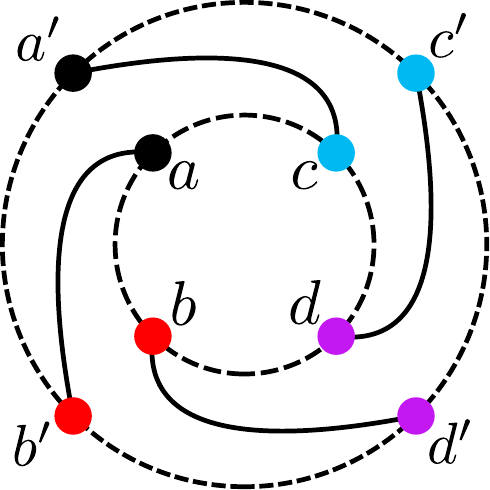}
    \caption{}
    \label{fig:rotate}
\end{figure}

Following the notation of \cite{Rational}, let $-\frac{1}{T}$ be the result of gluing $T$ into the inner boundary of Figure \ref{fig:rotate}.

\begin{lemma}
\label{lem:rotate}
The character variety $\RR(-\frac{1}{T}) \cong \RR(T)$ and $p(\RR(-\frac{1}{T})) = \bmat{0&1\\1&0}p(\RR(T))$. As a result, $\pcs(-\frac{1}{T}) = -\frac{1}{\pcs(T)}$.
\end{lemma}

\begin{proof}
$a' = c$, $b' = a$, $d' = b$, and $c' = d$. Conjugating gives $a' = \i$, $b' = e^{-\theta\k}\i$, and $c' = e^{\gamma\k}\i$. So $\theta' = \gamma$ and $\gamma' = -\theta$. Thus, $s' = -\frac{1}{s}$.
\end{proof}

For a tangle $(B^3, T)$, let $-T$ be the result of an orientation-reversing homeomorphism on $B^3$. Given a tangle diagram of $T$, a tangle diagram of $-T$ can be obtained by reversing each crossing.

\begin{lemma}
For any rational tangle $T$, $\pcs(-T) = -\pcs(T)$.%
\end{lemma}

\begin{proof}
Any rational tangle can be constructed by performing a series of moves described in Lemma \ref{lem:nTwist} and \ref{lem:rotate}. To get the mirror image, every time $T$ was turned into $T_n$ as described in Lemma \ref{lem:nTwist}, it is replaced by $T_{-n}$. Following the effect of this on $\pcs$ shows that $\pcs(-T)=-\pcs(T)$.
\end{proof}

\begin{proof}[Proof of Proposition \ref{prop:rational}]
Consider the tangle $Q_0$ in Figure \ref{fig:slope0} and a representation $\rho\in\RR(Q_0)$. By noting the directions of the meridians given in Figure \ref{fig:4PS}, it is easy to see that $a = c$ and $b = d$. Thus, $\theta(\rho) = 0$ and $\gamma(\rho) \in [0,\pi]$. So the image on the pillowcase has slope 0.

Composing with the tangles in Lemmas \ref{lem:nTwist} and \ref{lem:rotate} is sufficient for generating all rational tangles from $Q_0$. Each of these compositions affects the pillowcase slope the same way as Conway's slope (compare \cite{Rational}*{Proposition 5}). Thus, $\pcs(Q_\frac{p}{q}) = \frac{p}{q}$.
\end{proof}

\begin{figure}
    \centering
    \begin{subfigure}[b]{0.3\textwidth}
        \centering
        \includegraphics[width=\textwidth]{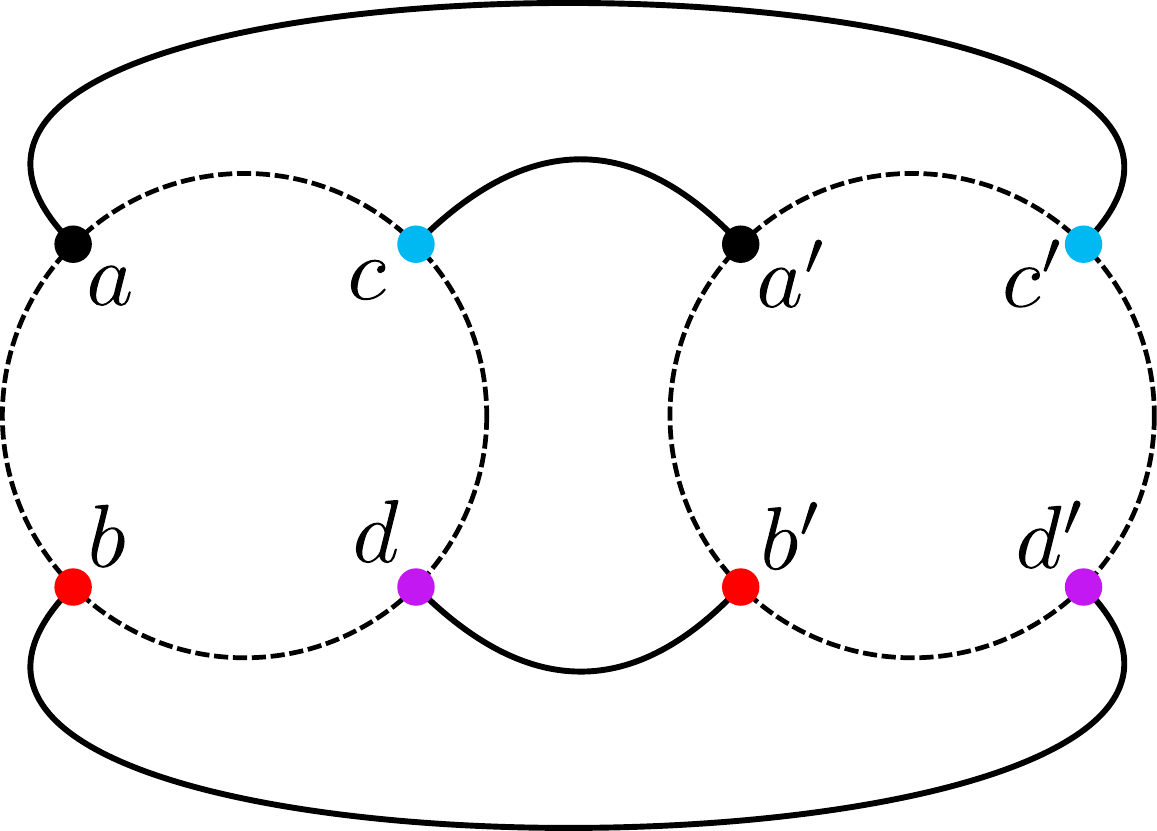}
        \caption{}
        \label{fig:flip}
    \end{subfigure}
    \hspace{1cm}
    \begin{subfigure}[b]{0.25\textwidth}
        \centering
        \includegraphics[width=\textwidth]{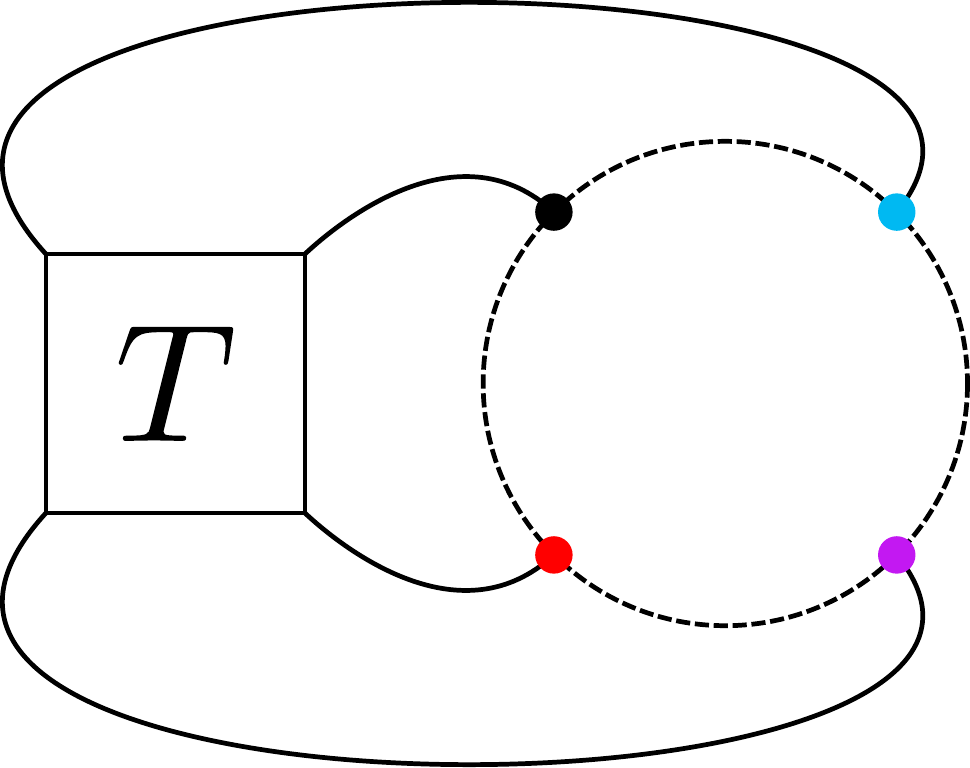}
        \caption{}
        \label{fig:outside}
    \end{subfigure}
    \hspace{1cm}
    \begin{subfigure}[b]{0.15\textwidth}
        \centering
        \includegraphics[width=\textwidth]{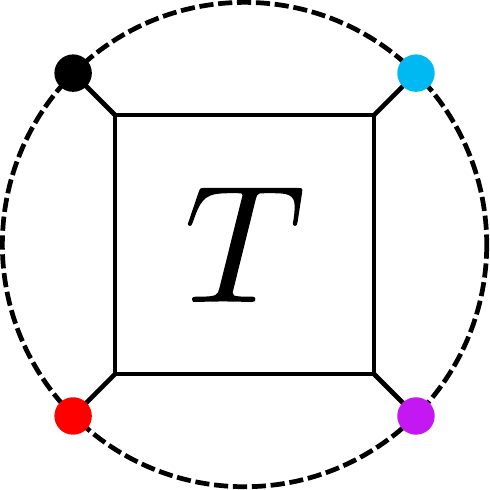}
        \caption{}
        \label{fig:inside}
    \end{subfigure}

    \caption{(A) shows a certain tangle cobordism. The tangle (B) can be obtained by gluing tangle (C) to tangle (A).}
\end{figure}

For a tangle $(X,T)$ such that $\partial(X,T) = (S^2, 4)$, let $\widehat{(X,T)}$ denote the tangle obtained by gluing $(X,T)$ into the tangle cobordism in Figure \ref{fig:flip}. There is a homeomorphism $\phi:(X,T)\to\widehat{(X,T)}$, however $\phi$ is not a tangle equivalence in the sense of Definition \ref{def:equiv} because $\psi^T$ and $\psi^{\widehat{T}}\circ \phi$ differ by an orientation reversing involution on $\mathbf{S}$. As a consequence, $\widehat{\widehat{T}} = T$. %
Diagrammatically, going from $T$ to $\widehat{T}$ is equivalent to switching between the tangle diagrams in Figures \ref{fig:outside} and \ref{fig:inside}.

\begin{lemma}
\label{lem:flip}
The homeomorphism $\phi:T\to \widehat{T}$ induces a homeomorphism $\phi^*: \RR(T)\to\RR(\widehat{T})$ such that $\bmat{1&0\\0&-1}\circ p = p\circ \phi^*$.
\end{lemma}

\begin{proof}
$a' = c$, $b' = d$, $c' = a$, and $d' = b$. Conjugating to be in in standard gauge with respect to the ordering $(a', b', c', d')$ gives $a' = \i$, $b' = e^{\gamma\k}\i$, and $c' = e^{-\theta\k}\i$. Thus $\gamma'=\gamma$ and $\theta' = -\theta$. Take $\rho\in\RR(T)$ with $p(\rho) = (\gamma,\theta)$, then $p(\phi^*(\rho)) = (\gamma',\theta') = (\gamma,-\theta)$.
\end{proof}

\subsection{Gluing Characters}

Finding the character variety of a complicated space can be achieved using cut and paste methods to decompose the space into simpler pieces and gluing the character varieties of these simpler spaces together. However, there may be many ways to glue two characters together depending on their stabilizers. Applying \cite{CHK}*{Lemma 6.1} to tangles makes this statement more precise:

\begin{lemma}
\label{lem:bending}
    Let $(X,T) = (X_1, T_1) \cup_{(\Sigma, S)} (X_2, T_2)$. Let $\pi_i$ be perturbation data for $(X_i, T_i)$. Then there is a surjection 
    \[f: \RR_{\pi_1\cup \pi_2}(X, T) \to \RR_{\pi_1}(X_1, T_1) \times_{\RR(\Sigma, S)} \RR_{\pi_2}(X_2, T_2).\]
    The fiber over $([\rho_1], [\rho_2])$ is homeomorphic to the double coset space 
    \[\Stab(\rho_1)\backslash \Stab(\rho_1|_{(\Sigma, S)})/\Stab(\rho_2).\]
\end{lemma}

These fibers, $\Stab(\rho_1)\backslash \Stab(\rho_1|_{(\Sigma, S)})/\Stab(\rho_2)$, are called \emph{gluing parameters} (or sometimes \emph{bending parameters}). In the case of $\SU$, representations can have stabilizers isomorphic to one of three subgroups: $\SU$, $\U$, or $\Z/2$. Note that the stabilizer type is conjugation invariant, so any point in an $\SU$ character variety has a well defined stabilizer type. For $\rho\in\RR_\pi(X,T)$ with $T$ non-empty, the stabilizer of $\rho$ cannot be $\SU$ since there is a traceless element of $\SU$ in its image. When $T$ is non-empty, an $\SU$ representation has stabilizer $\U$ if it is abelian, which means that the image is contained in a subgroup of the form $\{e^{\theta Q}\in\SU \mid \theta\in S^1\}$ for some fixed $Q\in C(\i)$. In particular, for $\rho\in\RR_\pi(X,T)$ with stabilizer $U(1)$ and $T$ non-empty, each meridian must be sent to $\{\pm Q\}$. If the representation does not satisfy this condition, the stabilizer is $\Z/2$. Following the notation of \cite{HKRegularity}, let $\RR_\pi(X, T)^{\Z/2}$ and $\RR_\pi(X, T)^\U$ be the subspaces of $\RR_\pi(X, T)$ with the prescribed stabilizers. So we have a decomposition: $\RR_\pi(X, T) = \RR_\pi(X, T)^{\Z/2} \sqcup \RR_\pi(X, T)^\U$.

As an example, recall that $\RR(S^2, 4)$ is the pillowcase $P$. It is not difficult to see that $\RR(S^2, 4)^\U = P^c$, the corners of the pillowcase and $\RR(S^2, 4)^{\Z/2} = P^*$, the smooth stratum of the pillowcase.

When restricting a representation to a submanifold, the stabilizer of the restriction is at least as big as the stabilizer of the original representation. Of particular interest is when the representation is restricted to the boundary. Then for a tangle $(X,T)$ such that $T\ne \emptyset$ and $\partial T\ne \emptyset$, we can decompose its character variety as follows: \[\RR_\pi(X,T) = \RR_\pi(X,T)^{\Z/2, \Z/2} \sqcup \RR_\pi(X,T)^{\Z/2, \U} \sqcup \RR_\pi(X,T)^{\U, \U}\] where $\RR_\pi(X,T)^{\Z/2, \U}$ is the subspace of $\RR_\pi(X,T)$ which has stabilizer $\Z/2$, but the restriction to the boundary $(\partial X, \partial T)$ has $\U$ stabilizer. $\RR_\pi(X,T)^{\Z/2, \Z/2}$ and $\RR_\pi(X,T)^{\U, \U}$ are defined analogously. 

\begin{definition}
For $\rho\in\RR_\pi(X,T)$ and $Y\subset X$ a submanifold, define \begin{equation}
\label{eq:stab}
\wStab_{Y}(\rho) := (\Stab(\rho), \Stab(\rho|_{(Y, Y\cap T)})) \text{ and } \wStab(\rho) := \wStab_{\partial X}(\rho).
\end{equation}
\end{definition}

By Lemma \ref{lem:bending}, the gluing parameters can be non-trivial if $\rho_1$ or $\rho_2$ are in $\RR_\pi(X,T)^{\Z/2, \U}$. For this reason, we are particularly interested in character varieties $\RR_\pi(X,T)$ for which $\RR_\pi(X,T)^{\Z/2, \U}$ is empty. Note that if $(\partial X, \partial T) = (S^2, 4)$, then $\RR_\pi(X,T)^{\Z/2,\U}$ is empty if and only if the non-abelian representations map to $P^*$.

\begin{lemma}
\label{lem:generic}
If either $\RR_{\pi_1}(X_1, T_1)^{\Z/2, \U}$ or $\RR_{\pi_2}(X_2, T_2)^{\Z/2, \U}$ is empty, then the map \[f: \RR_{\pi_1\cup \pi_2}(X, T) \to \RR_{\pi_1}(X_1, T_1) \times_{\RR(\Sigma, S)} \RR_{\pi_2}(X_2, T_2)\] as described in Lemma \ref{lem:bending} is a continuous bijection of compact Hausdorff spaces.
\end{lemma}

\begin{proof}
 Without loss of generality, let $\RR_{\pi_1}(X_1, T_1)^{\Z/2, \U}$ be empty. As a result, $\Stab(\rho_1) = \Stab(\rho_1|_{(\Sigma, S)})$ and so $\Stab(\rho_1)\backslash \Stab(\rho_1|_{(\Sigma, S)}) \cong \{I\}$ where $I\in\SU$ is the identity. Then by Lemma \ref{lem:bending}, \[f\inv(([\rho_1], [\rho_2])) \cong \Stab(\rho_1)\backslash \Stab(\rho_1|_{(\Sigma, S)})/\Stab(\rho_2) \cong \{I\}/\Stab(\rho_2) \cong \{I\}\]
 Thus the fiber of $f$ over any point is always a point, finishing the lemma.
\end{proof}

As a result of this, many results will be simplified if the character varieties involved have no representations with stabilizer $(\Z/2, \U)$. Fortunately, the following result of Herald and Kirk ensures that this condition can be achieved with holonomy perturbations.

\begin{theorem}[\cite{HKRegularity}*{Corollary D}]
\label{thm:reg}
Let $(X,T)$ be a 2-tangle in a homology ball. There exists an arbitrarily small holonomy perturbation $\pi$ so that the restriction map \[\RR_\pi(X,T)^{\Z/2,\Z/2}\to P\] is a Lagrangian immersion into $P^*$ and $\RR_\pi(X,T)^{\Z/2, \U}$ is empty.
\end{theorem}

\subsection{Lagrangian Correspondences}

Suppose that $M_0$, $M_1$, and $M_2$ are manifolds. Let $N_{01}\hookrightarrow M_0\times M_1$ and $N_{12}\hookrightarrow M_1\times M_2$ be subspaces.

\begin{definition}
The \emph{geometric composition} of $N_{01}$ and $N_{12}$ is $N_{01}\circ N_{12}:= \pi_{02}(N_{01}\times_{M_1}N_{12})$ where $\pi_{02}: M_0\times M_1 \times M_1 \times M_2 \to M_0\times M_2$ is the projection map.

\end{definition}

Geometric compositions play nicely with the structure of symplectic manifolds and their Lagrangian submanifolds. If $M = (M,\omega)$ is a symplectic manifold, let $M^-:=(M,-\omega)$. If $M$ and $N$ are symplectic manifolds, then a \emph{Lagrangian correspondence} from $M$ to $N$ is a Lagrangian submanifold $L\subset M^-\times N$.

\begin{theorem}[\cite{WW}*{Lemma 2.0.5}]
\label{thm:lagcomp}
Let $M_0$, $M_1$, and $M_2$ be symplectic manifolds and let $L_{01}\subset M^-_0\times M_1$ and $L_{12}\subset M^-_1\times M_2$ be Lagrangian correspondences.
If $\pi_{1}(L_{01})\pitchfork\pi_{1}(L_{12})$, then $L_{01}\times_{M_1} L_{12}$ is a manifold and the geometric composition $L_{01}\circ L_{12}$ is an immersed Lagrangian submanifold of $M^-_0\times M_2$.
\end{theorem}

\subsection{Overview of Reduced Singular Instanton Homology}
\label{sec:instanton}

\begin{figure}
    \centering
    \includegraphics[height=2in]{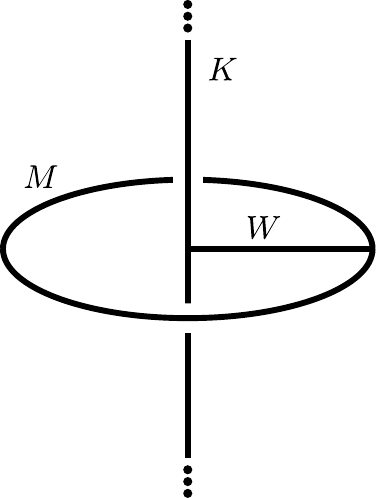}
    \caption{}    
    \label{fig:Earring}
\end{figure}

Let $X$ be an integer homology sphere and let $K:S^1\to X$ be a knot. For technical reasons, Kronheimer and Mrowka modify the knot in a process known as adding a marking. The particular atom considered here is called the \emph{earring}. Let $x\in K$ and let $M$ be a meridian of $K$ at $x$ and let $W$ be an (unknotted) arc from $M$ to $x$ as in Figure \ref{fig:Earring}. Let $K^\natural := K \cup M \cup W$ be $K$ \emph{with an earring}. Then consider the principal $\SU$ bundle over $X$, $E$.

Let $A$ be the space of connections $c$ on $E$ such that:

\begin{itemize}
    \item $c$ is singular on $K^\natural$ in the sense of \cite{KMinst}.
    \item If $\gamma_i$ is a sequence of loops in $X\setminus K^\natural$ such that $\gamma_i$ is a meridian of $K$ or $M$, limiting to a single point on $K$ or $M$ respectively, then the holonomy of $\gamma_i$ should limit to an element of $C(\i)$.
    \item If $\gamma_i$ is a sequence of loops in $X\setminus K^\natural$ such that $\gamma_i$ is a meridian of $W$, limiting to a single point on $W$, then the holonomy of $\gamma_i$ should limit to -1.
\end{itemize}

\begin{definition}
The \emph{Chern-Simons function} $\CS:A\to \R$ is given by
\[\CS(a) = \frac{1}{8\pi^2}\int_X \tr(a\wedge da + \frac{2}{3}a\wedge a\wedge a) \text{ for }a\in A.\]
\end{definition}

Then let $G$ be the gauge group and $B:= A/G$. The Chern-Simons function descends to a well defined function $\CS:B\to \R/\A$. The chain group $CI^\natural$ is freely generated by the \textit{flat} connections mod gauge in $B$, which are the critical points of $\CS$. Generally, $\CS$ is not Morse, so it is often necessary to perturb $\CS$ with a holonomy perturbation, $\pi$, which takes as input the same data as mentioned in Section \ref{sec:holpert}. In this case, $CI^\natural$ is generated by \emph{holonomy perturbed flat connections} in $B$, which are flat outside of a neighborhood of the perturbation curves in $\pi$ and satisfy a slightly different equation inside those neighborhoods. See \cites{Taubes, KMinst} for details.

\begin{theorem}[\cites{Taubes ,Herald}]
\label{thm:translate}
For (possibly trivial) holonomy perturbation data $\pi$, there is a one-to-one correspondence between the moduli space of holonomy perturbed flat connections in $B$ and the perturbed character variety $\RR_\pi(X, K^\natural)$ (which is defined with the analogous rule that meridians of $W$ get sent to $-1\in\SU$, see Section \ref{sec:Earring}).
\end{theorem}

The differentials of the chain complex come from instantons, which are connections on a principal $\SU$ bundle over $(X,K^\natural)\times \R$ satisfying the anti-self dual equation:
\[F_A = -\star F_A\] where $F_A$ is the curvature 2-form of $A$ and $\star$ is the Hodge star operator. Then the \emph{reduced singular instanton homology}, $I^\natural$, is obtained by taking the homology of this chain complex.

Given a knot $K$, it is extremely difficult to compute $I^\natural(K)$ directly. However, there are some easy to calculate bounds on the rank. There is a spectral sequence from the reduced Khovanov homology of $\ol{K}$ to $I^\natural(K)$ \cite{KM} and thus $\rank(I^\natural(K)) \le \rank(\Kh(\ol{K}))$. Let $u(K):= \rank(\Kh(\ol{K}))$ be this upper bound. Kronheimer and Mrowka proved that $I^\natural$ detects the unknot, and thus Khovanov homology detects the unknot. If the Alexander polynomial of $K$ is given by $\sum_{-d}^d a_i t^i$ then $\rank(I^\natural(K))$ is bounded below by $l(K):= \sum_{-d}^d \abs{a_i}$ \cites{KMAlex, Lim}.

The Atiyah-Floer conjecture posits that every instanton Floer theory has a corresponding, isomorphic Lagrange Floer theory defined using character varieties \cite{AtiyahFloer}. In some settings, Atiyah-Floer counterparts have been constructed for instanton theories \cites{DostSal, DaemiFukLip}, but there is no counterpart known for $I^\natural$ yet. A program initiated by Hedden, Herald, and Kirk seeks to find an Atiyah-Floer counterpart to $I^\natural$ \cites{PCI, PCII}. This counterpart is explored further in Section \ref{sec:Hnat}.

\section{Unperturbed Character Varieties}
\label{sec:unpert}

\subsection{Unperturbed Traceless \texorpdfstring{$\SU$}{SU(2)} Character Variety of \texorpdfstring{$C_3$}{C3}}
\label{sec:C3Unpert}

\begin{figure}
    \centering
    \includegraphics[height=2.5in]{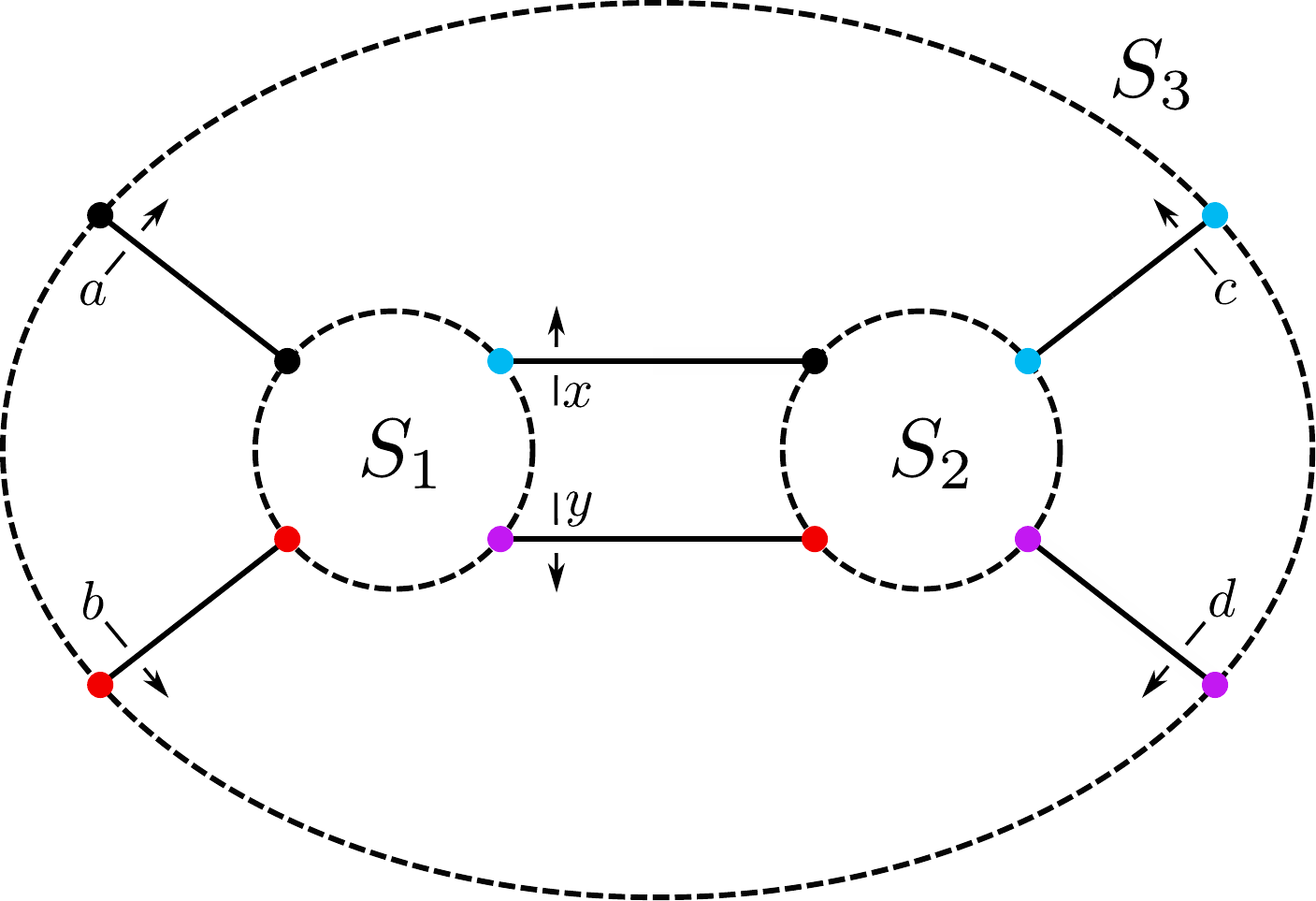}
    \caption{The tangle $C_3$}
    \label{fig:C3}
\end{figure}

Let $C_3 = (S^3\setminus 3B^3, T)$ be the tangle shown in Figure \ref{fig:C3}. Label the meridians of $C_3$ as in the figure. Note that $\partial C_3$ is a disjoint union of 3 copies of $(S^2, 4)$. Call these components $S_i$ as labeled in the figure. In the literature of arborescent knots, $C_3$ is also known as the \textit{(hollow) elementary tangle} \cites{This, BonSieb}.

The tangle sum, $T_1+T_2$, is constructed by gluing tangles $T_1$ and $T_2$, into $C_3$ along $S_1$ and $S_2$. Therefore, a representation $\rho\in\RR(T_1+T_2)$ can be decomposed into a representation $\rho|_{T_1}\in\RR(T_1)$, a representation $\rho|_{T_2}\in\RR(T_2)$, and a representation $\rho|_{C_3}\in\RR(C_3)$. 
Conversely, given representations $\rho_1\in\RR(T_1)$, $\rho_2\in\RR(T_2)$ and $\rho_C\in\RR(C_3)$ such that each $\rho_i$ and $\rho_C$ agree when restricted to $S_i$, these representations can be glued together to obtain representations of $T_1+T_2$.
In this section, the character variety $\RR(C_3)$ is studied, and its structure is used to compute the character varieties of tangle sums, $\RR(T_1+T_2)$.

\begin{proposition}
A presentation of the fundamental group of the tangle complement
\[X_0 := (S^3\setminus 3B^3)\setminus T\]
is given as follows:

\begin{description}
\item[Generators] The set $\{a,b,c,d,x,y\}$ generates the group.

\item[Relations] The following relations are complete.
\begin{enumerate}[{\upshape(I)}]
    \item $\ol{a}b = \ol{c}d$ (from $\SThree$)
    \label{eq:rel1p}
    
    \item $y = x\ol{a}b$ (from $\SOne$)
    
    \item $\ol{x}y=\ol{c}d$ (from $\STwo$)
\end{enumerate}
\end{description}
\end{proposition}

\begin{remark}
Relation (III) can be derived from Relations (I) and (II):
$\ol{a}b=\ol{c}d \implies \ol{x}x\ol{a}b=\ol{c}d \stackrel{\text{(II)}}{\implies} \ol{x}y = \ol{c}d$, which is Relation (III).
\end{remark}

\begin{corollary}
$\pi_1(X_0)  = \langle a, b, c, x\rangle$ and thus $\Ro(X_0)\cong \SU^4$.
\end{corollary}

\begin{remark}
$\Ro(X_0) = \Ro(X_0,\emptyset)$ and $\Ro(C_3)$ are both spaces of representations of $\pi_1(X_0)$ but $\Ro(C_3)$ requires that the meridians $a$, $b$, $c$, $d$, $x$, and $y$ are sent to traceless elements while $\Ro(X_0)$ does not.
\end{remark}

Let
\[\Pr := \{\i\}\times \Sk \times \Sk \times C(\i)\cong T^2\times S^2\]
and define the embedding $\Lambda: \Pr\to\Ro(X_0)$ so that $\Lambda(a_0, b_0, c_0, x_0)$ is the representation $\rho$ determined by $\rho(a) = a_0$, $\rho(b) = b_0$, $\rho(c) = c_0$, and $\rho(x) = x_0$.

It will be convenient to parameterize the $S^2$ factor of $\Pr$ with spherical coordinates. So take the map $\ol{\Gamma}:T^2\to C(\i)$ given by \[\ol{\Gamma}(\alpha, \beta) = \sin\alpha\cos\beta\i+\sin\alpha\sin\beta\j+\cos\alpha\k.\] With this in mind, define $\widetilde{\Pr} = T^4$ parametrized by $(\gamma,\theta,\alpha,\beta)$ along with the map 
\[\widetilde{\Gamma}: \widetilde{\Pr}\to \Pr\]
given by 
\[\widetilde{\Gamma}(\gamma, \theta, \alpha, \beta) = (\i, e^{\gamma\k}\i, e^{\theta\k}\i, \ol{\Gamma}(\alpha, \beta)).\]
Furthermore, define $\Gamma:\widetilde{\Pr}\to \Ro(X_0)$ to be the composition $\Lambda \circ \widetilde{\Gamma}$. The following diagram summarizes the spaces and maps defined so far:

\begin{center}
\begin{tikzcd}
\widetilde{\Pr}\ar[d,"\widetilde{\Gamma}"]\ar[dd, "\Gamma"',bend right = 80] \cong T^4  \\
\Pr \ar[d,"\Lambda"] \cong T^2\times S^2 \\
\Ro(X_0)\cong \SU^4
\end{tikzcd}
\end{center}

There is an involution on $\widetilde{\Pr}$ defined by the map
\[(\gamma, \theta, \alpha, \beta) \mapsto (-\gamma, -\theta, \alpha+\pi, \pi-\beta).\]
This induces an involution on $\Pr$ given by conjugation by $\i$, or more explicitly, 
\[(\i, e^{\gamma\k}\i, e^{\theta\k}\i, x_1\i+x_2\j+x_3\k)\mapsto (\i, e^{-\gamma\k}\i, e^{-\theta\k}\i, x_1\i-x_2\j-x_3\k).\]
Abusing notation, we refer to both of these involutions as $\iota$.
Note that $\Ro(C_3)\subset \Ro(X_0)$ is the subset of representations for which every element in the set $\{a,b,c,d,x,y\}\subset \pi_1(X_0)$ is sent to a traceless element of $\SU$. Let \[\Sols := \Lambda\inv(\Ro(C_3))\] be the subspace of $\Pr$ which $\Lambda$ maps to these traceless representations.

\begin{theorem}
\label{thm:C3}
$\Sols = \{\widetilde{\Gamma}(\gamma,\theta,\alpha,\beta)\in \Pr \mid \sin\gamma\cos\alpha=0\} \cong T^3\cup_{S^0\times T^2} (S^0\times S^1\times S^2)$
\end{theorem}

\begin{proof}
We start by identifying $\Gamma\inv(\Ro(C_3))$.
Consider $\rho = \Gamma(\gamma,\theta,\alpha,\beta)$. By construction, $\rho(a)$, $\rho(b)$, $\rho(c)$, and $\rho(x)$ are traceless. Furthermore, by Lemma \ref{lem:boundcond}, $\rho(d)$ is traceless. Thus, $\rho\in \Ro(C_3)$ if and only if $\rho(y)$ is traceless. Using Relation (II), the trace of $\rho(y)$ can be written as \[\tr(\rho(y)) = 2\Re(\rho(y)) = 2\Re(\rho(x\ol{a}b)) = -2\Re(\rho(x)e^{\gamma \k})=-2\sin\gamma\Re(\rho(x)\k) = 2\sin\gamma\cos\alpha.\]
Thus $\Gamma\inv(\Ro(C_3))$ is a reducible analytic variety, \[\{(\gamma,\theta,\alpha,\beta)\in\widetilde{\Pr}\mid \sin\gamma=0\}\cup \{(\gamma,\theta,\alpha,\beta)\in\widetilde{\Pr}\mid \cos\alpha=0\}.\]
Since $\widetilde{\Gamma}$ is surjective, \[\Sols = \widetilde{\Gamma}(\Gamma\inv(\Ro(C_3))) = \{\widetilde{\Gamma}(\gamma,\theta,\alpha,\beta)\in\Pr\mid \sin\gamma=0\}\cup \{\widetilde{\Gamma}(\gamma,\theta,\alpha,\beta)\in \Pr\mid \cos\alpha=0\}.\]

To understand the topology of this space, note that
\begin{align*}
\{\ol{\Gamma}(\alpha,\beta)\in C(\i) \mid \cos\alpha = 0\} &= \{\ol{\Gamma}(\frac{\pi}{2},\beta)\in C(\i) \} \cup \{\ol{\Gamma}(-\frac{\pi}{2},\beta)\in C(\i) \} = \Sk.
\end{align*}
Thus,
\[\{\widetilde{\Gamma}(\gamma,\theta,\alpha,\beta)\in \Pr\mid \cos\alpha=0\}= \{(\i,b, c, x)\in\Pr \mid b, c, x\in\Sk\}  \cong T^3\]
\[\{\widetilde{\Gamma}(\gamma,\theta,\alpha,\beta)\in\Pr\mid \sin\gamma=0\}= \{(\i,b, c, x)\in\Pr \mid b\in\{\pm \i\},~ c\in\Sk, ~ x\in C(\i)\}\cong S^0\times S^1\times S^2\]
\[\{(\gamma,\theta,\alpha,\beta)\in\widetilde{\Pr}\mid \sin\gamma=\cos\alpha=0\} = \{(\i,b, c, x)\in\Pr\mid b\in\{\pm\i\},~ c,x\in\Sk\}  \cong S^0\times T^2\]

\end{proof}

\begin{figure}
    \centering
    \includegraphics[height=2in]{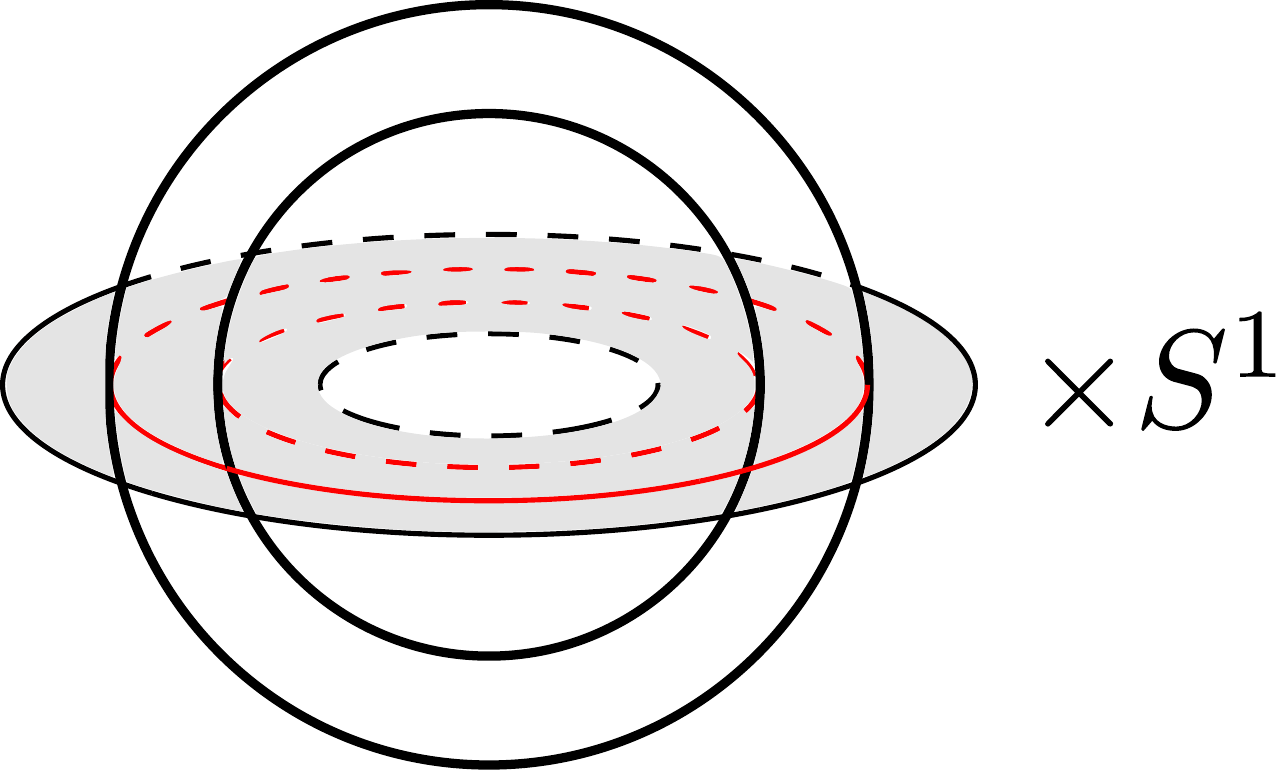}
    \caption{The space $\Sols$. The spherical coordinate is given by $\ol{\Gamma}(\alpha, \beta)$ and the radial coordinate is given by $\gamma$. The extra $S^1$ coordinate comes from $\theta$. The boundary circles of the shaded annulus are identified, giving a torus.}
    \label{fig:C3CV}
\end{figure}

Figure \ref{fig:C3CV} depicts $\Sols$. Define \[h:\Sols\to \RR(C_3),\quad h(x) = [\Lambda(x)].\]

\begin{proposition}
\label{prop:hfibers}
The map $h$ is surjective and the fiber over $[\rho]\in\RR(C_3)$ is:
\[\begin{cases}
    S^1 &\text{if }\rho(b), \rho(c) \in \{\pm\rho(a)\},~ \rho(x)\not\in\{\pm\rho(a)\}\\
    \pt &\text{if }\rho(b),\rho(c),\rho(x)\in\{\pm\rho(a)\}\\
    S^0 &\text{otherwise}
\end{cases}\]
\end{proposition}

\begin{proof}
To see that this map is surjective, first note that, by construction, the image of $\Lambda:\Pr\to \Ro(X_0)$ contains the subset of representations of $X_0$ in standard gauge with respect to $\bmu = (a, b, c, x)$, $\Ro_{\bmu}(X_0)$. As a result, $\Ro_\bmu(C_3)\subset \Lambda(\Sols)$.  Then quotienting out by the conjugation action of $\SU$ gives $\RR(C_3)\subset h(\Sols)$.

In order to identify the fibers, note that for $w\in\SU$ and $(\i,b,c,x)\in\Pr$, $w(\i,b,c,x)\ol{w} = (w\i\ol{w},wb\ol{w},wc\ol{w},wx\ol{w}) \in\Pr$ only if $w = e^{\psi\i}$ for $\psi\in S^1$, because otherwise $w\i\ol{w}\ne \i$. 

Let $[\rho]\in\RR(C_3)$ such that $\rho(b), \rho(c) \in \{\pm \rho(a)\}$ and $\rho(x)\not\in\{\pm\rho(a)\}$. Since $h$ is surjective, there is $(\i,b,c,x)\in\Pr$ such that $h(\i,b,c,x) = \rho$. Thus, $b,c\in\{\pm\i\}$ and $\rho(x)\not\in\{\pm\i\}$. So for $\psi,\psi'\in S^1$, $e^{\psi\i}(\i, b,c,x)e^{-\psi\i} = e^{\psi'\i}(\i, b,c,x)e^{-\psi'\i}$ if and only if $\psi' = \psi+\pi$k. So, the fiber of $[\rho]$ is homeomorphic to $\U/(\Z/2)\cong S^1$.

Next, let $[\rho]\in\RR(C_3)$ such that $\rho(b), \rho(c), \rho(x) \in \{\pm \rho(a)\}$. Since $h$ is surjective, there is $(\i,b,c,x)\in\Pr$ such that $h(\i,b,c,x) = \rho$ with $b,c,x\in\{\pm\i\}$. Thus, for any $\psi\in S^1$, $e^{\psi\i}(\i, b,c,x)e^{-\psi\i} = (\i,b,c,x)$. So, the fiber of $[\rho]$ is homeomorphic to a point.

Finally, let $[\rho]\in\RR(C_3)$ such that either $\rho(b)$ or $\rho(c)$ is not in $\{\pm \rho(a)\}$. Since $h$ is surjective, there is $(\i,b,c,x)\in\Pr$ such that $h(\i,b,c,x) = \rho$ such that $b$ or $c$ are in $\Sk\setminus\{\pm\i\}$. Thus, in order for $e^{\psi\i}be^{-\psi\i}\in\Sk$ and $e^{\psi\i}be^{-\psi\i}\in\Sk$, $e^{\psi\i}\in\{\pm 1,\pm\i\}$. Thus, $(\i,b,c,x)$ and $-\i(\i,b,c,x)\i$ are the unique points in the fiber of $[\rho]$ and so the fiber is homeomorphic to $S^0$.

\end{proof}

Let $\Sols^\dag:= \{(\i,b,c,x)\in\Sols\mid b\not\in\{\pm\i\} \text{ or } c\not\in\{\pm\i\}\}$ and let $\RR(C_3)^\dag:= h(\Sols^\dag)$. Note that $\RR(C_3)^\dag$ is precisely the subspace of $\RR(C_3)$ such that the fiber of $h$ over any point in $\RR(C_3)^\dag$ is $S^0$.

\begin{proposition}
\label{prop:restdcov}
$\Sols^\dag$ is a double cover of $\RR(C_3)^\dag$.
\end{proposition}

\begin{proof}

Because $\Lambda:\Pr\to \Ro(X_0)$ is an embedding and $\Sols$ is compact, $\Sols\cong \Lambda(\Sols)$. Because a restriction of a homeomorphism is a homeomorphism, $\Sols^\dag\cong \Lambda(\Sols^\dag)$. The quotient map $\Lambda(\Sols^\dag)\to \RR(C_3)^\dag$ has fiber $S^0$ and thus there is an involution action on $\Lambda(\Sols^\dag)$ swapping the points in the fiber. By the proof of Proposition \ref{prop:hfibers}, this involution is given by $\iota$, which conjugates by $\i$ on $\Lambda(\Sols^\dag)$. Proposition \ref{prop:hfibers} also shows that for any $\rho\in\Lambda(\Sols^\dag)$, either $\rho(b)$ or $\rho(c)$ is not $\pm\i$ and so $\iota$ acts properly discontinuously on $\Lambda(\Sols^\dag)$. Thus taking the quotient by the action of $\iota$ is a covering map. It's easy to see that \[h(\Sols^\dag)  \cong \Lambda(\Sols^\dag)/\iota = \RR(C_3)^\dag.\]

\end{proof}

A recurring theme for the remainder of this section is that the binary dihedral representations behave differently from the non-binary dihedral representations and so some results will take the form of showing what happens in each case separately before combining the results to the general case. In order to facilitate this strategy, it is useful to define and understand various subspaces of $\Sols$:

\begin{itemize}
    \item $\H := \{\widetilde{\Gamma}(\gamma, \theta, \alpha, \beta)\in \Sols \mid \sin\gamma = 0,~ \cos\alpha \ne 0\}$. In Figure \ref{fig:C3CV}, $\H$ is the white spheres minus the equators shown in red (resulting in four open hemispheres) times the $S^1$ factor.

    \item $\H^\dag = \H\cap \Sols^\dag$. Note that $h(\H^\dag) = \RR\nbd(C_3)$, the space of representations in $\RR(C_3)$ which are not binary dihedral because these are precisely the representations for which the images of $a$, $b$, $c$, and $x$ do not form a coequatorial set.

    \item $\A := \{\widetilde{\Gamma}(\gamma, \theta, \alpha, \beta)\in \Sols \mid \sin\gamma \ne 0,~ \cos\alpha = 0\}$. In Figure \ref{fig:C3CV}, $\A$ is depicted as the gray torus minus the red circles (leaving two annuli) times the $S^1$ factor. Note that $h(\ol{\A}) = \RR\bd(C_3)$, the space of representations in $\RR(C_3)$ which are binary dihedral.
    
    \item $\SV := \ol{\H}\cap\ol{\A} \cong \{\widetilde{\Gamma}(\gamma, \theta, \alpha, \beta)\in \Pr \mid \sin\gamma = 0,~ \cos\alpha = 0\} = \{(\i,b,c,x)\in\Sols\mid b = \pm\i, x\in\Sk\}$. Note that $\Sols = \ol{\H}\cup_\SV \ol{\A}$. In Figure \ref{fig:C3CV}, $\SV$ is depicted as the two red circles times the $S^1$ factor.

    \item $\Sols^0 := \{(\i, b, c, x)\in \Sols \mid b,c,x\in\{\pm \i\}\}$. $h(\Sols^0) = \RR(C_3)^\U$. $\Sols^0$ is composed of eight points, all in $\SV$.
       
    \item $\Sols^* :=  (\H\cup \A)\cap\Sols^\dag = \Sols^\dag\setminus \SV$. 
    Note that \[\Sols^*\cong \left[S^2 \times (S^1\setminus S^0)\right] \sqcup \left[ T^2\times (S^1\setminus S^0)\right]\] is a manifold.
\end{itemize}

\begin{definition}
\label{def:rcstar}
Let $\RR(C_3)^* := h(\Sols^*)$. Because $\Sols^*\subset \Sols^\dag$, $\Sols^*$ double covers $\RR(C_3)^*$. Then $\RR(C_3)^*$ is a manifold because $\Sols^*$ is a manifold. Furthermore, $\RR(C_3)\setminus \RR(C_3)^* = h(\SV)$. 
\end{definition}

Recall that $\partial C_3$ is made up of three copies of $(S^2, 4)$ which are labeled $S_1$, $S_2$, and $S_3$, whose respective character varieties are the pillowcases $P_1$, $P_2$, and $P_3$.
For $[\rho] = \Gamma(\gamma, \theta, \alpha, \beta)\in\RR(C_3)$, Equations \ref{eq:gamma} and \ref{eq:theta} lead to the following formulae which determine the image of $[\rho]$ under each of the restriction maps $p_i: \RR(C_3)\to P_i$.

\begin{align}
\gamma_1(\rho) &= \ang{\rho(a)}{\rho(b)} = \ang{\rho(x)}{\rho(y)} = \gamma \label{eq:pccoords1} \\ 
\gamma_2(\rho) &= \ang{\rho(x)}{\rho(y)} = \ang{\rho(c)}{\rho(d)} = \gamma \label{eq:pccoords3} \\
\gamma_3(\rho) &= \ang{\rho(a)}{\rho(b)} = \ang{\rho(c)}{\rho(d)} = \gamma \label{eq:pccoords5} \\
\theta_1(\rho) &= \rang{\rho(a)}{\rho(x)}{\rho(b)} = 
\begin{cases}
\arccos(\sin\alpha\cos\beta) & \text{ if } \sin\gamma=0\\
\beta\sin\alpha & \text{ if }\sin\gamma > 0
\end{cases} \label{eq:pccoords2} \\
\theta_2(\rho) &= \rang{\rho(x)}{\rho(c)}{\rho(y)} =
\begin{cases}
\arccos(\sin\alpha\cos\beta\cos\theta + \sin\alpha\sin\beta\sin\theta) & \text{ if } \sin\gamma = 0\\
\theta - \beta\sin\alpha& \text{ if } \sin\gamma > 0
\end{cases} \label{eq:pccoords4} \\
\theta_3(\rho) &= \rang{\rho(a)}{\rho(c)}{\rho(b)} = \theta \label{eq:pccoords6} 
\end{align}

Consider the map $p_1\times p_2:\RR(C_3)\to P_1\times P_2$ given by \[\rho\mapsto \left(\gamma_1(\rho),\theta_1(\rho),\gamma_2(\rho),\theta_2(\rho)\right).\] It is important to understand the preimage of this map because for $\rho_1\in\RR_{\pi_1}(T_1)$ and $\rho_2\in\RR_{\pi_2}(T_2)$, $(p_1\times p_2)\inv(\gamma(\rho_1), \theta(\rho_1), \gamma(\rho_2), \theta(\rho_2))$ gives the space of representations in $\RR(C_3)$ which agree with $\rho_1$ and $\rho_2$ along $S_1$ and $S_2$. Using the decomposition $\RR(C_3) = h(\ol{\H})\cup h(\ol{\A})$, the preimage of $p_1\times p_2$ can be identified in two steps by identifying the preimages of the maps \[p_{h(\ol{\A})}:= (p_1\times p_2)|_{h(\ol{\A})} \text{ and } p_{h(\ol{\H})}:=(p_1\times p_2)|_{h(\ol{\H})}.\] 

\begin{lemma}
\label{lem:revengA}
Fix $z_1\in[0,\pi]$ and $z_2, z_3 \in S^1$.

\begin{enumerate}
    \item If $z_1\in(0,\pi)$ there is a unique $\rho\in p_{h(\ol{\A})}\inv(z_1, z_2, z_1, z_3)$, and $\theta_2(\rho) = z_3$. Regarding its image in the pillowcase, $p_3(\rho) = (z_1, z_2+z_3)$.

    \item If $z_1\in \{0, \pi\}$ and $z_2, z_3\in(0,\pi)$, then there are two representations in the preimage: $\rho_1, \rho_2\in p_{h(\ol{\A})}\inv(z_1, z_2, z_1, z_3)$. In the pillowcase, $p_3(\rho_1) = (z_1, z_2+z_3)$ and $p_3(\rho_2) = (z_1, z_2-z_3)$.

    \item If $z_1\in \{0, \pi\}$ and $z_2, z_3\in[0,\pi]$ and at least one of $z_2$ or $z_3$ is in $\{0,\pi\}$ then there is a unique $\rho\in p_{h(\ol{\A})}\inv(z_1, z_2, z_1, z_3)$. In addition, $p_3(\rho) = (z_1, z_2+z_3)$.
\end{enumerate}

\end{lemma}

\begin{proof}
For any of the cases, if $\rho = \Gamma(\gamma, \theta, \alpha, \beta)\in\RR(C_3)$ then clearly by Equation \ref{eq:pccoords1}, $\gamma_i(\rho) = z_1$. In $\A$ we know that $\cos\alpha = 0$, and so $\rho$ must be of the form $\Gamma(z_1, \theta,\frac{\pi}{2}, \beta)$ for some $\theta$ and $\beta$. Finally, once we know that $\rho = \Gamma(\gamma, \theta, \alpha, \beta)\in\RR(C_3)$, then $p_3(\rho) = (\gamma, \theta)$.

\begin{enumerate}
    \item If $z_1\in (0,\pi)$, Equations \ref{eq:pccoords2} and \ref{eq:pccoords4} simplify to give $\beta = z_2$ and $\theta = \beta+z_3 = z_2+z_3$. So $\rho = \Gamma(z_1, z_2+z_3, \frac{\pi}{2}, z_2)$, showing the first claim.

    \item When $z_1\in\{0,\pi\}$, Equation \ref{eq:pccoords2} becomes $z_2 = \arccos(\cos\beta)$. Since $z_2\in(0,\pi)$, $\beta = z_2$. Equation \ref{eq:pccoords4} becomes 
    \begin{align*} z_3 &= \arccos(\cos\beta\cos\theta+\sin\beta\sin\theta)\\
    &= \arccos(\cos(\beta-\theta))\\
    &= \arccos(\cos(z_2-\theta)).
    \end{align*}
    There are two possible values for $\theta$ which solve this, $\theta = z_2\pm z_3$. So, let $\rho_1 = \Gamma(z_1, z_2+z_3, \frac{\pi}{2}, z_2)$ and $\rho_2 = \Gamma(z_1, z_2-z_3, \frac{\pi}{2}, z_2)$.

    \item This follows the same logic as the last case, but if $z_3\in\{0,\pi\}$, then $z_2 + z_3 = z_2 - z_3$ in $S^1$ and so $\rho_1 = \rho_2$. 
    On the other hand, if $z_2\in\{0,\pi\}$, 
    \begin{align*}
    -\i\widetilde{\Gamma}(z_1,z_2+z_3,\frac{\pi}{2},z_2))\i &= \widetilde{\Gamma}(-z_1, -z_2-z_3,-\frac{\pi}{2}, \pi-z_2)\\
    &= \widetilde{\Gamma}(-z_1, -z_2-z_3,\frac{\pi}{2}, -z_2)\\
    &= \Gamma(z_1,z_2-z_3,\frac{\pi}{2},z_2)
    \end{align*}
    since $z_1=-z_1$ and $z_2=-z_2$. Taking the quotient by the conjugation action of $\SU$ gives $\rho_1 = \rho_2$.
\end{enumerate}

\end{proof}

\begin{figure}
    \centering
    \includegraphics[height=2in]{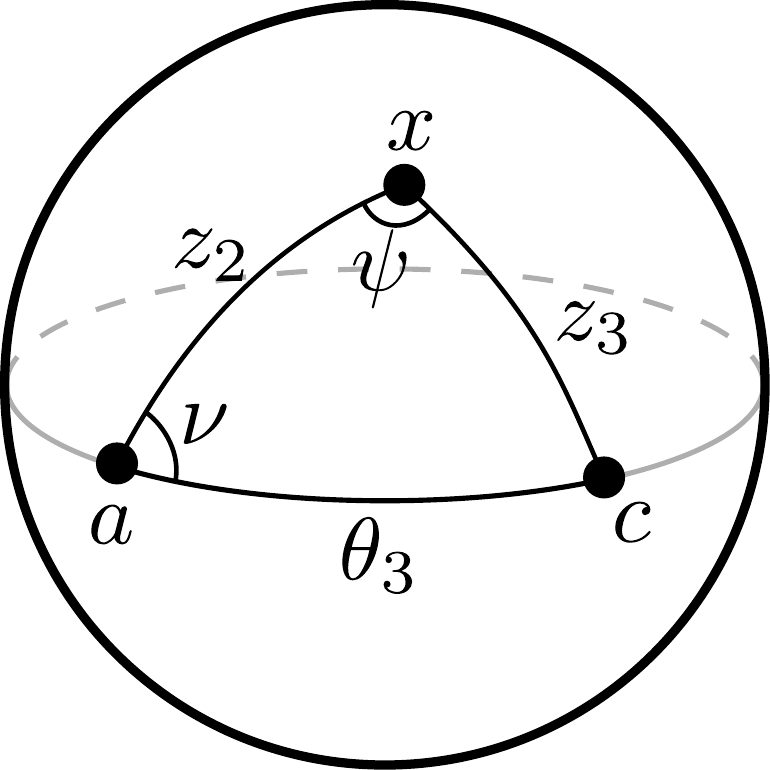}
    \caption{Spherical triangle formed by the images of $a$, $c$, and $x$.}    
    \label{fig:spheretriangle}
\end{figure}

\begin{lemma}
\label{lem:revengH}
Fix $z_1\in\{0,\pi\}$, $z_2,z_3\in[0,\pi]$.

\begin{enumerate}
    \item If $z_2,z_3\in(0,\pi)$, then $p_{h(\ol{\H})}\inv(z_1, z_2, z_1, z_3)\cong S^1$. This space can be parametrized by $\{\rho_\psi\}_{\psi\in S^1}$ such that \[p_3(\rho_\psi) = (z_1, \arccos(\cos z_2\cos z_3 + \sin z_2\sin z_3\cos\psi)).\]
    The image of $\{\rho_\psi\}_{\psi\in S^1}$ under $p_3$ is onto the line segment
    \[\{(\gamma,\theta)\in P\mid \gamma = z_1, \theta\in[\arccos(\cos(z_2-z_3)),\arccos(\cos(z_2+z_3))]\}.\]
    In particular, $\rho_0, \rho_\pi\in h(\ol{\A})$ and the endpoints of this line segment are given by \[p(\rho_0) = (z_1\arccos(\cos(z_2-z_3)))\]
    \[p(\rho_\pi) = (z_1,\arccos(\cos(z_2+z_3))).\]

    \item If $z_2$ or $z_3$ is in $\{0,\pi\}$, then there is a unique representation $\rho\in p_{h(\ol{\H})}\inv(z_1, z_2, z_1, z_3)$. In this case, $\rho\in h(\ol{\A})$ as well and $p_3(\rho) = (z_1, z_2+z_3)$.
\end{enumerate}
\end{lemma}

\begin{proof}

\begin{enumerate}
    \item By Equations \ref{eq:pccoords1}, \ref{eq:pccoords3}, and \ref{eq:pccoords5}, if $z_1 = 0$, then $\rho(b) = \rho(a)$, $\rho(d) = \rho(c)$, and $\rho(y) = \rho(x)$. On the other hand if $z_1 = \pi$, then $z_1$ means that $\rho(b) = -\rho(a)$, $\rho(d) = -\rho(c)$, and $\rho(y) = -\rho(x)$. In either case, the values of $\rho(a)$, $\rho(c)$, and $\rho(x)$ determine the representation. Additionally, if $\rho\in p_{\ol{\H}}\inv(z_1, z_2, z_1, z_3)$, then $\ang{\rho(a)}{\rho(x)} =\theta_1(\rho) = z_2$ and $\ang{\rho(x)}{\rho(c)} = \theta_2(\rho) = z_3$. Thus the images of $\rho(a)$, $\rho(c)$, and $\rho(x)$ make a triangle on the sphere where the length of the side from $\rho(a)$ to $\rho(x)$ is $z_2$ and the length of the side from $\rho(x)$ to $\rho(c)$ is $z_2$. %
    Specifying the angle $\psi$ between these two sides uniquely specifies a triangle in the sphere up to rotation (with marked vertices) and thus a representation on $C_3$ up to conjugation, see Figure \ref{fig:spheretriangle}. Define $\rho_\psi \in \RR(C_3)$ to be the corresponding representation. Note that $\rho(a)$, $\rho(c)$, and $\rho(x)$ are coequatorial if and only if $\psi\in\{0,\pi\}$. Thus $\rho_\psi$ is binary dihedral and thus in $h(\ol{\A})$ if $\psi\in\{0,\pi\}$.

    By Equation \ref{eq:pccoords6}, $\theta_3(\rho_\psi) = \ang{\rho(a)}{\rho(c)}$ and is thus the length of the side from $\rho_\psi(a)$ to $\rho_\psi(c)$. Let $\nu(\rho_\psi)$ be the angle between the side from $\rho_\psi(a)$ to $\rho_\psi(x)$ and the side form $\rho_\psi(a)$ to $\rho_\psi(c)$, as shown in Figure \ref{fig:spheretriangle}. Using spherical trigonometry, $\theta_3(\rho_\psi)$ and $\nu(\rho_\psi)$ can be expressed in terms of $\psi$, $z_2$, and $z_3$.

    The spherical law of cosines and the spherical law of sines respectively give: 
    \[\cos(\theta_3(\rho_\psi)) = \cos z_2\cos z_3 +\sin z_2\sin z_3\cos\psi\]
    and
    \[\sin(\nu(\rho_\psi)) = \frac{\sin z_3\sin\psi}{\sin(\theta_3(\rho_\psi))}.\]
    
    Thus if $\rho_\psi$ is in standard gauge, 
    \[a\mapsto \i,\quad b\mapsto \pm \i,\quad c\mapsto e^{\theta_3(\rho_\psi)\k}\i\]
    \[x\mapsto \cos(\theta_1) \i + \sin(\theta_1)\cos(\nu)\j + \sin(\theta_1)\sin(\nu)\k = e^{\theta_1 e^{\nu\i}\k}\i.\]

The most interesting cases are when $\psi\in\{0,\pi\}$ in which case the expression for $\theta_3(\rho_\psi)$ can be considerably simplified.
First consider $\psi = 0$. Since $z_2,z_3\in[0,\pi]$, $\sin z_2\sin z_3>0$ and so $\cos(\theta_3(\rho_\psi))$ attains its maximum value at $\psi=0$ and so $\theta_3(\rho_\psi)$ attains its minimum value. In particular, $\cos(\theta_3(\rho_0)) = \cos(z_2-z_3)$ and so $\theta_3(\rho_0) = \arccos(\cos(z_2-z_3))$. If $z_2\ge z_3$, then $\theta_3(\rho_\psi) = z_2-z_3$. If $z_2\le z_3$, then since cosine is even $\theta_3(\rho_\psi) = z_3-z_2$.

Now consider the case $\psi = \pi$. Again, $\sin z_2\sin z_3>0$ and so now $\cos(\theta_3(\rho_\psi))$ attains its minimum value at $\psi=\pi$ and so $\theta_3(\rho_\psi)$ attains its maximum value. Now $\cos(\theta_3(\rho_\psi)) = \cos(z_2+z_3)$ and so $\theta_3(\rho_\psi) = \arccos(\cos(z_2+z_3))$.
If $z_2+z_3\le\pi$, then $\theta_3(\rho_\psi) = z_2+z_3$. If $z_2 + z_3 > \pi$, then $\theta_3(\rho_\psi) = 2\pi-(z_2+z_3)$.

Combining these possibilities, there are four possible cases:

\noindent\textbf{Case 1}: $z_2\ge z_3$ and $z_2+z_3\le \pi$

$\rho_0$ sends $a\mapsto \i$, $b\mapsto \pm \i$, $x\mapsto e^{z_2\k}\i$, $c\mapsto e^{(z_2-z_3)\k}\i$, so $\theta_3(\rho_0) = z_2-z_3$

$\rho_\pi$ sends $a\mapsto \i$, $b\mapsto \pm \i$, $x\mapsto  e^{z_2\k}\i$, $c\mapsto e^{(z_2+z_3)\k}\i$, so $\theta_3(\rho_\pi) = z_2+z_3$

\noindent\textbf{Case 2}: $z_2\ge z_3$ and $z_2+z_3> \pi$

$\rho_0$ sends $a\mapsto \i$, $b\mapsto \pm \i$, $x\mapsto e^{-z_2\k}\i$, $c\mapsto e^{(z_2-z_3)\k}\i$, so $\theta_3(\rho_0) = z_2-z_3$

$\rho_\pi$ sends $a\mapsto \i$, $b\mapsto \pm \i$, $x\mapsto  e^{-z_2\k}\i$, $c\mapsto e^{(2\pi-z_2-z_3)\k}\i$, so $\theta_3(\rho_0) = 2\pi-z_2-z_3$

\noindent\textbf{Case 3}: $z_2\le z_3$ and $z_2+z_3\le \pi$

$\rho_0$ sends $a\mapsto \i$, $b\mapsto \pm \i$, $x\mapsto e^{-z_2\k}\i$, $c\mapsto e^{(z_3-z_2)\k}\i$, so $\theta_3(\rho_0) = z_3-z_2$

$\rho_\pi$ sends $a\mapsto \i$, $b\mapsto \pm \i$, $x\mapsto e^{z_2\k}\i$, $c\mapsto e^{(z_2+z_3)\k}\i$, so $\theta_3(\rho_0) = z_2+z_3$

\noindent\textbf{Case 4}: $z_2\le z_3$ and $z_2+z_3> \pi$

$\rho_0$ sends $a\mapsto \i$, $b\mapsto \pm \i$, $x\mapsto e^{-z_2\k}\i$, $c\mapsto e^{(z_3-z_2)\k}\i$, so $\theta_3(\rho_0) = z_3-z_2$

$\rho_\pi$ sends $a\mapsto \i$, $b\mapsto \pm \i$, $x\mapsto e^{z_2\k}\i$, $c\mapsto e^{(2\pi-z_2-z_3)\k}\i$, so $\theta_3(\rho_0) = 2\pi-z_2-z_3$ 
    
\item If $z_1, z_2\in\{0,\pi\}$ then any $\rho\in p_{h(\ol{\H})}\inv(z_1, z_2, z_1, z_3)$ must send $a$, $b$, $x$, and $y$ to $\pm\i$. Further, $\rho(d) = \pm \rho(c)$. Thus $\rho$ is binary dihedral and thus $\rho\in\ol{\A}$ and is thus covered by Case 3 of Lemma \ref{lem:revengA}. If instead  $z_1, z_3\in\{0,\pi\}$ then $\rho$ sends $a$ and $b$ to $\pm\i$ and $\rho$ maps $c$, $d$, $x$, and $y$ to $\pm e^{z_2\k}\i$. Thus, $\rho$ is again binary dihedral and covered by Case 3 of Lemma \ref{lem:revengA}.
\end{enumerate}

\end{proof}

\subsection{Tangle Sums}
\label{sec:keylemma}

\begin{definition}
Suppose $T_1$ and $T_2$ are tangles such that $\partial T_i = (S^2,4)$. Define the \emph{tangle sum}, $T_1+T_2$, to be $T_1\cup_{S_1} C_3 \cup_{S_2} T_2$ (see Figure \ref{fig:TangleSum}).
\end{definition}

This notation is used since this operation corresponds to Conway's notion of tangle addition \cite{Conway}. 
Note that $\partial(T_1+T_2) = S_3$. Our goal is to describe the character variety $\RR(T_1+T_2)$ and its image in the pillowcase $\PThree = \RR(S_3)$. If $\pi_i$ is perturbation data for $T_i$, then we are more generally interested in the perturbed character variety $\RR_{\pi_1\cup \pi_2}(T_1+T_2)$. In this section, we will see how this character variety relates to $\RR_{\pi_1}(T_1)$ and $\RR_{\pi_2}(T_2)$.

\begin{remark}
There are two useful decompositions of the tangle sum $T_1+T_2$.
First there is
\begin{equation}
\label{eq:sequential}
T_1+T_2 = (C_3\cup_{\SOne} T_1)\cup_{\STwo} T_2.
\end{equation}
Taking this perspective will allow us to compare $\RR_{\pi_1\cup \pi_2}(T_1+T_2)$ to $\RR_{\pi_1}(T_1)$ and $\RR_{\pi_2}(T_2)$.
The other decomposition is  
\begin{equation}
\label{eq:disjoint}
T_1+T_2 = C_3\cup_{\SOne\sqcup \STwo}(T_1\sqcup T_2).
\end{equation}
This subtle change in perspective allows us instead to compare $\RR_{\pi_1\cup \pi_2}(T_1+T_2)$ to $\RR_{\pi_1\cup \pi_2}(T_1 \twedge T_2)$. %
\end{remark}

\subsubsection{Constructing $\RR_{\pi_1\cup\pi_2}(T_1+T_2)$}

Applying Lemma \ref{lem:bending} to the decomposition (\ref{eq:disjoint}) gives the following:

\begin{proposition}
\label{prop:decomp1}
There is a surjective map 
\[\RR_{\pi_1\cup \pi_2}(T_1+T_2) \to \RR_{\pi_1\cup\pi_2}(T_1\twedge T_2)\times_{P_1\times P_2} \RR(C_3)\] 
where the fiber over a point $(\rho_1, \rho_2)\in \RR_{\pi_1\cup\pi_2}(T_1\twedge T_2)\times_{P_1\times P_2} \RR(C_3) \subset \RR_{\pi_1\cup\pi_2}(T_1\twedge T_2)\times \RR(C_3)$ is homeomorphic to the double coset space \[\Stab(\rho_1)\backslash \Stab(\rho_1|_{(S_1\twedge S_2)})/ \Stab(\rho_2).\]
\end{proposition}

Since $\RR_{\pi_1\cup\pi_2}(T_1\twedge T_2)$ is a little tricky to work with, we can get a more explicit description of $\RR_{\pi_1\cup\pi_2}(T_1 + T_2)$ by using the decomposition in Equation \ref{eq:sequential}. Then Lemma \ref{lem:bending} tells us that $\RR_{\pi_1\cup\pi_2}(T_1 + T_2)$ can be understood by studying the fibers of a surjective map \[\phi:\RR_{\pi_1\cup\pi_2}(T_1 + T_2) \to \RR_{\pi_2}(T_2)\times_{P_2}(\RR_{\pi_1}(T_1)\times_{P_1} \RR(C_3)).\]
It will be useful to decompose this character variety into pieces. Recall that $\RR(C_3) = h(\ol{\A})\cup h(\ol{\H})$.
Thus 
\begin{align*}
\RR_{\pi_2}(T_2)\times_{P_2}(\RR_{\pi_1}(T_1)\times_{P_1} \RR(C_3)) =& \RR_{\pi_2}(T_2)\times_{P_2}(\RR_{\pi_1}(T_1)\times_{P_1} h(\ol{\H}))\cup \\&\RR_{\pi_2}(T_2)\times_{P_2}(\RR_{\pi_1}(T_1)\times_{P_1} h(\ol{\A})).
\end{align*}
Using this, we can define:
\[\RR_{\pi_1\cup\pi_2}^{\ol{\H}}(T_1 + T_2) := \phi\inv(\RR_{\pi_2}(T_2)\times_{P_2}(\RR_{\pi_1}(T_1)\times_{P_1} h(\ol{\H})))\]
\[\RR_{\pi_1\cup\pi_2}^{\ol{\A}}(T_1 + T_2) := \phi\inv(\RR_{\pi_2}(T_2)\times_{P_2}(\RR_{\pi_1}(T_1)\times_{P_1} h(\ol{\A})))\]
and thus
\[\RR_{\pi_1\cup\pi_2}(T_1+T_2) = \RR_{\pi_1\cup\pi_2}^{\ol{\H}}(T_1 + T_2)\cup \RR_{\pi_1\cup\pi_2}^{\ol{\A}}(T_1 + T_2)\]

\begin{definition}
Recall the map $\gamma_i:\RR_{\pi_i}(T_i)\to [0,1]$ which sends $\rho$ to the $\gamma$-coordinate of its image in the pillowcase. Let $\RR_{\pi_1}(T_1)\times_{[0,\pi]} \RR_{\pi_2}(T_2)$ be the fiber product with respect to the maps $\gamma_1$ and $\gamma_2$. 
\end{definition}

\begin{proposition}
\label{thm:notniceA}
Then there is a surjection $\phi^{\ol{\A}}: \RR_{\pi_1\cup\pi_2}^{\ol{\A}}(T_1 + T_2) \to \RR_{\pi_1}(T_1)\times_{[0,\pi]} \RR_{\pi_2}(T_2)$ given by $\rho\mapsto (\rho|_{T_1}, \rho|_{T_2})$. The fiber over $(\rho_1, \rho_2)$ is homeomorphic to:
\begin{itemize}
\item $S^0$ if $\gamma(\rho_1)\in\{0,\pi\}$ and $\wStab(\rho_1) = \wStab(\rho_2) = (\Z/2,\Z/2)$. In this case, there are two representations $\rho_1, \rho_2 \in \RR_{\pi_1\cup\pi_2}^{\ol{\A}}(T_1 + T_2)$ with $\phi^{\ol{\A}}(\rho) = (\rho_1,\rho_2)$ such that $p_3(\rho_1) = (\gamma(\rho_1), \gamma(\rho_1)+\gamma(\rho_2))$ and $p_3(\rho_1) = (\gamma(\rho_1), \gamma(\rho_1)-\gamma(\rho_2))$.

\item $S^1$ if $\wStab(\rho_i) = (\Z/2,\U)$ and $\wStab(\rho_j)$ is $(\Z/2,\U)$ or $(\Z/2, \Z/2)$. For any $\rho$ in this fiber over $(\rho_1,\rho_2)$, \[p_3(\rho) = (\gamma(\rho_1), \gamma(\rho_1)+\gamma(\rho_2)).\]

\item $\pt$ otherwise. If $\rho$ is the unique representation in the fiber over $(\rho_1, \rho_2)$, then \[p_3(\rho) = (\gamma(\rho_1), \gamma(\rho_1)+\gamma(\rho_2)).\]
\end{itemize}

\end{proposition}

\begin{proof}
A priori, the range of $\phi^{\ol{\A}}$ is $\RR_{\pi_1}(T_1)\times \RR_{\pi_2}(T_2)$. However, if $\phi^{\ol{\A}}(\rho) = (\rho_1, \rho_2)$, then $\gamma(\rho_1) = \gamma_1(\rho) = \gamma_2(\rho) = \gamma(\rho_2)$, $\theta_1(\rho) = \theta(\rho_1)$ and $\theta_2(\rho) = \theta(\rho_2)$, showing that 
the image of $\phi^{\ol{\A}}$ lies in the subspace $\RR_{\pi_1}(T_1)\times_{[0,\pi]} \RR_{\pi_2}(T_2)$.

To identify the fiber of the map over $(\rho_1, \rho_2)$, it is first necessary to calculate \[p_{h(\ol{\A})}\inv(\gamma(\rho_1), \theta(\rho_1), \gamma(\rho_2), \theta(\rho_2)).\] For each $\eta\in p_{h(\ol{\A})}\inv(\gamma(\rho_1), \theta(\rho_1), \gamma(\rho_2), \theta(\rho_2))$, one must then compute the space of gluing parameters representing the possible ways to glue the representations $\eta\cup\rho_1\cup\rho_2$. The first of these spaces is identified by Lemma \ref{lem:revengA}, while the second can be identified using Lemma \ref{lem:bending}.

\noindent\textbf{Case 1:} $\gamma(\rho_1)\in(0,\pi)$.

\noindent By Lemma \ref{lem:revengA}, there is a unique representation $\eta\in p_{h(\ol{\A})}\inv(\gamma(\rho_1), \theta(\rho_1), \gamma(\rho_2), \theta(\rho_2))$. Because $\sin(\gamma(\eta)) \ne 0$, we know that $\wStab_{S_1}(\eta) = \wStab(\rho_1) = \wStab(\rho_2) = (\Z/2, \Z/2)$. By Lemma \ref{lem:bending}, there is a unique way to glue the representations $\eta \cup \rho_1$. Clearly $\wStab_{P_2}(\eta\cup \rho_1) = (\Z/2, \Z/2)$ and so there is a unique representation $\eta \cup \rho_1 \cup \rho_2$. Thus the fiber over $(\rho_1, \rho_2)$ is a point.

\noindent\textbf{Case 2:} $\gamma(\rho_1)\in\{0,\pi\}$ and $\theta(\rho_1), \theta(\rho_2)\in(0,\pi)$.

\noindent By Lemma \ref{lem:revengA}, there are two representations $\eta_1,\eta_2\in p_{h(\ol{\A})}\inv(\gamma(\rho_1), \theta(\rho_1), \gamma(\rho_2), \theta(\rho_2))$.
Since $\theta(\rho_1)$ and $\theta(\rho_2)$ lie in $(0,\pi)$, we know that $\wStab_{S_1}(\eta_i) = \wStab_{S_2}(\eta_i) = \wStab(\rho_1) = \wStab(\rho_2) = (\Z/2, \Z/2)$. By Lemma \ref{lem:bending}, there is a unique way to glue the representations $\eta_i \cup \rho_1$. Clearly $\wStab_{P_2}(\eta_i\cup \rho_1) = (\Z/2, \Z/2)$ and so there is a unique representation $\eta_i \cup \rho_1 \cup \rho_2$. Thus the fiber over $(\rho_1, \rho_2)$ is two points.

\noindent\textbf{Case 3:} $\gamma(\rho_1)\in\{0,\pi\}$, $\wStab(\rho_i) = (\U,\U)$, and $\wStab(\rho_j) = (\Z/2,\Z/2)$.

\noindent By Lemma \ref{lem:revengA}, there is a unique $\eta\in p_{h(\ol{\A})}\inv(\gamma(\rho_1), \theta(\rho_1), \gamma(\rho_2), \theta(\rho_2))$.
$\wStab_{S_j}(\eta) = \wStab(\rho_j) = (\Z/2, \Z/2)$. Thus there is a unique representation $\eta\cup\rho_j$ and $\wStab_{\rho_i}(\eta\cup\rho_j) = (\Z/2, \U)$.
Lemma \ref{lem:bending} tells us that there is a unique representation $\eta\cup\rho_1\cup\rho_2$ and so the fiber of $(\rho_1, \rho_2)$ is a point.

\noindent\textbf{Case 4:} $\gamma(\rho_1)\in\{0,\pi\}$, $\wStab(\rho_i) = (\Z/2,\U)$, and $\wStab(\rho_j) = (\Z/2,\Z/2)$.

\noindent By Lemma \ref{lem:revengA}, there is a unique $\eta\in p_{h(\ol{\A})}\inv(\gamma(\rho_1), \theta(\rho_1), \gamma(\rho_2), \theta(\rho_2))$.
$\wStab_{S_j}(\eta) = (\Z/2, \Z/2)$. Thus there is a unique representation $\eta\cup\rho_j$ and $\wStab_{\rho_i}(\eta\cup\rho_j) = (\Z/2, \U)$. Lemma \ref{lem:bending} tells us that the gluing parameter for gluing $\eta\cup\rho_1$ to $\rho_2$ is $\Z/2\backslash \U/\Z/2\cong S^1$. So the fiber over $(\rho_1,\rho_2)$ is $S^1$.

\noindent\textbf{Case 5:} $\gamma(\rho_1)\in\{0,\pi\}$ and $\wStab(\rho_1) = \wStab(\rho_2) = (\Z/2,\U)$.

\noindent By Lemma \ref{lem:revengA}, there is a unique $\eta\in p_{h(\ol{\A})}\inv(\gamma(\rho_1), \theta(\rho_1), \gamma(\rho_2), \theta(\rho_2))$.
$\wStab_{S_1}(\eta) = (\U, \U)$. There is a unique representation $\eta\cup\rho_1$ and $\wStab_{S_2}(\eta\cup\rho_1) = (\Z/2, \U)$. There is a unique representation $\eta\cup\rho_1\cup\rho_2$ and so the fiber over $(\rho_1, \rho_2)$ is a point.

\noindent\textbf{Case 6:} $\gamma(\rho_1)\in\{0,\pi\}$, $\wStab(\rho_i) = (\Z/2,\U)$, and $\wStab(\rho_j) = (\U,\U)$.

\noindent By Lemma \ref{lem:revengA}, there is a unique $\eta\in p_{h(\ol{\A})}\inv(\gamma(\rho_1), \theta(\rho_1), \gamma(\rho_2), \theta(\rho_2))$.
$\wStab_{S_j}(\eta) = (\U, \U)$. There is a unique representation $\eta\cup\rho_j$ and $\wStab_{S_i}(\eta\cup\rho_j) = (\Z/2, \U)$. Lemma \ref{lem:bending} tells us that the gluing parameter for gluing $\eta\cup\rho_1$ to $\rho_2$ is $\Z/2\backslash \U/\Z/2\cong S^1$. So the fiber over $(\rho_1,\rho_2)$ is $S^1$.

\noindent\textbf{Case 7:} $\gamma(\rho_1)\in\{0,\pi\}$, $\wStab(\rho_i) = (\U,\U)$, and $\wStab(\rho_j) = (\U,\U)$.

\noindent By Lemma \ref{lem:revengA}, there is a unique $\eta\in p_{h(\ol{\A})}\inv(\gamma(\rho_1), \theta(\rho_1), \gamma(\rho_2), \theta(\rho_2))$.
$\wStab_{S_1}(\eta) = (\U, \U)$. There is a unique representation $\eta\cup\rho_1$ and $\wStab_{S_2} = (\U, \U)$. There is a unique representation $\eta\cup\rho_1\cup\rho_2$ and so the fiber over $(\rho_1, \rho_2)$ is a point.

\bigskip

\noindent To see the images of the fibers under $p_3$ note that $p_3(\eta\cup\rho_1\cup\rho_2) = p_3(\eta)$, which is given in Lemma \ref{lem:revengA}.

\end{proof}

\begin{definition}
Let $\RR_{\pi_1}(T_1)\times_{\{0,\pi\}}\RR_{\pi_2}(T_2) := \{(\rho_1, \rho_2)\in \RR_{\pi_1}(T_1)\times_{\{0,\pi\}}\RR_{\pi_2}(T_2) \mid \gamma(\rho_1)\in\{0,\pi\}\}$. 
\end{definition}

\begin{proposition}
\label{thm:notniceH}
There is a surjective map $\phi^{\ol{\H}}: \RR_{\pi_1\cup\pi_2}^{\ol{\H}}(T_1+T_2)\to \RR_{\pi_1}(T_1)\times_{\{0,\pi\}}\RR_{\pi_2}(T_2)$ and the fiber over $(\rho_1,\rho_2)$ is homeomorphic to: 
\begin{itemize}
    \item $S^1$ if $\wStab(\rho_1) = \wStab(\rho_2) = (\Z/2,\Z/2)$. In this case, parameterize the fiber by $[\rho_1+\rho_2]_\psi$ for $\psi\in S^1$. Then \[p_3([\rho_1+\rho_2]_\psi) = (\gamma(\rho_1),\arccos(\cos\theta(\rho_1)\cos\theta(\rho_2) +\sin\theta(\rho_1)\sin\theta(\rho_2)\cos\psi)).\] 
    The image of $\{[\rho_1+\rho_2]_\psi\}_{\psi\in S^1}$ under $p_3$ is the line segment
    \[\{(\gamma,\theta)\in P\mid \gamma = \gamma(\rho_1), \theta\in[\arccos(\cos(\theta(\rho_1)-\theta(\rho_2))),\arccos(\cos(\theta(\rho_1)+\theta(\rho_2)))]\}.\]
    In particular, $\rho_0, \rho_\pi\in \ol{\A}$ and the endpoints of this line segment are given by \[p(\rho_0) = (\gamma(\rho_1),\arccos(\cos(\theta(\rho_1)-\theta(\rho_2))))\]
    \[p(\rho_\pi) = (\gamma(\rho_1),\arccos(\cos(\theta(\rho_1)+\theta(\rho_2)))).\]
    
    \item $S^1$ if $\wStab(\rho_i) = (\Z/2,\U)$ and $\wStab(\rho_j)$  is  $(\Z/2,\U)$ or $(\Z/2, \Z/2)$. For any $\rho$ in this fiber over $(\rho_1,\rho_2)$, \[p_3(\rho) = (\gamma(\rho_1), \gamma(\rho_1)+\gamma(\rho_2)).\]
    
    \item $\pt$ otherwise. If $\rho$ is the unique representation in this fiber over $(\rho_1,\rho_2)$, then \[p_3(\rho) = (\gamma(\rho_1), \gamma(\rho_1)+\gamma(\rho_2)).\]
\end{itemize}

\end{proposition}

\begin{proof}

In the last two cases, Lemma \ref{lem:revengH} shows that every representation in the fiber also lies in $\RR^{\ol{\A}}_{\pi_1\cup \pi_2}(T_1+T_2)$, and thus the results follow directly from Proposition \ref{thm:notniceA}. Thus, only the first case needs to be examined.

If $\wStab(\rho_1) = \wStab(\rho_2) = (\Z/2,\Z/2)$, then $\theta(\rho_1),\theta(\rho_2)\in (0,\pi)$ and by Lemma \ref{lem:revengH}, \[p_{\ol{\H}}\inv (\gamma(\rho_1), \theta(\rho_1), \gamma(\rho_2), \theta(\rho_2))\cong S^1.\] Let $\eta_\psi\in\RR(C_3)$ be the parameterization of this fiber from Lemma \ref{lem:revengH}. According to Lemma \ref{lem:bending}, when gluing along $S_1$ and $S_2$, the gluing parameters are both trivial. Thus, there is a circle's worth of representations $[\rho_1+\rho_2]_\psi := \rho_1\cup\eta_\psi\cup \rho_2$ in the fiber. $p_3([\rho_1+\rho_2]_\psi) = p_3(\eta_\psi)$ and so the rest of the theorem follows from Lemma \ref{lem:revengH}.
\end{proof}

\begin{remark}
If $(\rho_1,\rho_2)\in \RR_{\pi_1}(T_1)\times_{\{0,\pi\}}\RR_{\pi_2}(T_2)$ such that $\wStab(\rho_i) = (\Z/2,\Z/2)$ then $[\rho_1+\rho_2]_\psi$ is necessarily non-binary dihedral if $\sin\psi\ne 0$. This is because the restriction of such a representation to $C_3$ is not binary dihedral, which is clear from Figure \ref{fig:spheretriangle}.
\end{remark}

Combining the results of Propositions \ref{thm:notniceA} and \ref{thm:notniceH}, the fibers of $\phi: \RR_{\pi_1\cup\pi_2}(T_1+T_2) \to \RR_{\pi_1}(T_1)\times_{[0,\pi]} \RR_{\pi_2}(T_2)$ can be identified.

\begin{theorem}
\label{thm:notnice}
There is a surjective map $\phi: \RR_{\pi_1\cup\pi_2}(T_1+T_2) \to \RR_{\pi_1}(T_1)\times_{[0,\pi]} \RR_{\pi_2}(T_2)$ and the fiber over $(\rho_1, \rho_2)$ is homeomorphic to:

\begin{itemize}
    \item $S^1$ if $\gamma(\rho_1) \in \{0,\pi\}$ and $\wStab(\rho_1) = \wStab(\rho_2) = (\Z/2,\Z/2)$. In this case, parameterize the fiber by $[\rho_1+\rho_2]_\psi$ for $\psi\in S^1$. Then \[p_3([\rho_1+\rho_2]_\psi) = (\gamma(\rho_1),\arccos(\cos\theta(\rho_1)\cos\theta(\rho_2) +\sin\theta(\rho_1)\sin\theta(\rho_2)\cos\psi)).\] 
    The image of $\{[\rho_1+\rho_2]_\psi\}_{\psi\in S^1}$ under $p_3$ is the line segment
    \[\{(\gamma,\theta)\in P\mid \gamma = \gamma(\rho_1), \theta\in[\arccos(\cos(\theta(\rho_1)-\theta(\rho_2))),\arccos(\cos(\theta(\rho_1)+\theta(\rho_2)))]\}.\]
    In particular, $\rho_0, \rho_\pi\in \ol{\A}$ and the endpoints of this line segment are given by \[p(\rho_0) = (\gamma(\rho_1),\arccos(\cos(\theta(\rho_1)-\theta(\rho_2))))\]
    \[p(\rho_\pi) = (\gamma(\rho_1),\arccos(\cos(\theta(\rho_1)+\theta(\rho_2)))).\]
    
    \item $S^1$ if $\wStab(\rho_i) = (\Z/2,\U)$ and $\wStab(\rho_j)$  is  $(\Z/2,\U)$ or $(\Z/2, \Z/2)$. For any $\rho$ in this fiber , $p_3(\rho) = (\gamma(\rho_1), \theta(\rho_1)+\theta(\rho_2))$.
    
    \item $\pt$ otherwise. If $\rho$ is the unique representation in this fiber over $(\rho_1,\rho_2)$, then $p_3(\rho) = (\gamma(\rho_1), \theta(\rho_1)+\theta(\rho_2))$.
\end{itemize}
\end{theorem}

\begin{remark}
Tangle addition is \textit{not} commutative, however by Theorem \ref{thm:notnice} $\RR_{\pi_1\cup \pi_2}(T_1+T_2) = \RR_{\pi_1\cup \pi_2}(T_2+T_1)$.
\end{remark}

\subsubsection{Good Pairs}

\begin{definition}
For tangles $T_1$ and $T_2$ with boundary $(S^2,4)$ and perturbation data $\pi_1$ and $\pi_2$, call $(T_1,\pi_1)$ and $(T_2,\pi_2)$ a \emph{good pair} if the following conditions hold:

\begin{enumerate}[label={\bfseries (G\arabic*):}, ref=(G\arabic*), leftmargin=3\parindent]
    \item \label{con:generic} $\RR_{\pi_i}(T_i)^{(\Z/2,\U)}$ is empty. This is equivalent to the condition that $p_i(\RR_{\pi_i}(T_i)^{\Z/2})$ misses the corners of the pillowcase.
    
    \item \label{con:immersion} Each $\RR_{\pi_i}(T_i)$ is a manifold with boundary and $p_i:\RR_{\pi_i}(T_i)\to P_i$ are immersions.
    
    \item \label{con:edges} %
    $p_i$ is transverse to $\{(\gamma,\theta)\in P_{i+1} \mid \gamma\in \{0,\pi\}, \theta\in(0,\pi)\}$.
    
    \item \label{con:samevert} If there exist $\rho_1\in\RR_{\pi_1}(T_1)$ and $\rho_2\in\RR_{\pi_2}(T_2)$ such that $d_{\rho_i}p_i(T_{\rho_i}\RR_{\pi_1}(T_i)) = \langle (0,1) \rangle$ with respect to the basis $(\frac{\partial}{\partial \gamma}, \frac{\partial}{\partial\theta})$ of $TP_i$, then $\gamma(\rho_1)\ne \gamma(\rho_2)$.
\end{enumerate}
\end{definition}

\begin{proposition}
\label{prop:goodpair}
For tangles $T_1$ and $T_2$ it is always possible to find arbitrarily small perturbation data $\pi_1$ and $\pi_2$ so that $(T_1, \pi_1)$ and $(T_2, \pi_2)$ are a good pair.
\end{proposition}

\begin{proof}
First, Theorem \ref{thm:reg} guarantees that there is some arbitrarily small perturbation $\pi_i'$ such that $\RR_{\pi_i'}(T_i)$ satisfies Conditions \ref{con:generic} and \ref{con:immersion}. Then each $T_i$ can be composed with the tangle cobordism in Figure \ref{fig:ShearingB}. The result is the same tangle, but with perturbation $\pi_i := \pi_i'\cup (S_\gamma,f_i, t_i)$. Since the shearing functions $f_i:\R\to \R$ are smooth, this guarantees that this $\RR_{\pi_i}(T_i)$ still satisfies Condition \ref{con:immersion}. Since Condition \ref{con:generic} is an open condition, for small enough $t_i>0$, $\RR_{\pi_i}(T_i)$ still satisfies Condition \ref{con:generic}.
It is not hard to see that functions $f_i$ can always be chosen such that Conditions \ref{con:edges} and \ref{con:samevert} are also satisfied.
\end{proof}

\begin{remark}
Being a good pair is a stable condition, that is, if $(T_1, \pi_1)$ and $(T_2,\pi_2)$ are a good pair and $\pi'_1$ and $\pi'_2$ are perturbations on $T_1$ and $T_2$ with perturbation parameters $t_1$ and $t_2$, then there exist constants $0<s_1, s_2$ such that for all $t_i\in[0, s_i)$ the pair $(T_1,\pi_1+\pi'_1)$ and $(T_2, \pi_2+\pi'_2)$ is good. This can be seen by observing that each individual condition is stable in this sense.
\end{remark}

To get a more explicit description of $\RR_{\pi_1\cup \pi_2}(T_1+T_2)$ along with its image in the pillowcase, we can get simplified versions of Proposition \ref{prop:decomp1} and Theorem \ref{thm:notnice} when Condition \ref{con:generic} is satisfied.

\begin{lemma}
\label{prop:nicecomp}
If the pair $(T_1,\pi_1)$ and $(T_2,\pi_2)$ satisfy condition \ref{con:generic}, then $\RR_{\pi_1\cup \pi_2}(T_1+T_2) \cong \RR_{\pi_1\cup\pi_2}(T_1\twedge T_2)\times_{P_1\times P_2} \RR(C_3)$
\end{lemma}

\begin{proof}
If the pair $(T_1,\pi_1)$ and $(T_2,\pi_2)$ satisfy Condition \ref{con:generic}, then $\RR_{\pi_1\cup\pi_2}(T_1\twedge T_2)^{\Z/2,\U}$ is empty. Then Lemma \ref{lem:generic} proves the claim.%
\end{proof}

\begin{theorem}
\label{thm:nice}
If the pair $(T_1,\pi_1)$ and $(T_2, \pi_2)$ satisfy Condition \ref{con:generic} then there is a surjective map $\phi: \RR_{\pi_1\cup\pi_2}(T_1+T_2) \to \RR_{\pi_1}(T_1)\times_{[0,\pi]} \RR_{\pi_2}(T_2)$ such that the fiber over $(\rho_1, \rho_2)$ is homeomorphic to:

\begin{itemize}
    \item $S^1$ if $\gamma(\rho_1) \in \{0,\pi\}$ and $\theta(\rho_1),\theta(\rho_2)\in(0,\pi)$. In this case parametrize the fiber by $[\rho_1+\rho_2]_\psi$ for $\psi\in S^1$. Then $p_3([\rho_1+\rho_2]_\psi)$ is described by the same formula as Theorem \ref{thm:notnice}.
    
    \item $\pt$ otherwise. If $\rho$ is the unique representation in this fiber over $(\rho_1,\rho_2)$, then \[p_3(\rho) = (\gamma(\rho_1), \theta(\rho_1)+\theta(\rho_2)).\]
\end{itemize}
\end{theorem}

Using the fact from Lemma \ref{prop:nicecomp} that for good pairs, $\RR_{\pi_1\cup\pi_2}(T_1+T_2) \cong \RR_{\pi_1\cup\pi_2}(T_1\twedge T_2)\times_{P_1\times P_2} \RR(C_3)$ and the decomposition $\RR(C_3) = \RR(C_3)^*\twedge h(\SV)$ from Definition \ref{def:rcstar}, we can decompose $\RR_{\pi_1\cup\pi_2}(T_1+T_2)$ as $M\sqcup N$ where:
\begin{equation}
\label{eq:decompM}
M := \RR_{\pi_1\cup\pi_2}(T_1\twedge T_2)\times_{P_1\times P_2} \RR(C_3)^*
\end{equation}
\begin{equation}
\label{eq:decompN}
N := \RR_{\pi_1\cup\pi_2}(T_1\twedge T_2)\times_{P_1\times P_2} h(\SV)
\end{equation}

The next goal is to show that $M$ is a 1-manifold and that $p_3|_M:M\to P_3$ is an immersion. First note the maps $\RR(C_3)^*\to P_1\times P_2\times P_3$ and $\RR_{\pi_1\cup\pi_2}(T_1\twedge T_2)\to P_1\times P_2$ are Lagrangian immersions by \cite{CHK}*{Theorem A} and the fact that Condition \ref{con:immersion} holds. Then by Theorem \ref{thm:lagcomp}, $M$ is a manifold and $p_3|_M$ is an immersion if the maps $\widehat{\mu}:\RR(C_3)^*\to P_1\times P_2$ and $\nu:\RR_{\pi_1\cup\pi_2}(T_1\twedge T_2)\to P_1\times P_2$ are transverse. 
Define the map $\mu:\Sols^* \to P_1\times P_2$ as the composition of the double cover map $\Sols^*\to \RR(C_3)^*$ and $\widehat{\mu}$.
Because this quotient map is surjective, %
$\widehat{\mu}$ and $\nu$ are transverse if $\mu$ and $\nu$ are transverse.

\begin{proposition}
\label{prop:transverse}
If $(T_1, \pi_1)$ and $(T_2, \pi_2)$ are a good pair, then $\mu:\Sols^*\to \POne\times \PTwo$ and $\nu:\RR_{\pi_1}(T_1)\times \RR_{\pi_2}(T_2)\to \POne\times \PTwo$ are transverse. 
\end{proposition}

\begin{proof}

Parameterize $\POne\times \PTwo$ by $\{(\gamma_1,\theta_1,\gamma_2,\theta_2)\mid (\gamma_1,\theta_1)\in \POne, (\gamma_2, \theta_2)\in \PTwo\}$. At any point $p\in \POne\times \PTwo$, the tangent space $T_p(\POne\times \PTwo)$ has basis $\{\frac{\partial}{\partial\gamma_1}, \frac{\partial}{\partial\theta_1}, \frac{\partial}{\partial\gamma_2}, \frac{\partial}{\partial\theta_2}\}$. Because $\Sols^* = \H^\dag\sqcup \A$,
$\mu\pitchfork \nu$ can be shown by showing that $\mu|_\A\pitchfork\nu$ and $\mu|_{\H^\dag}\pitchfork\nu$.

First, we show $\mu|_\A\pitchfork\nu$. Every point in $\A$, is of the form $\widetilde{\Gamma}(\gamma,\theta,\frac{\pi}{2},\beta)$. Then the map $\mu|_\A:\A\to \POne\times \PTwo$ is given by
$\mu|_\A:(\gamma, \theta,\frac{\pi}{2},\beta)\mapsto (\gamma,\beta,\gamma,\theta-\beta)$.
For any $\mathbf{x}\in \A$, $T_\mathbf{x} \A$ has the basis $\{\frac{\partial}{\partial\gamma}, \frac{\partial}{\partial\theta}, \frac{\partial}{\partial\beta}\}$. This gives
\[d\mu|_\A = \bmat{
1&0&0\\
0&0&1\\
1&0&0\\
0&1&-1}\]
Thus $d_\mathbf{x}\mu|_\A(T_\mathbf{x}\A)=\langle(1,0,1,0), (0,1,0,0), (0,0,0,1)\rangle$.
The Conditions \ref{con:immersion} and \ref{con:edges} guarantee that the range of $d\nu$ includes a vector of the form $(1,x,0,0)$ or $(0,0,1,x)$ for $x\in\R$. Therefore $\mu|_\A\pitchfork \nu$ because
\begin{align*}
T_{(p_1\times p_2)(h(\mathbf{x}))}(\POne\times \PTwo)&=\langle(1,0,1,0), (0,1,0,0), (0,0,0,1)\rangle\oplus\langle(1,x,0,0)\rangle\\
&=\langle(1,0,1,0), (0,1,0,0), (0,0,0,1)\rangle\oplus\langle(0,0,1,x)\rangle.
\end{align*}

\bigskip

Now we show that $\mu|_{\H^\dag}\pitchfork\nu$. For any $\mathbf{x}\in \H^\dag$, $T_\mathbf{x} \H^\dag$ has the basis $\{\frac{\partial}{\partial\theta}, \frac{\partial}{\partial\alpha}, \frac{\partial}{\partial\beta}\}$. If $\gamma_0\in \{0,\pi\}$, then 
\[\mu|_{\H^\dag}:(\gamma_0, \theta,\alpha,\beta)\mapsto (\gamma_0,\arccos(\sin\alpha\cos\beta),\gamma_0,\arccos(\sin\alpha\cos\beta\sin\theta+\sin\alpha\sin\beta\cos\theta))\]

whose derivative is

\[d\mu|_{\H^\dag} = \bmat{0&0&0\\
0&-\frac{\cos\alpha\cos\beta}{\sqrt{1-\sin^2\alpha\cos^2\beta}} & \frac{\sin\alpha\sin\beta}{\sqrt{1-\sin^2\alpha\cos^2\beta}}\\
0&0&0\\
-\frac{\sin\alpha\cos(\beta+\theta)}{\sqrt{1-\sin^2\alpha\sin^2(\beta+\theta)}}&-\frac{\cos\alpha\sin(\beta+\theta)}{\sqrt{1-\sin^2\alpha\sin^2(\beta+\theta)}}&-\frac{\sin\alpha\cos(\beta+\theta)}{\sqrt{1-\sin^2\alpha\sin^2(\beta+\theta)}}
}\]
Since $\sin\alpha\ne1$ in $\H^\dag$, the denominators are never 0. It is not hard to check that \[d_\mathbf{x}\mu|_{\H^\dag}(T_\mathbf{x} \H^\dag) = \langle (0,1,0,0),(0,0,0,1) \rangle.\] Conditions \ref{con:immersion} and \ref{con:samevert} guarantee that vectors of the form $(1,x,0,0)$ and $(0,0,1,y)$ will be in the range of $d\nu$ for $x,y\in\R$. Then since \[T_{(p_1\times p_2)(h(\mathbf{x}))}\mu|_{\H^\dag}(\rho)(P_1 \times P_2)=\langle (0,1,0,0),(0,0,0,1) \rangle \oplus \langle (1,x,0,0),(0,0,1,y) \rangle,\] we know that $\mu|_{\H^\dag} \pitchfork \nu$.

\end{proof}

\begin{corollary}
\label{cor:mman}
$M$ is a 1-manifold and the map $p_3:M\to P_3$ is an immersion.
\end{corollary}

\begin{lemma}
\label{lem:isolates}
$N$ is homeomorphic to a finite set of points. The neighborhood of each such point in $\RR_{\pi_1\cup \pi_2}(T_1 + T_2)$ is either a cone on one point, two points, or four points.
\end{lemma}

\begin{proof}
Recall that $N = \RR^{\ol{\H}}_{\pi_1\cup\pi_2}(T_1+T_2)\cap \RR^{\ol{\A}}_{\pi_1\cup\pi_2}(T_1+T_2)$. By Condition \ref{con:edges}, $\RR_{\pi_1}(T_1)\times_{\{0,\pi\}}\RR_{\pi_2}(T_2)$ is a finite set of points.
Thus, by Theorem \ref{thm:nice}, $\RR^{\ol{\H}}_{\pi_1\cup\pi_2}(T_1+T_2)$ has a finite number of components, $n_p$ points and $n_c$ circles. Each circle component has precisely two points which are also in $\RR^{\ol{\A}}_{\pi_1\cup\pi_2}(T_1+T_2)$ and thus $N$ has $n_p+2n_c$ points in it. Let $\rho\in N$. Depending on $\rho$, it will have a neighborhood which is one of three types.

\noindent\textbf{Case 1:} $\Stab(\rho|_{T_1}) = \Stab(\rho|_{T_2}) = \U$.

In this case $\Stab(\rho) = \U$ and it is thus a boundary of $\RR_{\pi_1\cup \pi_2}(T_1+T_2)$ and so $\rho$ has a neighborhood in $\RR_{\pi_1\cup \pi_2}(T_1 + T_2)$ which is a cone on one point.

\noindent\textbf{Case 2:} $\Stab(\rho|_{T_1}) = \U$ and $\Stab(\rho|_{T_2}) = \Z/2$ or vice versa.

In this case, $\rho$ is isolated in $\RR^{\ol{\H}}_{\pi_1\cup\pi_2}(T_1+T_2)$, so the neighborhood of $\rho$ in $\RR_{\pi_1\cup\pi_2}(T_1+T_2)$ is homeomorphic to its neighborhood in $\RR^{\ol{\A}}_{\pi_1\cup\pi_2}(T_1+T_2)$. Condition \ref{con:samevert} guarantees that $\RR^{\ol{\A}}_{\pi_1\cup\pi_2}(T_1+T_2)$ has no isolated points. Since Proposition \ref{prop:transverse} tells us that $\RR^{\A}_{\pi_1\cup\pi_2}(T_1+T_2)$ is a 1-manifold, this tells us that the neighborhood of $\rho$ is an interval (since it is not a point in the boundary) and thus a cone on two points. %

\noindent\textbf{Case 3:} $\Stab(\rho|_{T_1}) = \Stab(\rho|_{T_2}) = \Z/2$.

In this case, Theorem \ref{thm:nice} tells us that $\rho$ lies on a circle in $\RR^{\ol{\H}}_{\pi_1\cup\pi_2}(T_1+T_2)$. In this case $\rho$ must lie on exactly one such circle.
Suppose instead that $\rho$ lies on two circles in $\RR^{\ol{\H}}_{\pi_1\cup\pi_2}(T_1+T_2)$, $[\rho_1+\rho_2]_\psi$ and $[\rho_1'+\rho_2']_\psi$. But since the restrictions to $T_1$ and $T_2$ are constant along each circle, we would get that $\rho|_{T_1} = \rho_1=\rho_1'$ and $\rho|_{T_2} = \rho_2 = \rho_2'$, which means that the two circles actually must have been the same. Thus, the neighborhood of $\rho$ in $\RR^{\ol{\H}}_{\pi_1\cup\pi_2}(T_1+T_2)$ is a cone on two points.

For the same reason as the previous case, the neighborhood of $\rho$ in $\RR^{\ol{\A}}_{\pi_1\cup\pi_2}(T_1+T_2)$ is a cone on two points. The neighborhood of $\rho$ in $\RR_{\pi_1\cup\pi_2}(T_1+T_2)$ is the union of these two neighborhoods, which only have $\rho$ in common. Thus, the neighborhood of $\rho$ in $\RR_{\pi_1\cup\pi_2}(T_1+T_2)$ is a cone on four points.
\end{proof}

\begin{corollary}
\label{cor:singular}
Let $N'$ be the singular points of $\RR_{\pi_1\cup\pi_2}(T_1+T_2)$. Then $N'\subset N$ and as a result $\RR_{\pi_1\cup\pi_2}(T_1+T_2)$ has a finite number of singular points.
\end{corollary}

\begin{proof}
Recall that $\RR_{\pi_1\cup\pi_2}(T_1+T_2) = M \sqcup N$. Corollary \ref{cor:mman} proved that $M$ is a manifold, and so a point in $M$ could only be singular in $\RR_{\pi_1\cup\pi_2}(T_1+T_2)$ if it is the limit point of elements of $N$. However, this cannot happen since $N$ is a finite set of points and $M$ and $N$ are disjoint. Thus, $N'\subset N$. It is clear that $N'$ is precisely the points which are cones on four points as described in Lemma \ref{lem:isolates}. 
\end{proof}

\begin{definition}
\label{def:circles}
When $\RR^{\ol{\H}}_{\pi_1\cup\pi_2}(T_1+T_2)$ has a finite number of point and circle components, there are two types of circles. For $\rho_i\in \RR_{\pi_i}(T_i)$, with $\gamma(\rho_1)=\gamma(\rho_2)\in\{0,\pi\}$ call the circle $[\rho_1+\rho_2]_\psi$ a \emph{corner circle} if there exists a $\psi$ for which $\theta([\rho_1+\rho_2]_\psi)\in\{0,\pi\}$ (in other words, the circle intersects a corner of the pillowcase) and an \emph{internal circle} otherwise.
\end{definition}

By Theorem \ref{thm:nice}, $[\rho_1+\rho_2]_\psi$ is a corner circle if and only if $\theta(\rho_1) = \theta(\rho_2)$ or $\theta(\rho_1) = \pi-\theta(\rho_2)$. The following lemma shows that by using arbitrarily small perturbations, it is possible to ensure that all circles are internal.

\begin{lemma}
\label{lem:internal}
If $(T_1, \pi_1)$ and $(T_2, \pi_2)$ are a good pair, then there is perturbation data, $\pi'$, for $T_1$, with perturbation parameter $t'$, such that for sufficiently small $t'$, $(T_1, \pi_1+\pi')$ and $(T_2, \pi_2)$ are a good pair and $\RR_{\pi_1\cup\pi'\cup\pi_2}(T_1+T_2)$ has no corner circles.
\end{lemma}

\begin{proof}
By Theorem \ref{thm:nice}, a corner circle occurs in $\RR_{\pi_1\cup\pi_2}(T_1+T_2)$ if and only if there exist representations $\rho_i\in \RR_{\pi_i}(T_i)$ such that $\sin(\gamma(\rho_i))=0$, $\sin(\theta(\rho_i))\ne 0$, and $p(\rho_1) = p(\rho_2)$ where $p$ is the map from each character variety $\RR_{\pi_i}(T_i)$ to its respective pillowcase.

Consider the set \[W^i_{\pi} = \{\rho\in\RR_{\pi_i\cup\pi}(T_i)\mid \sin(\gamma(\rho))=0,~ \sin(\theta(\rho))\ne 0\}.\]
In this language, $\RR_{\pis}(T_1+T_2)$ contains a corner circle if and only if $p(W^1_\emptyset)\cap (W^2_\emptyset) \ne \emptyset$. Condition \ref{con:edges} guarantees that $W^i_\emptyset$ is finite. Furthermore, Condition \ref{con:edges} implies that for each $\rho\in W_\emptyset$, $d_\rho p(T_\rho \RR_{\pi_1}(T_1)) = \langle (1,x_\rho)\rangle$ for some $x_\rho\in\R$.

Consider applying shearing perturbation in the $\theta$ direction to $T_1$ as described in Section \ref{sec:shearing} with perturbation data $(S_\theta,\sin,t_\theta)$. 
Since this perturbation acts by a homeomorphism of the pillowcase which fixes the edges $\{(\gamma,\theta)\in P\mid \sin\gamma=0\}$, there is a bijection $g:W^1_\emptyset\to W^1_{S_\theta}$ such that $p(\rho) = p(g(\rho))$. Then by Equation \ref{eq:sheartheta}, 
\begin{align*}
d_{g(\rho)} p(T_{g(\rho)} \RR_{\pi_1\cup S_\theta}(T_1)) &= \langle (1,x_\rho-2t_\theta\cos(\gamma(\rho)))\rangle\\
&= \langle (1,x_\rho\pm 2t_\theta)\rangle
\end{align*}
Fix $t_\theta\in\R$ such that $x_\rho-2t_\theta\cos(\gamma(\rho))\ne 0$ for each $\rho\in W^1_\emptyset$.

Next, add in the perturbation $(S_\gamma,\sin,t_\gamma)$. For small $t_\gamma$, there is a further bijection $f:W^1_{S_\theta}\to W^1_{S_\theta\cup S_\gamma}$. Then one can calculate that 
\[\frac{\partial}{\partial t_\gamma}\theta(f\circ g(\rho))|_{t_\gamma=0} = \sin(\theta(\rho))(x_\rho-2t_\theta\cos(\gamma(\rho))) \ne 0.\]

Thus as one increases the perturbation parameter $t_\gamma$ from 0, all points in $W^1_{S_\theta\cup S_\gamma}$ move in the $\theta$-coordinate. Thus, a sufficiently small $t_\gamma$ can be found such that no representation in $W^1_{S_\theta\cup S_\gamma}$ has the same $\theta$-coordinate as any representation in $W^2_\emptyset$, $p(W^1_{S_\theta\cup S_\gamma})\cap p(W^2_\emptyset) = \emptyset$, and thus there are no corner circles for such $\pi' = S_\theta\cup S_\gamma$. Because being a good pair is a stable condition and perturbation parameters can be arbitrarily small, $(T_1, \pi_1\cup S_\theta\cup S_\gamma)$ and $(T_2,\pi_2)$ are a good pair.
\end{proof}

\begin{figure}
\label{fig:Ex}
    \centering
    \begin{subfigure}[b]{0.25\textwidth}
        \centering
        \includegraphics[width=\textwidth]{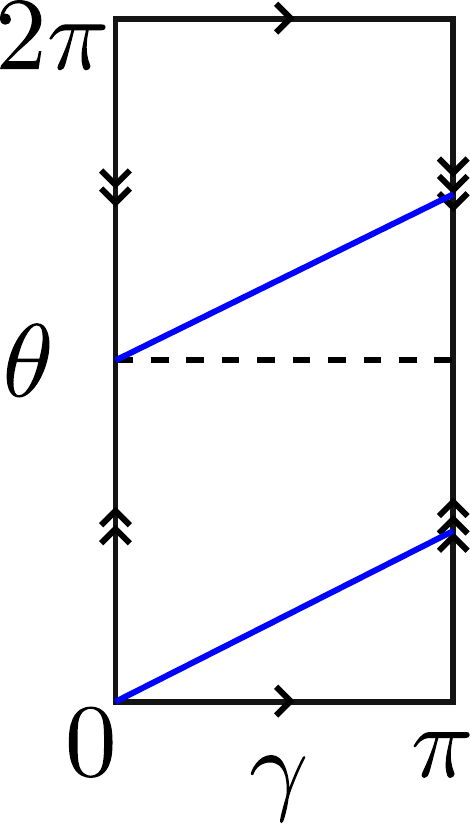}
        \caption{}
        \label{fig:Ex1}
    \end{subfigure}
    \hspace{1cm}
    \begin{subfigure}[b]{0.25\textwidth}
        \centering
        \includegraphics[width=\textwidth]{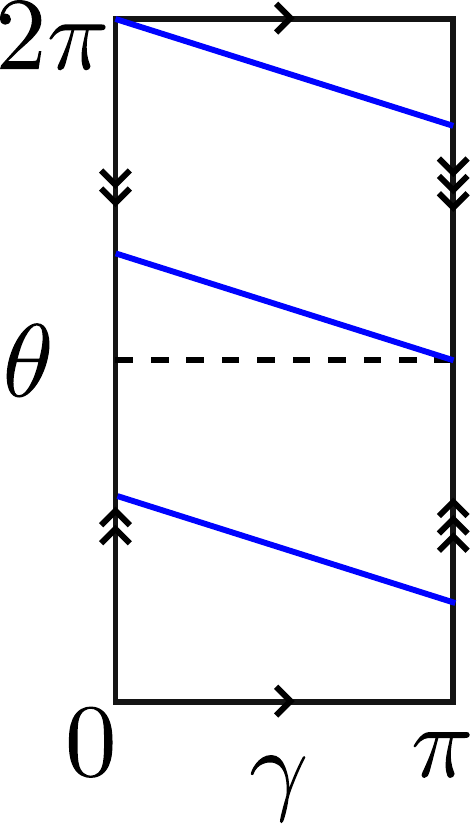}
        \caption{}
        \label{fig:Ex2}
    \end{subfigure}
    \hspace{1cm}
    \begin{subfigure}[b]{0.25\textwidth}
        \centering
        \includegraphics[width=\textwidth]{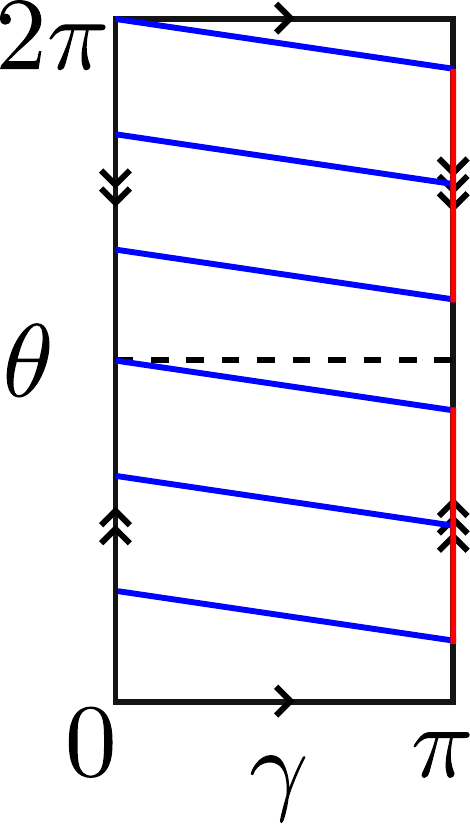}
        \caption{}
        \label{fig:Ex3}
    \end{subfigure}
    \caption{}
\end{figure}

\begin{figure}
\label{fig:Exa}
    \centering
    \begin{subfigure}[b]{0.25\textwidth}
        \centering
        \includegraphics[width=\textwidth]{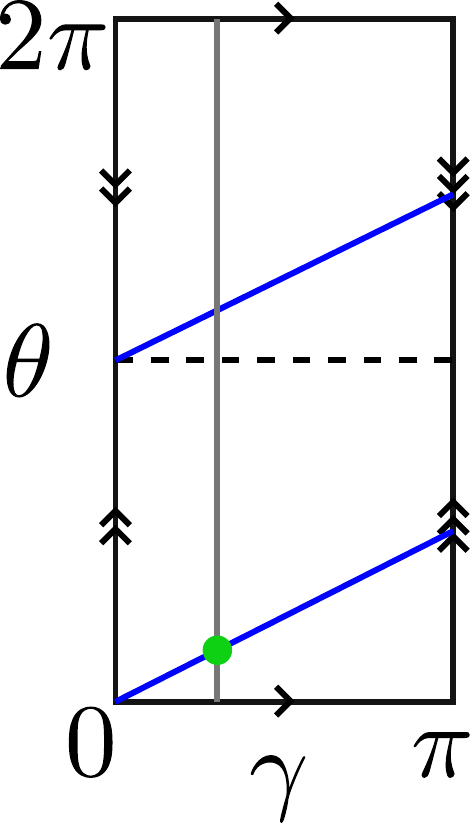}
        \caption{}
        \label{fig:Ex1Add}
    \end{subfigure}
    \hspace{1cm}
    \begin{subfigure}[b]{0.25\textwidth}
        \centering
        \includegraphics[width=\textwidth]{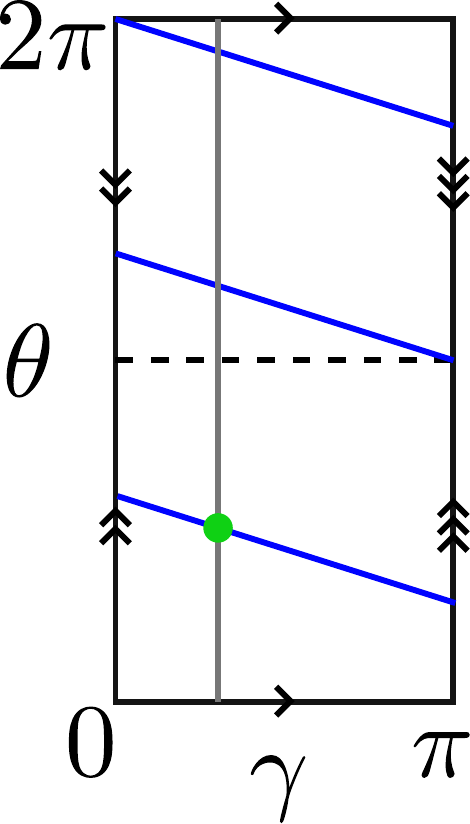}
        \caption{}
        \label{fig:Ex2Add}
    \end{subfigure}
    \hspace{1cm}
    \begin{subfigure}[b]{0.25\textwidth}
        \centering
        \includegraphics[width=\textwidth]{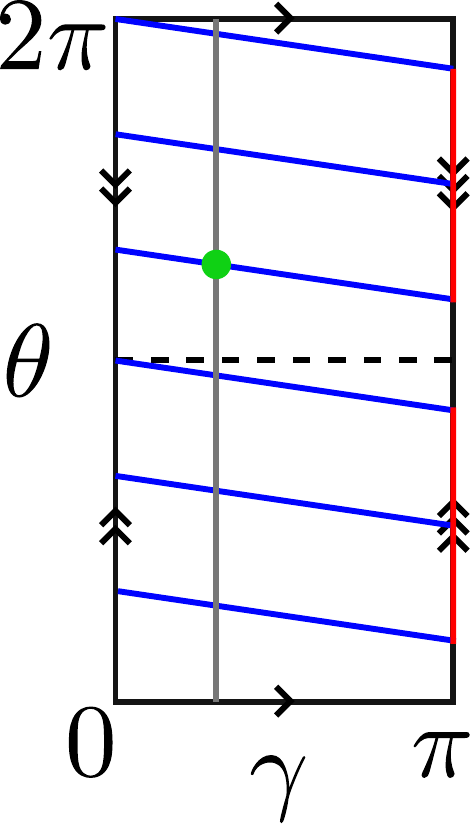}
        \caption{}
        \label{fig:Ex3Add}
    \end{subfigure}
    \caption{}
\end{figure}

\begin{figure}
\label{fig:Exb}
    \centering
    \begin{subfigure}[b]{0.25\textwidth}
        \centering
        \includegraphics[width=\textwidth]{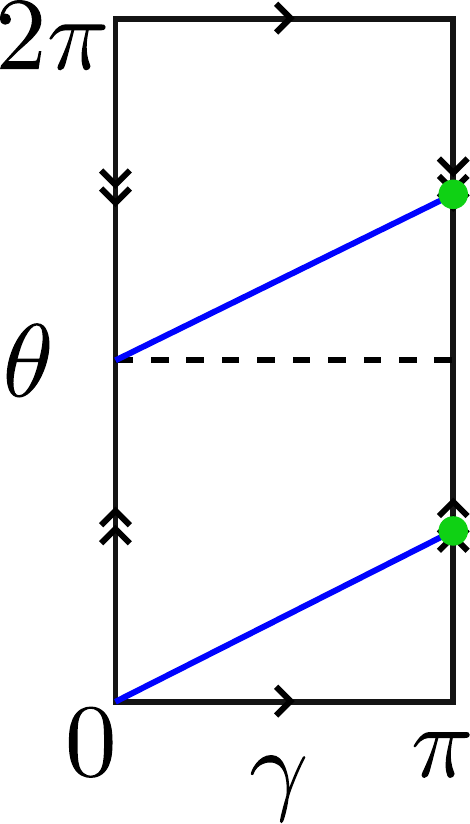}
        \caption{}
        \label{fig:Ex1Circ}
    \end{subfigure}
    \hspace{1cm}
    \begin{subfigure}[b]{0.25\textwidth}
        \centering
        \includegraphics[width=\textwidth]{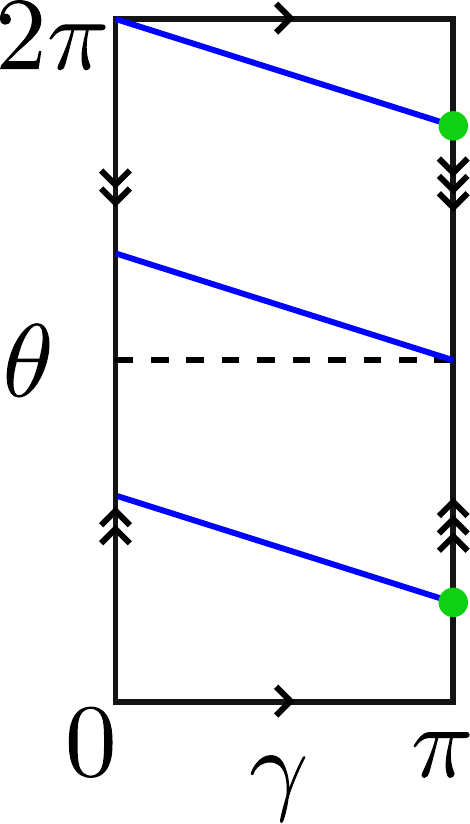}
        \caption{}
        \label{fig:Ex2Circ}
    \end{subfigure}
    \hspace{1cm}
    \begin{subfigure}[b]{0.25\textwidth}
        \centering
        \includegraphics[width=\textwidth]{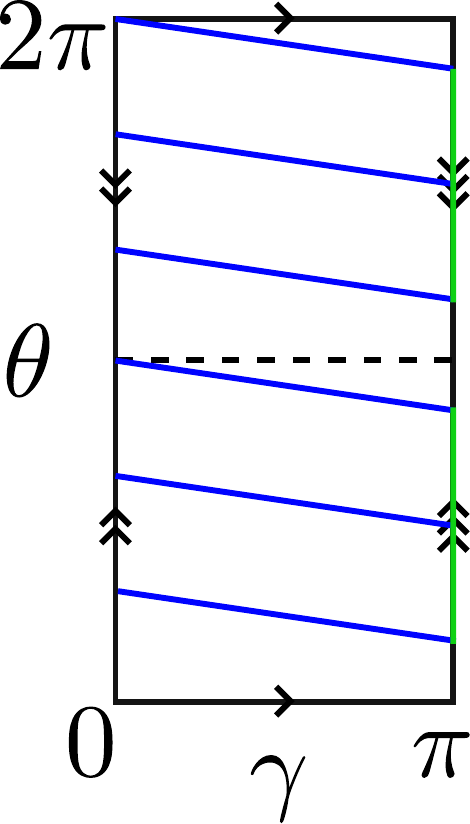}
        \caption{}
        \label{fig:Ex3Circ}
    \end{subfigure}
    \caption{}
\end{figure}

The following theorem combines the results of this subsection to summarize the structure of character varieties $\RR_{\pi_1\cup\pi_2}(T+S)$ where $\pi_1$ and $\pi_2$ are specially chosen, arbitrarily small perturbations for the tangles $T$ and $S$, respectively.

\begin{theorem}
    \label{thm:gpsum}
    Let $T$ and $S$ be tangles. Then there exist arbitrarily small perturbations $\pi_1$ in $T$ and $\pi_2$ in $S$ such that $\RR_{\pi_1\cup \pi_2}(T+S)$ has the following structure:

    There exists $N'\subset \RR_{\pi_1\cup \pi_2}(T+S)$ which is a finite set of points such that $\RR_{\pi_1\cup \pi_2}\setminus N'$ is a 1-manifold with boundary.%
    For every point $x\in N'$, there is a neighborhood of $x$ in $\RR_{\pi_1\cup \pi_2}(T+S)$ that is homeomorphic to a cone on four points. The points in $N'$ come in pairs such that if $x$ and $y$ are points in $N'$ which form a pair, then there exists a circle embedded in $\RR_{\pi_1\cup \pi_2}$ containing $x$ and $y$ and no other points from $N'$. The image of this circle in the pillowcase lies entirely within the edges $\{(\gamma, \theta)\in P \mid \gamma \in \{0,\pi\}, \theta \in(0,\pi)\}$. Note that this image does not intersect any corner of the pillowcase. One of the points $x$ or $y$ will attain the maximum $\theta$-coordinate of this image while the other will attain the minimum $\theta$-coordinate.
\end{theorem}

\begin{exmp}
If $T$ and $S$ are rational tangles whose slopes are not $\infty$, then $(T,\emptyset)$ and $(S,\emptyset)$ are a good pair.
Let $T = Q_{\frac{1}{2}}$ and $S = Q_{-\frac{1}{3}}$.
Figure \ref{fig:Ex1} shows the image of $\RR(T)$ in the pillowcase. Figure \ref{fig:Ex2} shows the image of $\RR(S)$ in the pillowcase. Then Figure \ref{fig:Ex3} shows the image of $\RR(T+S)$ in the pillowcase. $\RR^{\ol{\A}}(T+S)$ is shown in blue and $\RR^{\ol{\H}}(T+S)$ is in red. %

Figures \ref{fig:Ex1Add} and \ref{fig:Ex2Add} show representations $\rho_1 \in \RR(T)$ and $\rho_2 \in \RR(S)$ with the same $\gamma$-coordinate. Figure \ref{fig:Ex3Add} shows the unique representation $[\rho_1+\rho_2]\in \RR(T+S)$ described by Proposition \ref{thm:notniceA}.

Figures \ref{fig:Ex1Circ} and \ref{fig:Ex2Circ} show representations $\rho_1 \in \RR(T)$ and $\rho_2 \in \RR(S)$ with $\gamma(\rho_i) = \pi$. As shown by Proposition \ref{thm:notniceH}, there is a circle's worth of representations $[\rho_1+\rho_2]_\psi\in \RR(T+S)$ which restrict to $\rho_1$ and $\rho_2$ on $T$ and $S$, respectively. Figure \ref{fig:Ex3Add} shows the image of this circle in the pillowcase, which is the interval described in Proposition \ref{thm:notniceH}. Note that in this example, the circle is an internal circle.

\end{exmp}

\begin{remark}
The tangle sum is an associative operation, so arbitrary sums of tangles $\sum_{i=1}^n T_i$ are obtained by iteratively applying pairwise sums.
Then the character variety $\RR(\sum_{i=1}^n T_i)$ can be computed by repeated applications of Theorem \ref{thm:notnice}.
\end{remark}

\subsection{A Decomposition of \texorpdfstring{$\Pr$}{P}}
\label{sec:decomp}

In this section, we define a certain decomposition of $\Pr$, which, in turn, provides a decomposition of $\Sols$. These decompositions will be useful in understanding the perturbed character variety of $C_3$ in Section \ref{sec:Perturbing}.

\begin{figure}
    \centering
    \includegraphics[height=2.5in]{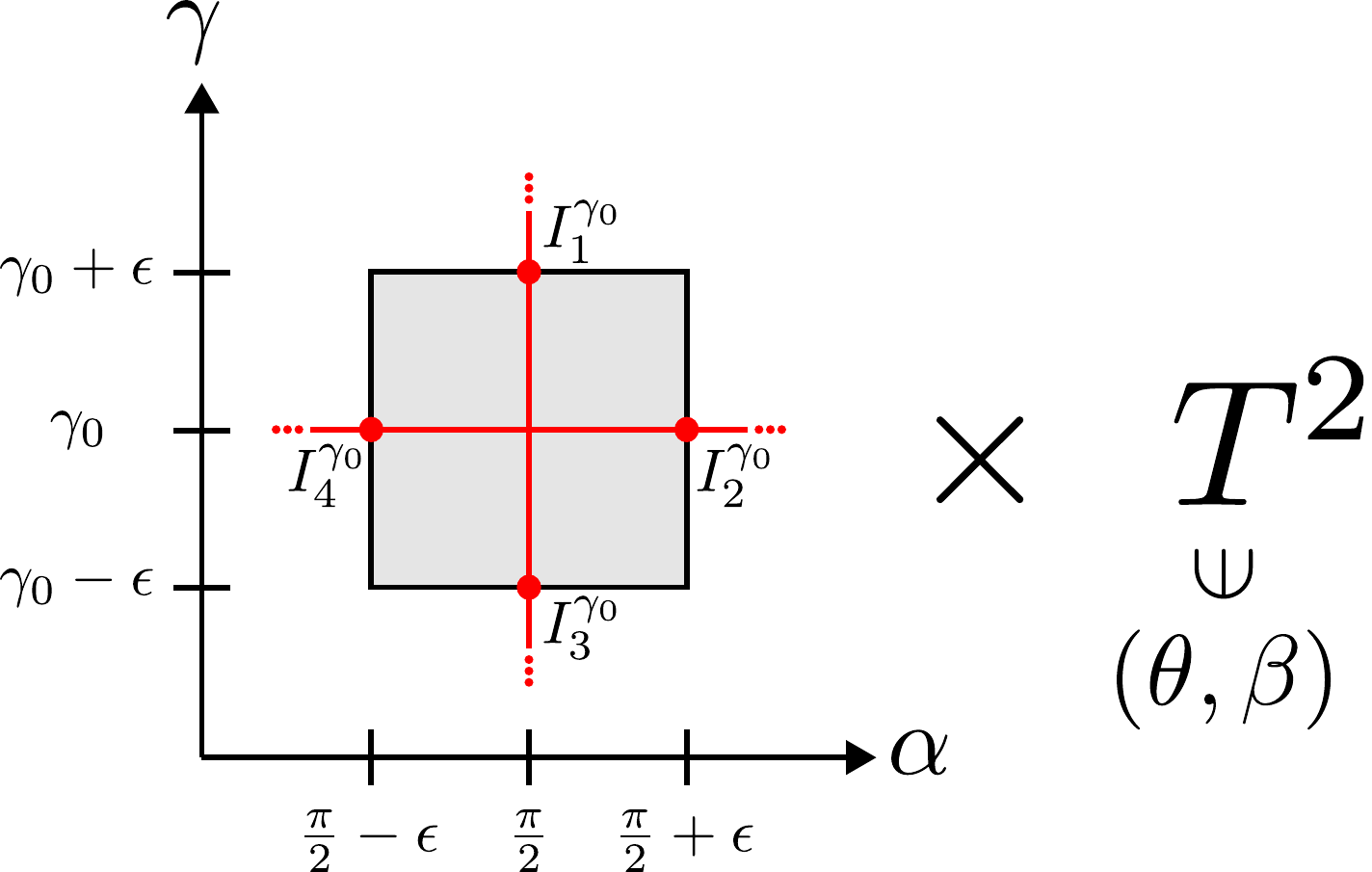}
    \caption{The gray box represents $\NS_\epsilon^{\gamma_0}$ and the red represents $\Sols$. The intersection of red lines in the middle of the square gives $\SV^{\gamma_0}$.}
    \label{fig:Nslice}
\end{figure}

Let $\SV^0:= \{(\widetilde{\Gamma}(0,\theta, \frac{\pi}{2}, \beta))\in\Pr\mid \theta,\beta\in S^1\}$ and $\SV^\pi:= \{(\widetilde{\Gamma}(\pi,\theta, \frac{\pi}{2}, \beta))\in\Pr\mid \theta,\beta\in S^1\}$. Thus $\SV = \SV^0 \sqcup \SV^\pi$ For $\gamma_0\in\{0,\pi\}$ define 
\[\NS_\epsilon^{\gamma_0} := \{\widetilde{\Gamma}(\gamma,\theta,\alpha,\beta) \in \Pr \mid \abs{\gamma-\gamma_0}\le \epsilon,~ \abs{\alpha-\frac{\pi}{2}}\le \epsilon\}\]
which is a closed neighborhood of $\SV^{\gamma_0}$. Let $\Psi^{\gamma_0}_\epsilon: \NS^{\gamma_0}_\epsilon \to \SV^{\gamma_0}$ be the projection given by $\widetilde{\Gamma}(\gamma,\theta,\alpha,\beta)\mapsto \widetilde{\Gamma}(\gamma_0,\theta,\frac{\pi}{2},\beta)$. Note that if $\mathbf{x}=\widetilde{\Gamma}(\gamma_0, \theta_0, \frac{\pi}{2}, \beta_0)\in\SV$, then 
\[(\Psi_\epsilon^{\gamma_0})\inv(\mathbf{x})=\{\widetilde{\Gamma}(\gamma,\theta,\alpha,\beta)\in\Pr\mid \abs{\gamma-\gamma_0}\le \epsilon,\theta=\theta_0, \abs{\alpha-\frac{\pi}{2}}\le \epsilon, \beta=\beta_0\}\cong I^2.\] So $\NS_\epsilon^{\gamma_0} \cong T^2 \times I^2 \cong T^2 \times D^2$ and $\partial \NS_\epsilon^{\gamma_0} \cong T^2\times \partial (I^2) \cong T^3$. $\partial\NS^{\gamma_0}_\epsilon$ naturally decomposes into four pieces (corresponding to the edges of the square in Figure \ref{fig:Nslice}):
\[(\partial\NS^{\gamma_0}_\epsilon)_1 := \{\widetilde{\Gamma}(\gamma,\theta,\alpha,\beta) \in \Pr \mid \gamma = \gamma_0 + \epsilon,~ \abs{\alpha-\frac{\pi}{2}}\le \epsilon\}\]
\[(\partial\NS^{\gamma_0}_\epsilon)_2 := \{\widetilde{\Gamma}(\gamma,\theta,\alpha,\beta) \in \Pr \mid \abs{\gamma-\gamma_0}\le \epsilon,~ \alpha = \frac{\pi}{2} + \epsilon\}\]
\[(\partial\NS^{\gamma_0}_\epsilon)_3 := \{\widetilde{\Gamma}(\gamma,\theta,\alpha,\beta) \in \Pr \mid \gamma = \gamma_0 - \epsilon,~ \abs{\alpha-\frac{\pi}{2}}\le \epsilon\}\]
\[(\partial\NS^{\gamma_0}_\epsilon)_4 := \{\widetilde{\Gamma}(\gamma,\theta,\alpha,\beta) \in \Pr \mid \abs{\gamma-\gamma_0}\le \epsilon,~ \alpha = \frac{\pi}{2} - \epsilon\}\]

$\Sols$ intersect each of these four pieces in a torus:

\[I^{\gamma_0}_1 := \Sols\cap(\partial\NS^{\gamma_0}_\epsilon)_1 = \{\widetilde{\Gamma}(\gamma,\theta,\alpha,\beta) \in \Pr \mid \gamma = \gamma_0 + \epsilon,~ \alpha=\frac{\pi}{2}\}\]
\[I^{\gamma_0}_2 := \Sols\cap(\partial\NS^{\gamma_0}_\epsilon)_2 = \{\widetilde{\Gamma}(\gamma,\theta,\alpha,\beta) \in \Pr \mid \gamma=\gamma_0,~ \alpha = \frac{\pi}{2} + \epsilon\}\]
\[I^{\gamma_0}_3 := \Sols\cap(\partial\NS^{\gamma_0}_\epsilon)_3 = \{\widetilde{\Gamma}(\gamma,\theta,\alpha,\beta) \in \Pr \mid \gamma = \gamma_0 - \epsilon,~ \alpha=\frac{\pi}{2}\}\]
\[I^{\gamma_0}_4 := \Sols\cap(\partial\NS^{\gamma_0}_\epsilon)_4 = \{\widetilde{\Gamma}(\gamma,\theta,\alpha,\beta) \in \Pr \mid \gamma=\gamma_0,~ \alpha = \frac{\pi}{2} - \epsilon\}\]

The following two propositions establish that this decomposition intersects the $\Sols$ nicely. This will be helpful when considering the perturbed character variety.%

\begin{proposition}
\label{prop:transversedecomp}
$\Sols$ intersects each $(\partial\NS^{\gamma_0}_\epsilon)_i$ transversely.
\end{proposition}

\begin{proof}
Clear. See Figure \ref{fig:Nslice}.
\end{proof}

Let $\NS_\epsilon:= \NS^0_\epsilon \sqcup \NS^\pi_\epsilon$ and $\Psi_\epsilon: \NS_\epsilon \to \SV$ be the corresponding projection.

\begin{proposition}
\label{prop:sqinttrans}
For each $\rho\in \SV$, $I^{\gamma(\rho)}_t$ intersects $\partial\Psi_\epsilon\inv(\rho)$ transversely in $(\partial\NS^{\gamma_0}_\epsilon)_i$.
\end{proposition}

\begin{proof}
Clear. See Figure \ref{fig:Nslice}.
\end{proof}

Let $\MS_\epsilon := \
\Pr\setminus \NS_\epsilon$. For $\epsilon>0$, $\Sols\cap\ol{\MS_\epsilon}$ has six disjoint components:

\[H_0^+ := \{\widetilde{\Gamma}(\gamma,\theta,\alpha,\beta) \in \Sols \mid \gamma = 0,~ \alpha \le \frac{\pi}{2}-\epsilon\}\subset\H\]
\[H_0^- := \{\widetilde{\Gamma}(\gamma,\theta,\alpha,\beta) \in \Sols \mid \gamma = 0,~ \alpha \ge \frac{\pi}{2}+\epsilon\}\subset\H\]
\[H_\pi^+ := \{\widetilde{\Gamma}(\gamma,\theta,\alpha,\beta) \in \Sols \mid \gamma = \pi,~ \alpha \le \frac{\pi}{2}-\epsilon\}\subset\H\]
\[H_\pi^- := \{\widetilde{\Gamma}(\gamma,\theta,\alpha,\beta) \in \Sols \mid \gamma = \pi,~ \alpha \ge \frac{\pi}{2}+\epsilon\}\subset\H\]

\[A^+ := \{\widetilde{\Gamma}(\gamma,\theta,\alpha,\beta) \in \Sols \mid \gamma \in[\epsilon, \pi-\epsilon], \alpha = \frac{\pi}{2} \}\subset\A\]
\[A^- := \{\widetilde{\Gamma}(\gamma,\theta,\alpha,\beta) \in \Sols \mid \gamma \in[\pi+\epsilon, 2\pi-\epsilon], \alpha = \frac{\pi}{2} \}\subset\A\]

Note that $\iota(H^\pm_0) = H^\mp_0$, $\iota(H^\pm_\pi) = H^\mp_\pi$, and $\iota(A^\pm) = A^\mp$.

\section{Perturbed Character Varieties}
\label{sec:Perturbing}

\subsection{Perturbed Traceless \texorpdfstring{$\SU$}{SU(2)} Character Variety of \texorpdfstring{$C_3$}{C3}}
\label{sec:pertC3}

Let $\Pert$ be the curve shown in red in Figure \ref{fig:C3Pert}. Let $\Pert_t$ be an abbreviation for the perturbation data $(\Pert,\sin, t)$. Note that $\RR_{\Pert_0}(C_3) = \RR(C_3)$, which was studied in Section \ref{sec:unpert}. The first part of this section is dedicated to studying the space $\RR_{\Pert_t}(C_3)$ for small, but nonzero, $t$. Then these results are used to understand the perturbed character variety of tangle sums. For a good pair $(T_1,\pi_1)$ and $(T_2, \pi_2)$, the character variety $\RR_{\t\cup \pi_1\cup \pi_2}(T_1+T_2)$ is studied.

\begin{figure}
    \centering
    \includegraphics[height=2.5in]{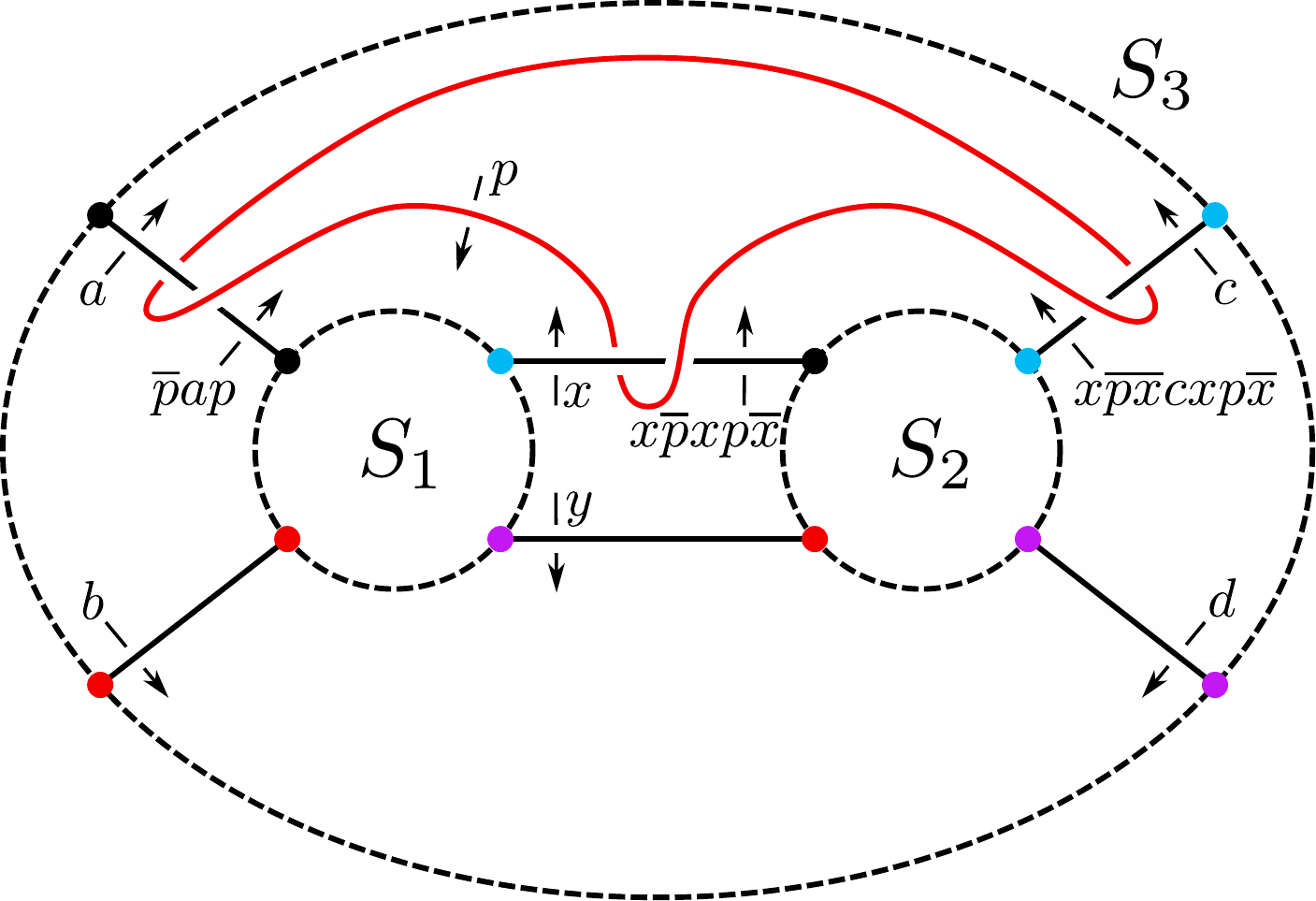}
    \caption{$C_3$ with perturbation curve $\Pert$ in red.}
    \label{fig:C3Pert}
\end{figure}

\subsubsection{Fundamental Group Calculation}

The fundamental group of $C_3$ with the perturbation curve $\Pert$ is calculated using the Wirtinger presentation of the diagram in Figure \ref{fig:C3Pert}. 

\begin{proposition}
With the notation $\lambda = a\ol{c}x$, a presentation of the fundamental group of the tangle complement

\[X_\Pert:=(S^3\setminus 3D^3)\setminus (T\cup \Pert)\]
is given as follows.

\begin{description}

\item[Generators] The set $\{a,b,c,d,x,y,p\}$ generates the group. %

\item[Relations] The following relations are complete.
\begin{enumerate}[{\upshape(I)}]
    \item $[p,\lambda] = 1$ (from $\Pert$)
    \item $\ol{a}b = \ol{c}d$ (from $\SThree$)
    \item $y = x\ol{p}\ol{a}pb$ (from $\SOne$)
    \item $x\ol{pxc}xp\ol{x}d = x\ol{px}p\ol{x}y$ (From $\STwo$)
\end{enumerate}
\end{description}
\end{proposition}

Wirtinger presentations always result in a redundant relation, and so we know that Relation (IV) can be recovered from (I), (II), and (III). %

Plugging in our specific perturbation curve $\Pert$ into Equation \ref{eq:pert}, we get the following:

\begin{equation}
    \label{eq:peq}
    \rho(p) = e^{t\Im(\rho(a\ol{c}x))}
\end{equation}

Whenever this condition is satisfied, so is Relation (I). For a fixed $t$, if the images of $a$, $b$, $c$, and $x$ under a representation $\rho\in\Ro_\t(X_\Pert)$ are fixed, then $\rho(d)$ is determined by (II), $\rho(p)$ is determined by Equation \ref{eq:peq}, and then $\rho(y)$ is determined by Relation (III). Thus, a representation $\rho\in\Ro_\t(X_\Pert)$ is determined by where it sends $a$, $b$, $c$, and $x$. Since there are no outstanding relations, $a$, $b$, $c$, and $x$ can be sent to any elements of $\SU$ and so $\Ro_\t(X_\Pert)\cong \SU^4$.

Thus, it is possible to define the maps
\[\Lambda:\Pr\to \Ro_\t(X_\Pert) \text{ and } \Gamma:\widetilde{\Pr}\to \Ro_\t(X_\Pert)\]
analogous to those defined in Section \ref{sec:C3Unpert}. Define \[\Sols_t := \Lambda\inv(\Ro_\t(C_3)).\]

By construction, any representation in the image of $\Lambda$ already sends $a$, $b$, $c$, $d$, and $x$ to traceless elements of $\SU$. Thus, a representation $\rho\in\Im(\Lambda)$ is in $\Ro_\t(C_3)$ iff $\rho(y)$ is traceless. Then we can define a function $\F_t: \Pr \to [-1,1]$ by
\begin{equation}
\label{eq:Fbp}
\F_t(\i,b,c,x) = \frac{1}{2}\tr(\rho(y)) = \Re(\rho(y)) \stackrel{\text{(III)}}{=} \Re(\rho(x\ol{pa}pb)) \text{ where } \rho = \Lambda(\i,b,c,x)
\end{equation}

\begin{remark}
$\F_t(\gamma,\theta,\alpha,\beta)$ will be used as shorthand for $\F_t(\widetilde{\Gamma}(\gamma,\theta,\alpha,\beta))$.
\end{remark}

With this function, we can identify $\Sols_t = \F\inv(0)$. Next, we expand out the expression for $\F_t(\gamma,\theta,\alpha,\beta)$.

\begin{remark}
In order to condense already lengthy calculations and reduce symbolic clutter, at times we will abuse notation and use the name of elements in $\pi_1(X_\Pert)$ to refer to their image under a given representation $\rho$.
\end{remark}

In order to calculate the image of $p$ under a particular representation $\rho = \Gamma(\gamma, \theta,\alpha,\beta)$ using Equation \ref{eq:peq}, we first need to calculate $\Im(\lambda)$:
\begin{align*}
\Im(\lambda) &= \Im(a\ol{c}x)\\
& = \Im(-e^{-\theta \k}x)\\
& =\Im(-\cos\theta x + \sin\theta \k x)\\
& =-\cos\theta x + \sin\theta \Im(\k x) \\
&=-\cos\theta(\sin\alpha\cos\beta\i+\sin\alpha\sin\beta\j+\cos\alpha\k) + \sin \theta(-\sin\alpha\sin\beta\i+\sin\alpha\cos\beta\j) \\
&=(-\sin\alpha\cos\beta\cos\theta-\sin\alpha\sin\beta\sin\theta)\i +(\sin\alpha\cos\beta\sin\theta - \sin\alpha\sin\beta\cos\theta)\j\\ &\phantom{==} - \cos\alpha\cos\theta \k \\
&=-\sin\alpha\cos(\theta-\beta)\i+\sin\alpha\sin(\theta-\beta)\j -\cos\alpha\cos\theta\k
\end{align*}

Recall that for $Q\in C(\i)$ and $\nu\in\R$, $e^{\nu Q} =\cos\nu+Q\sin\nu$. Thus, in order to calculate $p = e^{t\Im(\lambda)}$, we must rewrite it so the exponent is a scalar times a unit traceless quaternion. Letting \[n:=\abs{t\Im(\lambda)}\], we can rewrite our equation as $p = e^{n\frac{t\Im(\lambda)}{n}}$ where $\frac{t\Im(\lambda)}{n}\in C(\i)$ and $n\in\R$ as desired. Then we can calculate $n$ for our representation $\rho = \Gamma(\gamma, \theta,\alpha,\beta)$:

\begin{align}
\label{eq:n}
   n:=  \abs{t\Im(\lambda)}
&=t\sqrt{\sin^2\alpha\cos^2(\theta-\beta)+\sin^2\alpha\sin^2(\theta-\beta)+\cos^2\alpha\cos^2\theta} \\
&= t\sqrt{(1-\cos^2\alpha)+\cos^2\alpha\cos^2\theta} \nonumber\\
&= t\sqrt{1+\cos^2\alpha(\cos^2\theta-1)} \nonumber\\
&= t\sqrt{1-\cos^2\alpha\sin^2\theta} \nonumber
\end{align}

\noindent
Note that if $t>0$, then $0\le n\le t$. Putting these calculations together, we can find the image of $p$ under $\rho = \Gamma(\gamma, \theta,\alpha,\beta)$.

\[p = \cos n+\sinc(n)t[-\sin\alpha\cos(\theta-\beta)\i+\sin\alpha\sin(\theta-\beta)\j -\cos\alpha\cos\theta\k]\]

\noindent where \[\sinc(n) = \begin{cases}\frac{\sin n}{n}&n\ne 0\\ 1&n=0\end{cases}\]
From this, we can calculate
\begin{align}
\label{eq:pap}
\ol{pa}p =& [-\sinc(n)t\sin\alpha\cos(\beta-\theta)-\cos n\i+\sinc(n)t\cos\alpha\cos\theta\j-\sinc(n)t\sin\alpha\sin(\theta-\beta)\k]p \\\nonumber
=& [\sinc^2(n)t^2\sin^2\alpha\cos^2(\beta-\theta)-\cos^2 n]\i+\\\nonumber
&2[\cos n\sinc(n)t\cos\alpha\cos\theta-\sinc^2(n)t^2\sin^2\alpha\sin(\theta-\beta)\cos(\beta-\theta)]\j+\\\nonumber
&2[-\cos n\sinc(n)t\sin\alpha\sin(\theta-\beta)+\sinc^2(n)t^2\sin\alpha\cos\alpha\cos\theta\cos(\beta-\theta)]\k.
\end{align}
Plugging this into Equation \ref{eq:Fbp} gives
\begin{align}
\label{eq:F}
\F_t(\gamma,\theta,\alpha,\beta) =& \cos\gamma\left[2\cos n\sinc(n)t\left[\sin^2\alpha\sin\beta\sin(\theta-\beta) - \cos^2\alpha\cos\theta\right] + \right. \\\nonumber
&\phantom{\cos\gamma[[}\left. 2\sinc^2(n)t^2\sin^2\alpha\cos\beta\cos\alpha\sin\theta\cos(\beta-\theta)\right]+\\\nonumber
&\sin\gamma\left[\cos\alpha\cos^2 n-2\cos n\sinc(n)t\sin^2\alpha\cos\beta\sin(\theta-\beta)+ \right. \\\nonumber 
&\phantom{\sin\gamma[[}\left. \sinc^2(n)t^2\cos\alpha\sin^2\alpha\cos(\beta-\theta)\cos(\beta+\theta)\right]
\end{align}
We can break this complicated expression up into more manageable pieces by defining:
\begin{align}
F(t,\theta,\alpha,\beta) =& 2\cos n\sinc(n)t\left[\sin^2\alpha\sin\beta\sin(\theta-\beta) - \cos^2\alpha\cos\theta\right] + \\\nonumber
& 2\sinc^2(n)t^2\sin^2\alpha\cos\beta\cos\alpha\sin\theta\cos(\beta-\theta)\nonumber
\end{align}
\begin{align}
G(t,\theta,\alpha,\beta) =&\cos\alpha\cos^2 n-2\cos n\sinc(n)t\sin^2\alpha\cos\beta\sin(\theta-\beta)+\\\nonumber 
&\sinc^2(n)t^2\cos\alpha\sin^2\alpha\cos(\beta-\theta)\cos(\beta+\theta)\nonumber
\end{align}
so that now \[\F_t(\gamma,\theta,\alpha,\beta) = \cos\gamma F(t,\theta,\alpha,\beta)+ \sin\gamma G(t,\theta,\alpha,\beta).\]

\begin{lemma}
    \label{lem:function}
    For fixed choices of $t_0$, $\theta_0$, $\alpha_0$, and $\beta_0$, either:
    
    \begin{itemize}
        \item There are two antipodal values of $\gamma\in S^1$ such that $\F_{t_0}(\gamma,\theta_0,\alpha_0, \beta_0)=0$. In this case, the values of $\gamma$ depend smoothly on the values of $\theta$, $\alpha$, and $\beta$.
        \item For any $\gamma\in S^1$, $\F_{t_0}(\gamma,\theta_0,\alpha_0, \beta_0)=0$.
    \end{itemize}
\end{lemma}

\begin{proof}
If $F(t_0,\theta_0,\alpha_0,\beta_0)=G(t_0,\theta_0,\alpha_0,\beta_0)=0$, then $\F_{t_0}(\gamma,\theta_0,\alpha_0,\beta_0)=0$ for any value of $\gamma$. Otherwise, exactly two antipodal values of $\gamma$ can be found by $\arctan(-\frac{F(t_0,\theta_0,\alpha_0,\beta_0)}{G(t_0,\theta_0,\alpha_0,\beta_0)})$  which satisfy $\F(\gamma,\theta_0,\alpha_0,\beta_0)=0$.
\end{proof}

\subsubsection{Image in the Pillowcase}

Define the map $h_t:\Sols_t\to \RR_\t(C_3)$ by $x\mapsto [\Lambda(x)]$. By the same proof as in Proposition \ref{prop:hfibers}, $h_t$ is a surjective map. %
Let $S_1$, $S_2$, and $S_3$ be the components of $\partial(C_3)$ as in Figure \ref{fig:C3Pert}. Let $P_i=\RR(S_i)$ be the respective pillowcases. $\RR_\t(C_3)$ maps into $\P:=P_1 \times P_2 \times P_3$ by restricting the representation to the boundary $S_1 \cup S_2\cup S_3$. Let $\widetilde{p}:\RR_\t(C_3)\to \P$ be this map and let $p_i:\RR_\t(C_3)\to P_i$ be $\widetilde{p}$ composed with the projection to $P_i$. 

Each pillowcase has a 0 dimensional stratum (the corners) and a 2 dimensional stratum. Let $P_{i,0}$ and $P_{i,2}$ denote the 0 and 2 dimensional strata of $P_i$ respectively. Then let $\P_{ijk} := P_{1,i}\times P_{2,j}\times P_{3,k}$.
So \[\P = \P_{222}\sqcup \P_{220}\sqcup \P_{202}\sqcup \P_{022}\sqcup \P_{200}\sqcup \P_{020}\sqcup \P_{002}\sqcup \P_{000}.\]

\begin{proposition}
\label{prop:strat}
For $t\in (0,\frac{\pi}{2})$, the image of $\widetilde{p}: \RR_\t(C_3) \to \P$ lies entirely in $\P_{222}$ and $\P_{000}$
\end{proposition}

\begin{lemma}
\label{lem:bdh}
If $\rho\in\RR_\t(C_3)$ for $t\in(0,\frac{\pi}{2})$ for which $\{a,b,c,x,\ol{pa}p\}$ is a coequatorial set (see Definition \ref{def:coeq}), then $x = \pm c$.
\end{lemma}

\begin{proof}
Put $\rho$ into standard gauge with respect to $\bmu = (a,b,c,x)$. Thus, if $\{a,b,c,x,\ol{pa}p\}$ is a coequatorial set, $x$ must be in $\Sk$, so $x=\ol{\Gamma}(\frac{\pi}{2},\beta) = e^{\beta\k}\i$ for some $\beta$. Plugging this into Equation \ref{eq:pap}, we get that the dot product $(\ol{pa}p)\cdot\k = -\cos(t)\sinc(t)t\sin(\theta-\beta)$. Since $\ol{pa}p$ must also lie in $\Sk$, this expression must equal 0, which occurs if and only if $\beta\in\{\theta, \theta+\pi\}$ and thus $x = \pm c$.
\end{proof}

\begin{proof}[Proof of Proposition \ref{prop:strat}]
Let $\rho\in \RR_\t(C_3)$. Because $h_t$ is surjective, there exist $\gamma$, $\theta$, $\alpha$, $\beta \in S^1$ such that $\rho = h_t(\gamma,\theta,\alpha,\beta)$. We will show that if $\rho$ is in $\P_{0**}$, $\P_{*0*}$, or $\P_{**0}$, then it is actually in $\P_{000}$.

\noindent\textbf{Case 1:} Assume $\widetilde{p}(\rho)\in \P_{**0}$. Thus 
\begin{equation}
\label{eq:cornercase1}
a=\pm b=\pm c=\pm d.
\end{equation}
In standard gauge with respect to $\bmu = (a,b,c,x)$, $a=\i$, and so $b,c,d\in\{\pm \i\}$. Thus $\theta, \gamma\in\{0,\pi\}$. Note that since $\sin\theta=0$, by Equation \ref{eq:n}, $n=t$ and so 
\[\F_t(\gamma,\theta,\alpha,\beta) = \pm 2\cos(t)\sin(t)(\sin^2\alpha\sin^2\beta+\cos^2\alpha).\] 
This can only be 0 if $\cos\alpha=\sin\beta=0$ and thus $x= \ol{\Gamma}(\alpha,\beta) = \pm \i$. Thus $\lambda = a\ol{c}x = \pm \i$ and so $p=e^{t\i}$ and $\ol{p}ap = \i$. By Relation (II), $y = x\ol{pa}pc = \pm \i$. Thus $b=\pm x = \pm y = \pm \ol{p}ap$ and so $\widetilde{p}(\rho)\in \P_{0*0}$. Then it is easy to check that 
\[x\ol{p}xp\ol{x} = \pm\i \text{ and } x\ol{px}cxp\ol{x} = \pm\i\] 
and thus $\widetilde{p}(\rho) \in \P_{000}$.

\noindent\textbf{Case 2:} Assume $\widetilde{p}(\rho)\in \P_{0**}$. Thus 
\begin{equation}
\label{eq:cornercase2}
x = \pm\ol{p}ap = \pm b = \pm y.
\end{equation}
 In standard gauge with respect to $\bmu = (a,b,c,x)$, $b\in\Sk$, so $x$ and $\ol{p}ap$ are in $\Sk$. Then by Lemma \ref{lem:bdh}, $x=\pm c$. Thus $\lambda = a\ol{c}x = \pm a$ and so $p = e^{t\i}$. It follows that $\ol{p}ap = \i$ and so furthermore $a,b,c,d,x,y\in \{\pm \i\}$ and so $\widetilde{p}(\rho)\in \P_{0*0}$. By Case 1, $\widetilde{p}(\rho)\in \P_{000}$.

\noindent\textbf{Case 3:} Assume $\widetilde{p}(\rho)\in\P_{*0*}$. Thus 
\begin{equation}
\label{eq:cornercase3}
y = \pm d = \pm  x\ol{p}xp\ol{x} = \pm x\ol{px}cxp\ol{x}.
\end{equation}
In standard gauge with respect to $\bmu = (a,b,c,x)$, $d\in \Sk$ and so $y\in \Sk$. Recall that by Lemma \ref{lem:boundcond} $y$ must lie on the same great circle as $x$ and $b$. First assume that $x\notin\Sk$. This means that $y=\pm b$ and therefore $x\ol{pa}p=\pm 1$ by Relation (III). By the invariance of trace under cyclic permutations, $px\ol{pa} = \pm1$ and so $px\ol{p} = \pm \i$. By Equation \ref{eq:cornercase3}, $\pm x\ol{p}xp\ol{x} = \pm y\in \Sk$, and so $\pm x\i\ol{x}\in\Sk$. Since we are assuming that $x\notin \Sk$, this must mean that $x = \pm \k$ and therefore $x\ol{p}xp\ol{x} = \pm \i$.
This in turn implies that $y,b,d\in \{\pm \i\}$. Relation (II) then gives us that $c =\pm \i$ as well, thus we have $\widetilde{p}(\rho)\in \P_{0*0}$ and so by Case 1, we see that $x=\pm \i$ which contradicts our assumption that $x\notin\Sk$.

So $x$ must lie in $\Sk$. By applying Lemma \ref{lem:boundcond} to $S_1$ we know that $y$ lies on the same great circle as $x$ and $b$, and so $y\in\Sk$. Then by Relation (III), $\ol{p}ap\in \Sk$. Then by Lemma \ref{lem:bdh}, $x = \pm c$ and so $p = e^{t\i}$ and $\ol{p}ap = \i$. A calculation shows that for $d=\pm x\ol{p}xp\ol{x} = \pm c\ol{p}cp\ol{c}$ to lie in $\Sk$ and a direct computation shows that for this to be true for $t\in(0,\frac{\pi}{2})$, we must have $\theta\in\{0,\pi\}$ and so $c = x=\pm \i$. 
Then from Equation \ref{eq:cornercase3}, $y = d = \pm x\ol{p}xp\ol{x} = \pm \i$. By Relation (III) $b = \i$. Thus we have $\widetilde{p}(\rho)\in \P_{000}$.
\end{proof}

\begin{definition}
Let $\Sols_t^0 := h_t\inv(\widetilde{p}\inv(\P_{000})) = \{(\i,b,c,x)\in\Pr\mid b,c,x\in\{\pm\i\}\}$ be the \emph{corner points} of $\Sols_t$. This set does not depend on $t$ (including $t=0$), so we will instead simply denote if by $\Sols^0$.
Note that $h_t(\Sols^0) = \RR_\t(C_3)^\U$.
\end{definition}

\begin{proposition}
For $t\in(0,\frac{\pi}{2})$, the fiber of $h_t$ over $[\rho]\in \RR_\t(C_3)$ is homeomorphic to
\[\begin{cases}
\pt &\text{if } [\rho]\in\RR_\t(C_3)^\U\\
S^0 &\text{if } [\rho]\in\RR_\t(C_3)^{\Z/2}
\end{cases}\]
\end{proposition}

\begin{proof}
The same argument as Proposition \ref{prop:hfibers} applies, however we will see that the case in Proposition \ref{prop:hfibers} where the fiber is $S^1$ never occurs for $t\in(0,\frac{\pi}{2})$. The fiber of $h_t$ over $[\rho]$ is a circle if $\rho(b),\rho(c)\in\{\pm\rho(a)\}$ but $\rho(x)\not\in\{\pm\rho(a)\}$. Because $h_t$ is surjective, if there exists $[\rho]\in\RR_\t(C_3)$ satisfying this condition, there must exist $(\i,b,c,x)\in\Sols_t$ such that $b,c\in\{\pm\i\}$ but $x\not\in\{\pm\i\}$. But then $\widetilde{p}$ would map such a point to $\P_{2*0}$ and Proposition \ref{prop:strat} shows that such a point does not exist in $\Sols_t$. Thus, the only cases that occur from Proposition \ref{prop:hfibers} are those with fibers are $\pt$ or $S^0$.
\end{proof}

\begin{proposition}
\label{prop:dbranched}
For $t\in(0,\pi)$ the map $h_t$ is a double cover branched over $\RR_\t(C_3)^\U$.
\end{proposition}

\begin{proof}
The same proof as for Proposition \ref{prop:restdcov} shows that $h_t$ is a double cover over $\RR_\t(C_3)^{\Z/2}$. Because $\RR_\t(C_3)^\U$ only has eight points, the map is a branched cover. 
\end{proof}

Whenever the fiber is $S^0$, $\iota$ swaps the two points and whenever the fiber is a point, it is fixed by $\iota$. Thus, \[\Sols_t/\iota\cong \RR_\t(C_3).\]

\subsubsection{Smoothness}

At $\Sols^0$, all the partial derivatives of $\F_t$ vanish for $t>0$. Theorem \ref{thm:Cone} will show that neighborhoods of these points in $\Sols_t$ are cones on tori. However, away from these corner points, the perturbation smooths out the character variety:

\begin{theorem}
\label{thm:smooth}
$\Sols_t\setminus \Sols^0$ is a smooth manifold for sufficiently small $t\ne 0$. %
\end{theorem}

\begin{proof}

Because $\Sols_t\setminus \Sols^0$ is a subspace of the zero set of $\Phi_t$, to show that $\Sols_t\setminus \Sols^0$ is a smooth manifold, we need to show that at every point in $\Sols_t\setminus \Sols^0$, there is some partial derivative of $\F_t$ that does not vanish. That is, all derivatives vanish at $\rho$ if and only if $\rho$ is one of the eight corner point.

The easiest derivative to start with is $\gamma$. \[\frac{\partial\F_t}{\partial\gamma}=\Re(x\ol{pa}pb(-\k))\] Since $b\in\Sk$, $b\k\in \Sk$ and further $b\ne b\k$. In order for $\Re(x\ol{pa}pb)=\Re(x\ol{pa}pb(-\k))=0$, $x\ol{pa}p$ must be perpendicular to $\Sk$, meaning $x\ol{pa}p = e^{\nu \k}$. Then \[\Re(\ol{pa}p) = \Re(\ol{x}e^{\nu \k}) = -\sin(\nu) \Re(x\k).\] However, $\ol{pa}p\in C(\i)$ so $\Re(\ol{pa}p) = 0$. Thus any solution which has vanishing $\gamma$ derivative satisfies $\sin(\nu) \Re(x\k)=0$ and thus $x \in \Sk$ or $\sin(\nu)=0$, meaning that $x\ol{pa}p=\pm 1$.

\bigskip

\noindent\textbf{Case 1}: $x\in \Sk$.

In this case $x = \ol{\Gamma}(\frac{\pi}{2},\beta)$ for some $\beta$. Since $\alpha = \frac{\pi}{2}$, $n=t$ and so $\sinc(n)t = \sin t$.

\[%
\F_t(\gamma,\theta,\frac{\pi}{2},\beta)
=2\cos t\sin t\sin(\theta-\beta)\sin(\beta-\gamma)
\]%

So any zero of $\F_t$ with $\alpha = \pi/2$ satisfies either $\sin(\theta-\beta)=0$ or $\sin(\beta-\gamma)=0$. However,

\[%
\frac{\partial\F_t}{\partial \beta}|_{\alpha=\frac{\pi}{2}} = 2\cos t\sin t\sin(\theta+\gamma-2\beta) = 2\cos t\sin t[\sin(\theta-\beta)\cos(\gamma-\beta)+\sin(\gamma-\beta)\cos(\theta-\beta)]
\]%

From this, it can be seen that if one of $\sin(\theta-\beta)$ or $\sin(\gamma-\beta)$ is 0, then in order for $\frac{\partial\F_t}{\partial\beta}|_{\alpha=\frac{\pi}{2}}$ to be 0, in fact both $\sin(\theta-\beta) = \sin(\gamma-\beta) = 0$. This implies $b=\pm c=\pm x$. Finally,

\begin{equation}
\label{eq:da}
\frac{\partial\F_t}{\partial\alpha}|_{\alpha=\frac{\pi}{2}} =-\sin\gamma\left[\cos^2t+\sin^2t[2\sin\beta\cos\beta\cos\theta\sin\theta-\cos^2\beta-\sin^2\beta\cos(2\theta)]\right]
\end{equation}

Note that $\cos^2t+\sin^2t[2\sin\beta\cos\beta\cos\theta\sin\theta-\cos^2\beta-\sin^2\beta\cos(2\theta)] > \cos^2t-2\sin^2t$ which is positive if $t<2\arctan(\sqrt{5-2\sqrt{6}})$. So for such a $t$, $\frac{\partial\F_t}{\partial\alpha}|_{\alpha=\frac{\pi}{2}}=0$ if and only if $\sin\gamma = 0$ and thus $\gamma, \theta, \beta \in \{0,\pi\}$ and so any such representation must lie in $\Sols^0$.

\bigskip

\noindent\textbf{Case 2:} $x\ol{pa}p=\pm1$. Furthermore assume that $x\notin \Sk$, since this is covered by Case 1.

First note that $x = \pm \overline{pa}p$. For a fixed $\gamma$ and $t$, define the map $G_{\gamma, t}:S^2 \to S^2$ given by \[G_{\gamma,t}(x) = px\overline{p},\] recalling that $p$ is a function of $\gamma$, $x$, and $t$. Note that $G_{\gamma,0}$ is the identity map and hence a diffeomorphism. Since being a diffeomorphism is an open condition, $G_{\gamma,t}$ is also a diffeomorphism for sufficiently small $t$.

Thus for a fixed $t$ and $\gamma$, we can define $H_{\gamma, t} := G_{\gamma, t}\inv$. So if $x = \pm H_{\gamma, t}(\i)$ then $x = \pm\ol{pa}p$. Now take a path $[-1,1]\to S^2$ given by $s\mapsto e^{s\k}\i$ and let $x_s:= H_{\gamma,t}(e^{s\k}\i)$ and $p_s := e^{t\pi(abx_s)}$. Note that by definition we have $p_sx_s\overline{p_s} = e^{s\k}\i$ and so $x_s = \overline{p_s}e^{s\k}\i p_s$.
Now we can substitute this in to get 
\begin{align}
\Re(x_s\overline{p_sa}p_sc) &= -\Re(\overline{p_s}e^{s\k}\i ap_sc) = \Re(\overline{p_s}e^{s\k}p_sc) \nonumber\\
&=\Re(\overline{p_s}(\cos(s)+\sin(s)\k)p_sc)\nonumber\\
&= 
\cos(s)\Re(\overline{p_s}p_sc)+\sin(s)\Re(\overline{p_s}\k p_sc)\nonumber\\
&=\sin(s)\Re(\overline{p_s}\k p_sc).\nonumber
\end{align}
Taking the derivative in the $s$ direction when $s=0$ gives 
\begin{align}
\frac{\partial}{\partial s}\mid_{s=0}\F(\gamma, \theta, x_s) &= \frac{\partial}{\partial s}|_{s=0}[\sin(s)\Re(\overline{p_s}\k p_sc)] \nonumber\\
&= \cos(0)\Re(\overline{p}\k pc) + \sin(0)\frac{\partial}{\partial s}|_{s=0}[\Re(\overline{p_s}\k p_sc)] \nonumber
\\&= \Re(\overline{p}\k pc).\nonumber
\end{align}

Now consider the path given by $r\mapsto e^{r\j}\i$. In the same way, our equation $\Re(x_r\overline{p_ra}p_r)$ simplifies to $\sin(r)\Re(\overline{p_r}\j p_rc)$ and so \[\frac{\partial}{\partial r}\mid_{r=0}\F(\gamma, \theta, x_r) = \Re(\overline{p}\j pc).\] Both the $s$ and $r$ derivatives vanish if and only if $c$ is perpendicular to both $\overline{p}\k p$ and $\overline{p}\j p$, which only happens when $c = \pm \overline{p}\i p = \pm x$. This means that $x\in\Sk$, which contradicts our assumption.

\end{proof}

\begin{definition}
Based on the previous theorem, define $\Sols^*_t:= \Sols_t\setminus \Sols^0$ to be the smooth stratum of $\Sols$. Define $\RR_\t(C_3)^*:= h_t(\Sols^*_t)$. 
\end{definition}

As a consequence of Proposition \ref{prop:dbranched}, $\Sols^*_t$ is a double cover of $\RR_\t(C_3)^*$ with \[\Sols^*_t/\iota\cong \RR_\t(C_3)^*.\] The cover is no longer branched, since the branch points are precisely the points ignored when considering the smooth strata.

\begin{proposition}
$\Sols^*_t$ is an orientable manifold and $\iota$ is orientation preserving.
\end{proposition}

\begin{proof}

The proof of Theorem \ref{thm:smooth} shows that 0 is a regular value of the map $\F_t|_{\Sols_t^*}$. Thus, the normal bundle of $\{0\}$ in $\R$ can be pulled back, as well as a nowhere zero section of the normal bundle, showing that $\Sols_t^*$ is orientable. Orient $\Sols_t^*$ by pulling back the section $0\mapsto 1$.

Recall that the involution $\iota$ acts by conjugating by $\i$. Therefore, for $w\in \Pr$, 
\begin{equation}
\label{eq:iotainv}
    \F_t(w) = \tr(\Lambda(w)(y)) = \tr(\iota(\Lambda(w)(y))) = \tr(\Lambda(\iota(w))(y)) = \F_t(\iota(w))
\end{equation}
because traces are conjugation invariant.
Pick a point $w\in\Sols_t^*$ and consider a curve $r(s):[-1,1]\to \Pr$ such that $r(0) = w$ and $r'(0)$ agrees with the chosen orientation at $w$. In particular, this means that $\frac{d}{ds}[\F_t(r(s))]_{s=0} > 0$. Then consider the curve $\iota(r)$. By Equation \ref{eq:iotainv}, $\frac{d}{ds}[\F_t(\iota(r(s)))]_{s=0} > 0$ and so $\frac{d}{ds}[\iota(r(s))]_{s=0}$ is a normal vector at $\iota(w)$ which agrees with the chosen orientation. Therefore, $\iota$ is orientation-preserving.
\end{proof}

\begin{corollary}
\label{cor:smoothds}
$\RR_\t(C_3)^*$ is a smooth, orientable 3-manifold for sufficiently small perturbation $t\ne 0$. Furthermore, the map $\widetilde{p}:\RR_\t(C_3)^* \to \P_{222} = P_1^*\times P_2^*\times P_3^*$ is a Lagrangian immersion.
\end{corollary}

\begin{proof}
Because $\RR_\t(C_3)^*$ has a smooth, orientable double cover, $\Sols^*_t$, we know that $\RR_\t(C_3)^*$ must be a smooth manifold. Because $\iota$ is orientation-preserving, $\RR_\t(C_3)^*$ is orientable. The fact that $\widetilde{p}$ is a Lagrangian immersion follows from \cite{CHK}.
\end{proof}

\subsubsection{The Structure of $\Sols_t$}

Recall the decomposition of $\Pr$ given in Section \ref{sec:decomp}. In this section, we will see how the pieces of this decomposition intersect $\Sols_t$.
By Proposition \ref{prop:transversedecomp}, $\Sols_0$ intersects each $(\partial \NS_\epsilon^{\gamma_0})_i$ transversely in the torus $I^{\gamma_0}_i$ and thus there is a $t_\epsilon>0$ such that $\Sols_t$ intersects each $(\partial\NS^{\gamma_0}_\epsilon)_i$ transversely in a torus $(I^{\gamma_0}_i)_t$ for each $t\in[0,t_\epsilon)$. For each $\rho\in\SV$, each $I^{\gamma_0}_i$ intersects $\partial (\Psi_\epsilon\inv(\rho))$ transversely by Proposition \ref{prop:sqinttrans} and so for some $t_{\rho}\in(0,t_\epsilon)$, $(I^{\gamma_0}_i)_t$ intersects each $\partial (\Psi_\epsilon\inv(\rho))$ transversely for each $t\in[0,t_\rho)$. Because $\SV\cong S^0\times T^2$ is compact, there exists $t_0 = \min\{t_\rho\mid \rho\in\SV\}$ such that for every $t\in[0,t_0)$, $\Sols_t\cap (\partial \NS_\epsilon^{\gamma_0})$ intersects each $\partial(\Psi\inv_\epsilon(\rho))$ transversely in four points. For a fixed $\rho\in\SV$, parameterize $\Psi\inv_\epsilon(\rho)$ by $(\alpha, \gamma)$ as in Figure \ref{fig:Nslice} and define these four intersection points as follows:
\begin{align*}
p_t^1 &:= (\alpha^+_t,\epsilon) \in (I^{\gamma_0}_1)_t\\
p_t^2 &:= (\frac{\pi}{2}+\epsilon,\gamma^+_t) \in (I^{\gamma_0}_2)_t\\
p_t^3 &:= (\alpha^-_t,-\epsilon) \in (I^{\gamma_0}_3)_t\\
p_t^4 &:= (\frac{\pi}{2}-\epsilon,\gamma^-_t) \in (I^{\gamma_0}_4)_t
\end{align*}

\noindent with $\alpha^+_0 = \alpha^-_0 = \frac{\pi}{2}$ and $\gamma^+_0 = \gamma^-_0 = \gamma_0$. Note $p_t^i\in (\partial\NS_\epsilon^{\gamma_0})_i$.

Recall that $\NS_\epsilon$ was defined as a particular neighborhood of $\SV\in \Pr$ and $\MS_\epsilon = \Pr\setminus \NS_\epsilon$. The goal of this next lemma is to show that away from $\SV$, as $t$ increases from 0, $\Sols_t$ changes in as nice of a way as one could hope, by a smooth isotopy.

\begin{lemma}
\label{lem:isotopy}
For small $t$ and $\epsilon>0$, $\Sols_t\cap\ol{\MS_\epsilon}$ is smoothly isotopic to $\Sols_0\cap\ol{\MS_\epsilon}$.
\end{lemma}

\begin{proof}

Consider the map $\F:\R\times \ol{\MS_\epsilon}\to \R$ so that $\F(t,\mathbf{x}) = \F_t(\mathbf{x})$. Let $g:\R\times\ol{\MS_\epsilon}\to \R$ be the projection and let $f:\F\inv(0)\to \R$ be the restriction of $g$. Note that $\Sols_t\cap \ol{\MS_\epsilon} = f\inv(t)$. First, we will show that 0 is not a critical value of $f$.

\textit{Claim}: For $t_0\in\R$ and $\mathbf{x}\in\Sols_{t_0}\cap\ol{\MS_\epsilon}$, $\mathbf{x}$ is not a critical point of $f$ if $\nabla \F_{t_0}(\mathbf{x}) \ne 0$. There are two cases, depending on the value of $\frac{\partial \F}{\partial t}(t_0,\mathbf{x})$. 

\noindent\textbf{Case 1:} Suppose $\frac{\partial \F}{\partial t}(t_0,\mathbf{x}) = 0$.

Take a vector $v$ orthogonal to $\nabla \F_{t_0}(\mathbf{x})$ in $T_\mathbf{x}\Pr$. $v$ a tangent vector of $f\inv(t_0) = \Sols_{t_0}\cap \ol{\MS_\epsilon}$. Then because $\frac{\partial \F}{\partial t}(t_0,\mathbf{x}) = 0$, $(1,v)$ is a tangent vector to $\F\inv(0)$. Thus $D_\mathbf{x} f(1, v) = 1$ and so $D_\mathbf{x} f:T_\mathbf{x}\Pr\to \R$ is nonzero and thus $\mathbf{x}$ is not a critical point of $f$.

\noindent\textbf{Case 2:} Suppose $\frac{\partial \F}{\partial t}(t_0,\mathbf{x}) \ne 0$.

Let $n:=D_{(t_0, \mathbf{x})}\F(0,\nabla\F_{t_0}(\mathbf{x})) = \lVert\nabla\F_{t_0}(\mathbf{x})\rVert \ne 0$ and $m:=D_{(t_0, \mathbf{x})}\F(1,\mathbf{0}) \ne 0$. Then \[D_{(t_0, \mathbf{x})}(-\frac{n}{m}, \nabla\F_{t_0}(\mathbf{x})) = 0\] and so $(-\frac{n}{m}, \nabla\F_{t_0}(\mathbf{x}))$ is a tangent vector of $\F\inv(0)$ at $\mathbf{x}$. And clearly $D_{\mathbf{x}}f(-\frac{n}{m}, \nabla\F_{t_0}(\mathbf{x})) = -\frac{n}{m}\ne 0$ and so $\mathbf{x}$ is not a critical point.
This concludes the claim.

\bigskip

The proof of Theorem \ref{thm:smooth} shows that for any $\mathbf{x}\in \H\cap\ol{\MS}_\epsilon$, $\frac{\partial \F_0}{\partial \gamma}(\mathbf{x}) \ne 0$ and so $\mathbf{x}$ is not a critical point of $f$.
Now assume $\mathbf{x}\in \A\cap\ol{\MS}_\epsilon$. Thus, $\mathbf{x} = \widetilde{\Gamma}(\gamma, \theta, \frac{\pi}{2}, \beta)$ for $\sin\gamma\ne 0$.  Plugging $t=0$ into Equation \ref{eq:da} gives $\frac{\partial\F_0}{\partial\alpha}|_{\alpha = \frac{\pi}{2}} = -\sin\gamma \ne 0$. Thus, $\mathbf{x}$ is not a critical point of $f$.
Since $\Sols_0\cap \ol{\MS}_\epsilon = (\H\cap\ol{\MS}_\epsilon)\cup(\A\cap\ol{\MS}_\epsilon)$, no point in $\Sols_0\cap \ol{\MS}_\epsilon$ is a critical point of $f$.
So 0 is not a critical value of $f$.

Since 0 is not a critical value of $f$, by Sard's Theorem, there must be a $\delta>0$ such that there are no critical values in $(-\delta, \delta)$. Thus, the preimage of $(-\delta,\delta)$ under $f$ is homeomorphic to a cylinder $\Sols_0\times I$. Since $f$ is continuous in $t$, $f\inv(t)$ is an isotopy from $\Sols_0\cap\ol{\MS_\epsilon}$ to $\Sols_t\cap \ol{\MS_\epsilon}$ for any $0<t<\delta$.
\end{proof}

Thus, for sufficiently small $t$ we can define $(A^\pm)_t$ and $(H_{\gamma_0}^\pm)_t$ to be the relevant components of $\Sols_t\cap \ol{\MS_\epsilon}$ such that $(A^\pm)_0 = A^\pm$ and $(H_{\gamma_0}^\pm)_0 = H_{\gamma_0}^\pm$. Furthermore, note that if $p_0^i\in A^\pm$, then $p_t^i\in (A^\pm)_t$ and similarly if $p_0^i\in H_{\gamma_0}^\pm$, then $p_t^i\in (H_{\gamma_0}^\pm)_t$.

Recall the set $\SV\subset \Sols$. Define $s:\SV\to [-1,1]$ by
\begin{equation}
\label{eq:s}
s:\Gamma(\gamma_0,\theta,\frac{\pi}{2},\beta)\mapsto \sin\beta\sin(\theta-\beta).
\end{equation}
Note that if $\mathbf{x}\in\SV$, $s(\mathbf{x}) = s(\iota(\mathbf{x}))$ and so $s$ descends to a well-defined function $h(\SV)\to [-1,1]$. Abusing notation, this map will also be called $s$.

For each point $\mathbf{x}\in\SV$, the sign of $s(\mathbf{x})$ determines how the neighborhood of $\rho$ will be perturbed.

\begin{figure}
    \centering
    \begin{subfigure}[b]{0.3\textwidth}
        \centering
        \includegraphics[width=\textwidth]{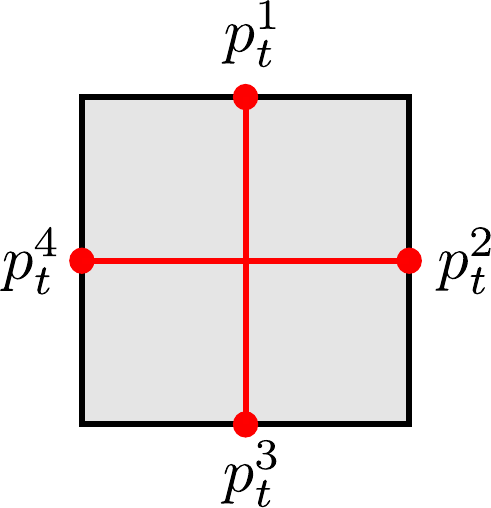}
        \caption{}
        \label{fig:Res1}
    \end{subfigure}
    \hspace{.5cm}
    \begin{subfigure}[b]{0.3\textwidth}
        \centering
        \includegraphics[width=\textwidth]{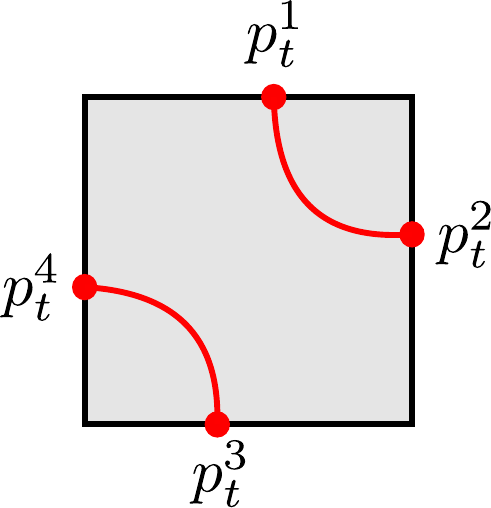}
        \caption{}
        \label{fig:Res2}
    \end{subfigure}
    \hspace{.5cm}
    \begin{subfigure}[b]{0.3\textwidth}
        \centering
        \includegraphics[width=\textwidth]{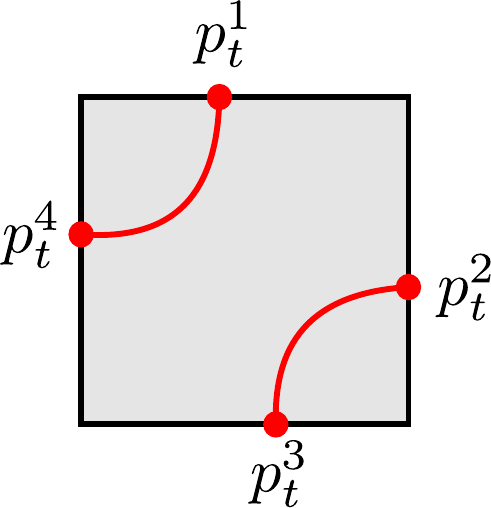}
        \caption{}
        \label{fig:Res3}
    \end{subfigure}
    \caption{For $\mathbf{x}\in\SV$, these are three possible ways that $\Psi\inv_\epsilon(\mathbf{x})$ can intersect $\Sols_t$.}
\end{figure}

\begin{theorem}
\label{thm:smoothing}
For $\mathbf{x}\in \SV$, there is a sufficiently small $t>0$ such that $\Psi_\epsilon\inv(\mathbf{x})$ intersects $\Sols_t$ in one of three ways, determined by $s(\mathbf{x})$:

\begin{itemize}
    \item If $s(\mathbf{x})>0$, the intersection is two intervals. The boundary of one interval is $\{p^1_t, p^2_t\}$ while the other has boundary $\{p^3_t, p^4_t\}$.
    
    \item If $s(\mathbf{x})<0$, the intersection is two intervals. The boundary of one interval is $\{p^1_t, p^4_t\}$ while the other has boundary $\{p^3_t, p^2_t\}$.
    
    \item If $s(\mathbf{x})=0$, the intersection is a cone on four points, like in the unperturbed case.
\end{itemize}

\end{theorem}

\begin{proof}
For $\gamma_0\in\{0,\pi\}$ and fix $\theta,\beta\in S^1$, consider $\mathbf{x} = \widetilde{\Gamma}(\gamma_0,\theta,\frac{\pi}{2},\beta)\in \SV$. Recall that we can parameterize $\Psi\inv_\epsilon(\mathbf{x})$ by $(\alpha,\gamma)$. $H^\pm_{\gamma_0}$ intersects $\partial\Psi\inv_\epsilon(\mathbf{x})$ at $(\frac{\pi}{2},\pm\epsilon)$. By Propositions \ref{prop:sqinttrans}, we know that for $t<t_\epsilon$, $(H^\pm_{\gamma_0})_t$ intersects $\partial\Psi\inv_\epsilon(\mathbf{x})$ at the coordinate $(\alpha^\pm_t,\pm\epsilon)$ where $\alpha^\pm_0=\frac{\pi}{2}$ and $\alpha^\pm_t(\theta,\beta) \in[\frac{\pi}{2}-\epsilon,\frac{\pi}{2}+\epsilon]$. For perturbation $t$, we want to know whether $\alpha^+_t$ or $\alpha^-_t$ is bigger. To do this, we calculate $\frac{d}{dt}\alpha^\pm_t|_{t=0}$.

With $\theta$, $\beta$, and $\gamma_0$ fixed, consider $\F$ as a function of $\alpha$ and $t$. 
\[\nabla\F(\alpha,t)|_{t=0,\alpha = \pi/2,\gamma=\pm \epsilon} = (\mp\sin\epsilon, 2\sin(\theta-\beta)\sin(\beta\pm\epsilon))\] An orthogonal vector to this is \[(2\sin(\theta-\beta)\sin(\beta\pm\epsilon), \pm\sin\epsilon).\]
So \[\frac{d}{dt}\alpha^\pm_t|_{t=0} = 2\sin(\theta-\beta)\sin(\beta\pm\epsilon)/\sin(\pm\epsilon).\] 

Now we can compare $\frac{d}{dt}\alpha^+_t|_{t=0}$ and $\frac{d}{dt}\alpha^-_t|_{t=0}$. There are a couple of cases to consider depending on $\epsilon$ and $\beta$:

\begin{enumerate}
    \item $\sign(\sin(\beta+\epsilon))=\sign(\sin(\beta))=\sign(\sin(\beta-\epsilon))$.
    \item $\sign(\sin(\beta+\epsilon))=\sign(\sin(\beta))=-\sign(\sin(\beta-\epsilon))$. In this case, $\abs{\sin(\beta-\epsilon)}<\abs{\sin(\beta+\epsilon)}$.
    \item $-\sign(\sin(\beta+\epsilon))=\sign(\sin(\beta))=\sign(\sin(\beta-\epsilon))$. In this case, $\abs{\sin(\beta+\epsilon)}<\abs{\sin(\beta-\epsilon)}$.
\end{enumerate}

If $s(\mathbf{x})>0$ and $\epsilon$ and $\beta$ satisfy case 1, then $\frac{d}{dt}\alpha^+_t|_{t=0}>0$ and $\frac{d}{dt}\alpha^-_t|_{t=0}<0$%
In case 2, then $\frac{d}{dt}\alpha^\pm_t|_{t=0}>0$ and $\abs{\frac{d}{dt}\alpha^+_t|_{t=0}}>\abs{\frac{d}{dt}\alpha^-_t|_{t=0}}$ and so $\frac{d}{dt}\alpha^+_t|_{t=0}>\frac{d}{dt}\alpha^-_t|_{t=0}>0$. 
In case 3, then $\frac{d}{dt}\alpha^\pm_t|_{t=0}<0$ and $\abs{\frac{d}{dt}\alpha^-_t|_{t=0}}>\abs{\frac{d}{dt}\alpha^+_t|_{t=0}}$ and so $\frac{d}{dt}\alpha^-_t|_{t=0}<\frac{d}{dt}\alpha^+_t|_{t=0}<0$. 
Thus whenever $s(\mathbf{x})>0$, for $t$ smaller than some $t_{\theta,\beta}$, $\alpha^+_t>\alpha^-_t$. In the same way, it can be shown that if $s(\mathbf{x})<0$ then for $t$ smaller than some $t_{\theta,\beta}$, $\alpha^-_t>\alpha^+_t$.

In particular, because $\alpha_t^+\ne \alpha_t^-$, there is no value of $\alpha_0\in [\frac{\pi}{2}-\epsilon, \frac{\pi}{2}+\epsilon]$ such that $\Phi_t(\gamma,\theta,\alpha_0,\beta) = 0$ for every $\gamma$. Thus by Lemma \ref{lem:function}, the values of $\gamma$ satisfying $\F_t(\gamma,\theta,\alpha,\beta)$ depend smoothly on $\alpha$ and thus $\Psi\inv(\mathbf{x})\cap\Sols_t$ is a smooth 1-manifold. %

Furthermore, the intersections at $p_t^i$ are transverse, so they must be the boundary of $\Psi\inv(\mathbf{x})\cap\Sols_t$. A smooth 1-manifold with four boundary components must contain at least two $I$ components. By Lemma \ref{lem:function}, the interval which has one boundary point at $p_t^4$ must have its other boundary point at $p^1_t$ if $\alpha_t^+<\alpha_t^-$ which occurs when $s(\mathbf{x})<0$. Likewise, its other boundary point is $p^3_t$ if $\alpha_t^-<\alpha_t^+$, which occurs when $s(\mathbf{x})>0$.

There cannot be any other components of $\Psi\inv(\mathbf{x} )\cap\Sols_t$. If there were, there would either have to be interval components with boundaries on the boundary of the square, interval components with boundaries in the interior of $\Psi\inv(\mathbf{x})$, or circle components. We have already seen that there are only four points where the intersection hits the boundary and so the first case cannot happen. Lemma \ref{lem:function} rules out the other cases.

\bigskip

Now we turn our attention to the $s(\mathbf{x})=0$ case. Let $\mathbf{x} = \widetilde{\Gamma}(\gamma_0,\theta_0,\frac{\pi}{2},\beta)$. What we want to show is that the function \[f_t(\gamma,\alpha) = \F_t|_{\sin(\beta_0)\sin(\theta_0-\beta_0)=0}(\gamma,\theta_0,\alpha, \beta_0)\] whose zero-set defines $\Psi\inv(\mathbf{x})\cap \Sols_t$ is equivalent to a product $g(\alpha)h(\gamma,\alpha)$ such that $g(\alpha)$ has a unique root in the range $(\frac{\pi}{2}-\epsilon, \frac{\pi}{2}+\epsilon)$. In this case, the solutions $g(\alpha)=0$ gives an interval with endpoints $p^1_t$ and $p^3_t$ (see the vertical line in Figure \ref{fig:Res1}). We know there can only be one such vertical line since $\Psi\inv(\mathbf{x})$ intersects $\Sols_t$ one time on each face of its boundary. Then by Lemma \ref{lem:function}, solutions of $h(\gamma,\alpha) = 0$ must cut out a smooth 1-dimensional submanifold of $\Psi\inv(\mathbf{x})$ which gives an interval with endpoint $p^4_t$ and $p^2_t$ (see the horizontal component of Figure \ref{fig:Res1}).

The representations for which $s(\mathbf{x})=0$ fall into two cases, either $\sin(\theta-\beta)=0$ or $\sin(\beta)=0$. %

\noindent\textbf{Case 1:} $\sin(\theta-\beta) = 0$

\noindent In this case, we can simplify Equation \ref{eq:F} to:
\begin{align}
\F_t(\gamma,\theta,\alpha,\beta) =& \cos\gamma\left[-2\cos n\sinc(n)t\cos^2\alpha\cos\theta + 2\sinc^2(n)t^2\sin^2\alpha\cos\beta\cos\alpha\sin\theta\right]+\\\nonumber
&\sin\gamma\left[\cos\alpha\cos^2 n+\sinc^2(n)t^2\cos\alpha\sin^2\alpha\cos(\beta+\theta)\right]\\\nonumber
=& \cos\alpha\Big[\cos\gamma\left[-2\cos n\sinc(n)t\cos\alpha\cos\theta + 2\sinc^2(n)t^2\sin^2\alpha\cos\beta\sin\theta\right]+\\\nonumber
&\quad\quad\quad\sin\gamma\left[\cos^2 n+\sinc^2(n)t^2\sin^2\alpha\cos(\beta+\theta)\right]\Big]\\\nonumber
\end{align}

Because this has a factor of $\cos\alpha$, the previous argument shows that the intersection must be a cone on four points.

\noindent\textbf{Case 2:} $\sin\beta = 0$. Further assume $\sin(\theta-\beta)\ne 0$, since this is covered by the previous case.

\begin{align}
\label{eq:beta0}
\F_t(\gamma,\theta,\alpha,\beta)\mid_{\sin\beta=0} =& \cos\gamma\left[-2\cos n\sinc(n)t\cos^2\alpha\cos\theta + \right.\\\nonumber
&\phantom{\cos\gamma[[}\left. 2\sinc^2(n)t^2\sin^2\alpha\cos\alpha\sin\theta\cos\theta\right]+\\\nonumber
&\sin\gamma\left[\cos\alpha\cos^2 n- 2\cos n\sinc(n)t\sin^2\alpha\sin\theta+ \right.\\\nonumber 
&\phantom{\sin\gamma[[}\left. \sinc^2(n)t^2\cos\alpha\sin^2\alpha\cos^2\theta\right]
\end{align}

To show that there is a value of $\alpha_0\in[\frac{\pi}{2}-\epsilon, \frac{\pi}{2}+\epsilon]$ such that \[\F_t(\gamma,\theta,\alpha_0,\beta) = 0\] for each value of $\gamma$, we just need to find a value of $\alpha_0$ for which $\frac{\partial\F}{\partial \gamma}=0$ for every point $(\alpha_0,\gamma)$. As shown in Theorem \ref{thm:smooth}, $\frac{\partial\F}{\partial \gamma}=0$ if either $x\in\Sk$ or $\ol{pa}p=\pm x$. 

If $x\in \Sk$, that means we can set $\alpha=\frac{\pi}{2}$.
\begin{align}
\F_t(\gamma,\theta,\frac{\pi}{2},\beta)\mid_{\sin\beta=0} =& -2\cos t\sin t\sin\gamma\sin\theta
\end{align}

Since we know that $\sin(\theta-\beta)\ne 0$ and $\sin\beta=0$, we know $\sin\theta\ne 0$ and so the only such solution is when $\sin\gamma=0$. Thus we must examine the other case where $\ol{pa}p=\pm x$. This condition only depends on $\alpha_0$, $\beta_0$, and $\theta_0$. Thus whenever this is satisfied there is a solution for each value of $\gamma$, proving the theorem. Since $\sin\beta = 0$, we know that $x = e^{-\alpha\j}\i$. In particular, $\ol{pa}p$ should have no $\j$ component. So Equation \ref{eq:pap} gives \[0=\cos n\cos\alpha\cos\theta-\sinc(n)t\sin^2\alpha\sin\theta\cos\theta.\] So either $\cos\theta = 0$ or $\cos n\cos\alpha = \sinc(n)t\sin^2\alpha\sin\theta$.

\noindent\textbf{Case 2a:} $\cos\theta=0$.

In this case, $\F_t$ simplifies to
\begin{align}
\F_t(\gamma,\theta,\alpha,\beta)\mid_{\sin\beta=0,\cos\theta=0} =& \sin\gamma\left[\cos\alpha\cos^2 n\pm 2\cos n\sinc(n)t\sin^2\alpha\right]
\end{align}
So we can set \[h(\gamma, \alpha) = \sin\gamma, \quad g(\alpha) =\cos\alpha\cos^2 n\pm 2\cos n\sinc(n)t\sin^2\alpha.\]
We must show that there exists a value of $t$ small enough there is some $\alpha\in[\frac{\pi}{2}-\epsilon, \frac{\pi}{2}+\epsilon]$ such that $g(\alpha)=0$. We can rearrange $g(\alpha)=0$ to be 
\[\tan\alpha = \pm\frac{\cos n}{2\sinc(n)t\sin\alpha}.\]

Since $0\le n\le t$, for small $t$, $\cos t\le \cos n\le 1$ and $0\le 2\sinc(n)t\sin\alpha \le 2\sin t$ and thus \[\frac{\cos n}{2\sinc(n)t\sin\alpha}> \frac{\cos t}{2\sin t}.\]
Thus as $t \to 0$, $\arctan(\pm\frac{\cos n}{2\sinc(n)t\sin\alpha})\to \frac{\pi}{2}$, giving a solution for $g(\alpha)=0$ arbitrarily close to $\frac{\pi}{2}$.

\noindent\textbf{Case 2b:} $\cos n\cos\alpha = \sinc(n)t\sin^2\alpha\sin\theta$.

Using this assumption to make substitutions in Equation \ref{eq:beta0}

\begin{align*}
\F_t(\gamma,\theta,\alpha,\beta)\mid_{\sin\beta=0} =& 2\cos\gamma\sinc(n)t\cos\alpha\cos\theta\left[-\cos n\cos\alpha + \sinc(n)t\sin^2\alpha\sin\theta\right]+\\\nonumber
&\sin\gamma\left[-\cos\alpha\cos^2 n+
\sinc^2(n)t^2\cos\alpha\sin^2\alpha\cos^2(\theta)\right]
\end{align*}

which further simplifies to 

\begin{align*}
\F_t(\gamma,\theta,\alpha,\beta)\mid_{\sin\beta=0} =& \sin\gamma\cos\alpha\left[-\cos^2 n+
\sinc^2(n)t^2\sin^2\alpha\cos^2(\theta)\right]
\end{align*}

Thus we have $h(\gamma,\alpha) = \sin\gamma$ and $g(\alpha) = \cos\alpha\left[-\cos^2 n+
\sinc^2(n)t^2\sin^2\alpha\cos^2(\theta)\right]$, finishing the proof for the $s(\mathbf{x}) = 0$ case.

\bigskip

Thus the theorem holds for each $\mathbf{x}\in\SV$ for some value of $t_{\theta_0,\beta_0}$. By compactness of $\SV$, there is a value of $t$ such that the theorem holds for all points $\mathbf{x}\in\SV$ simultaneously.

\end{proof}
\begin{corollary}
\label{cor:slices}
For $\mathbf{x} = \widetilde{\Gamma}(\gamma_0, \theta,\frac{\pi}{2},\beta)\in \SV$, there is a sufficiently small $t$ such that $\Psi_\epsilon\inv(\mathbf{x})$ intersects $\Sols_t$ in one of three ways, determined by $s(\mathbf{x})$ and $\gamma_0$:

\begin{itemize}
    \item If $s(\mathbf{x})>0$ and $\gamma_0=0$ or $s(\mathbf{x})<0$ and $\gamma_0=\pi$, the intersection is two intervals. One interval has boundary on $(H^+_{\gamma_0})_t$ and $(A^-)_t$ while the other has boundary on $(H^-_{\gamma_0})_t$ and $(A^+)_t$.
    
    \item If $s(\mathbf{x})<0$ and $\gamma_0=0$ or $s(\mathbf{x})>0$ and $\gamma_0=\pi$, the intersection is two intervals. One interval has boundary on $(H^+_{\gamma_0})_t$ and $(A^+)_t$ while the other has boundary on $(H^-_{\gamma_0})_t$ and $(A^-)_t$.
    
    \item If $s(\mathbf{x})=0$, the intersection is a cone on four points, like the unperturbed case.
\end{itemize}
\end{corollary}

\begin{proof}
This follows immediately from Theorem \ref{thm:smoothing} and checking for the two values of $\gamma(\mathbf{x})$ which points $p_t^i$ lie in which submanifolds.
\end{proof}

\begin{figure}
    \centering
    \includegraphics[height=2in]{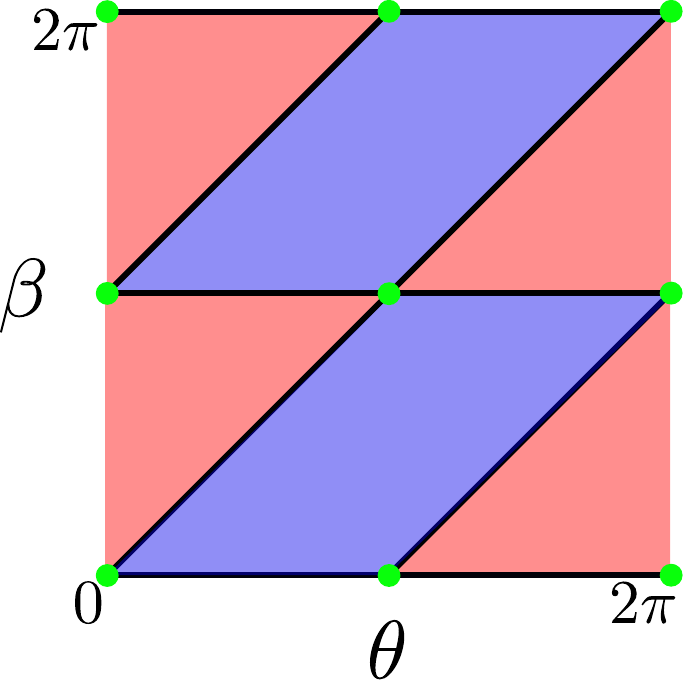}
    \caption{The torus $\SV^{\gamma_0}$, color coded by the sign of $s$. The blue regions have $s>0$, the red regions have $s<0$, the black lines have $s=0$. The green points denote the points in $\Sols^0$ (and so $s=0$ at those points).}
    \label{fig:sTorus}
\end{figure}

\begin{corollary}
\label{cor:nbhdsmooth}

For $\mathbf{x}=\widetilde{\Gamma}(\gamma_0,\theta,\frac{\pi}{2},\beta)\in \SV\setminus \Sols^0$, there exists a neighborhood $U\subset \Pr$ of $\mathbf{x}$ which intersects $\Sols_t$ in one of the following ways:

\begin{itemize}
    \item If $s(\mathbf{x})>0$ and $\gamma_0=0$ or $s(\mathbf{x})<0$ and $\gamma_0=\pi$, the intersection has two components. The boundary of one component intersects $(H^+_{\gamma_0})_t$ and $(A^-)_t$ while the boundary of the other intersects $(H^-_{\gamma_0})_t$ and $(A^+)_t$.
    
    \item If $s(\mathbf{x})<0$ and $\gamma_0=0$ or $s(\mathbf{x})>0$ and $\gamma_0=\pi$, the intersection has two components. The boundary of one component intersects $(H^+_{\gamma_0})_t$ and $(A^+)_t$ while the boundary of the other intersects $(H^-_{\gamma_0})_t$ and $(A^-)_t$.
    
    \item If $s(\mathbf{x})=0$, the intersection only has one component, whose boundary intersects $(H^+_{\gamma_0})_t$, $(H^-_{\gamma_0})_t$, $(A^+)_t$, and $(A^-)_t$.
\end{itemize}

\end{corollary}

\begin{proof}
In the first two cases, when $s(\mathbf{x})\ne 0$, then there is a neighborhood $U_\SV\subset \SV_{\gamma_0}$ of $\mathbf{x}$ such that $\sign(s(U_\SV)) = \sign(s(\mathbf{x}))$. %
Taking $U = \Psi\inv(U_\SV)$, Corollary \ref{cor:slices} finishes the claim.

If $s(\mathbf{x})=0$, again pick a neighborhood $U_\SV\subset \SV_{\gamma_0}$ in $\SV$ about $\mathbf{x}$. We can decompose $U_\SV = U_+\sqcup U_0 \sqcup U_-$ depending on the sign of $s$. Each piece is necessarily non-empty and $U_\SV$ can be picked so that each piece is connected. By the previous paragraph, $\Psi\inv(U_+\sqcup U_-)$ is made up of four components. By Corollary \ref{cor:slices}, $\Psi\inv(U_0)$ is a single component and by the smoothness of $\Sols_t$, we know that the four components of $\Psi\inv(U_+\sqcup U_-)$ limit to the single component of $\Psi\inv(U_0)$. Thus $U = \Psi\inv(U_\SV)$ is connected.
\end{proof}

With this, we understand what happens in a neighborhood of each point in $\SV\setminus \Sols^0 \subset \Sols_{0}$ as $t$ is increased from 0. Together with Lemma \ref{lem:isotopy} which explains the behavior of $\Sols_t$ away from $\SV$, we have a good understanding of $\Sols^*_t$. Finally, we examine the neighborhoods of the singular points $\Sols^0 \subset \Sols_t$:

\begin{theorem}
\label{thm:Cone}
Fix $t\in(0,\frac{\pi}{4})$. In $\Sols_t$, each point in $\Sols^0$ has a neighborhood which is a cone on a torus. Similarly, in $\RR_\t(C_3)$, each point in $\RR_\t(C_3)^\U$ has a neighborhood which is a cone on a torus.
\end{theorem}

\begin{proof}
For a fixed perturbation parameter $t$, the Hessian of $\F_t(\gamma,\theta,\alpha,\beta)$ at $\gamma =\theta=\beta=0$ and $\alpha=\frac{\pi}{2}$ is 

\[\bmat{0 & -2 \, \cos\left(t\right) \sin\left(t\right) & -1 & 2 \, \cos\left(t\right) \sin\left(t\right) \\
-2 \, \cos\left(t\right) \sin\left(t\right) & 0 & -2 \, \sin\left(t\right)^{2} & 2 \, \cos\left(t\right) \sin\left(t\right) \\
-1 & -2 \, \sin\left(t\right)^{2} & -4 \, \cos\left(t\right) \sin\left(t\right) & 0 \\
2 \, \cos\left(t\right) \sin\left(t\right) & 2 \, \cos\left(t\right) \sin\left(t\right) & 0 & -4 \, \cos\left(t\right) \sin\left(t\right)
}\]

This can be diagonalized to give: 

\[\bmat{-4 \, \cos\left(t\right) \sin\left(t\right) & 0 & 0 & 0 \\
0 & \cos\left(t\right) \sin\left(t\right) & 0 & 0 \\
0 & 0 & -\frac{2 \, {\left(2 \, \cos\left(t\right)^{2} - 1\right)} \sin\left(t\right)}{\cos\left(t\right)} & 0 \\
0 & 0 & 0 & \frac{\left(2 \, \sin\left(t\right)^{2} + 1\right)^{2} \cos\left(t\right)}{2 \, {\left(2 \, \cos\left(t\right)^{2} - 1\right)} \sin\left(t\right)}
}\]

The Hessian is non-singular for small, nonzero $t$ and the signature is easily seen to be 0 for any such value of $t$ (the first two entries in the diagonal clearly have opposite sign, as do the last two). Similar calculations can be performed for the other corner points yielding the same result.

By the Morse lemma, a neighborhood of a corner point in $\Sols_t$ is therefore homeomorphic to a neighborhood of $0$ of the zero set of $f(w,x,y,z) = w^2+x^2-(y^2+z^2)$, which can be seen to be a cone on a torus.%

Let $\mathbf{x}\in \Sols^0$ and let $U\subset\Pr$ be a small, $\iota$-equivariant neighborhood of $\mathbf{x}$. Then by the preceding argument $\partial U$ intersects $\Sols_t$ in a torus $T$. $\iota$ must act on $T$ by a free involution. $T$ in the boundary of $\Sols_t\setminus U$. Since $\iota$ is an orientation preserving involution on $\Sols_t\setminus U$, it must preserve the orientation of its boundary, $T$. Thus, $T/\iota\cong T^2$. Therefore, the corner points of $\Sols_t/\iota\cong \RR_\t(C_3)$ have neighborhoods which are cones on tori.

\end{proof}

\subsection{Perturbed Character Varieties of Tangle Sums}
\label{sec:pertsum}

For this section, we will assume that $(T_1, \pi_1)$ and $(T_2, \pi_2)$ are a good pair and, furthermore, that $\RR_{\pi_1\cup \pi_2}(T_1+T_2)$ has no corner circles. Recall that by Proposition \ref{prop:goodpair} and Lemma \ref{lem:internal}, such perturbations always exist and can be arbitrarily close to any given perturbations.

For a representation $\rho\in \RR_{\pi_1\cup\pi_2}(T_1+T_2)$, let \[c(\rho) := \rho|_{C_3} \in \RR(C_3).\] 
By Corollary \ref{cor:singular}, all the singular points of $\RR_{\pi_1\cup \pi_2}(T_1+T_2)$ have $c(\rho)\in\SV$. In the same way that the sign of $s:\SV\to [-1,1]$ given in Equation \ref{eq:s} controlled how a neighborhood of a point in $\SV$ perturbs (Corollary \ref{cor:nbhdsmooth}), the sign of $s(c(\rho))$ will control how the neighborhood of $\rho\in \RR_{\pi_1\cup \pi_2}(T_1+T_2)$ perturbs.

\begin{lemma}
\label{lem:sint}
Let $(\rho_1, \rho_2) \in \RR_{\pi_1}(T_1)\times_{\{0,\pi\}}\RR_{\pi_2}(T_2)$. If $\theta(\rho_1)$ or $\theta(\rho_2)$ is in $\{0,\pi\}$, by Theorem \ref{thm:nice} there is a unique representation $[\rho_1+\rho_2]\in\RR_{\pi_1\cup \pi_2}(T_1+T_2)$. For this representation, $s(c([\rho_1+\rho_2])) = 0$.
\end{lemma}

\begin{proof}
Let $\eta = c([\rho_1+\rho_2]) \in\RR(C_3)$. By Lemma \ref{lem:revengH},  $\eta$ is binary dihedral and so $\eta = \Gamma(\gamma,\theta,\frac{\pi}{2},\beta)$. In this case, $\theta_1(\eta) = \beta$ and $\theta_2(\eta) = \theta-\beta$ by Equations \ref{eq:pccoords2} and \ref{eq:pccoords4}.

If $\theta_1(\eta) \in \{0,\pi\}$, then $\sin\beta = \sin(\theta_1(\eta)) = 0$ and so $s(\eta) = 0$. If instead $\theta_2(\eta) \in \{0,\pi\}$, then $\sin(\theta-\beta) = \sin(\theta_2(\eta)) = 0$ and so $s(\eta) = 0$.
\end{proof}

\begin{lemma}
\label{lem:endop}
Let $(\rho_1, \rho_2) \in \RR_{\pi_1}(T_1)\times_{\{0,\pi\}}\RR_{\pi_2}(T_2)$ such that $\theta(\rho_1), \theta(\rho_2)\in (0,\pi)$. By Theorem \ref{thm:nice}, there is an associated circle's worth of representations $[\rho_1+\rho_2]_\psi \in \RR_{\pi_1\cup\pi_2}(T_1+T_2)$. Then
\[s(c([\rho_1+\rho_2]_0)) = -s(c([\rho_1+\rho_2]_\pi))\]
\end{lemma}

\begin{proof}
The proof of Theorem \ref{lem:revengH} includes explicit calculations for $[\rho_1 + \rho_2]_0$ and $[\rho_1 + \rho_2]_\pi$. Plugging these in for $s$ proves the theorem:

\noindent\textbf{Case 1}: $\theta(\rho_1)\ge\theta(\rho_2)$ and $\theta(\rho_1)+\theta(\rho_2)\le \pi$

$[\rho_1 + \rho_2]_0$ sends $a\mapsto \i$, $b\mapsto \pm \i$, $x\mapsto e^{\theta(\rho_1)\k}\i$, $c\mapsto e^{(\theta(\rho_1)-\theta(\rho_2))\k}\i$

$[\rho_1 + \rho_2]_\pi$ sends $a\mapsto \i$, $b\mapsto \pm \i$, $x\mapsto  e^{\theta(\rho_1)\k}\i$, $c\mapsto e^{(\theta(\rho_1)+\theta(\rho_2))\k}\i$
\begin{align*}
s([\rho_1 + \rho_2]_0) &= -2\sin(\theta(\rho_1))\sin(\theta(\rho_1)-(\theta(\rho_1)-\theta(\rho_2)))\\
&= -2\sin(\theta(\rho_1))\sin(\theta(\rho_2))\\
&= 2\sin(\theta(\rho_1))\sin(-\theta(\rho_2))\\
&= -s([\rho_1 + \rho_2]_\pi)
\end{align*}

\noindent\textbf{Case 2}: $\theta(\rho_1)\ge\theta(\rho_2)$ and $\theta(\rho_1)+\theta(\rho_2)> \pi$

$[\rho_1 + \rho_2]_0$ sends $a\mapsto \i$, $b\mapsto \pm \i$, $x\mapsto e^{-\theta(\rho_1)\k}\i$, $c\mapsto e^{(\theta(\rho_1)-\theta(\rho_2))\k}\i$

$[\rho_1 + \rho_2]_\pi$ sends $a\mapsto \i$, $b\mapsto \pm \i$, $x\mapsto  e^{-\theta(\rho_1)\k}\i$, $c\mapsto e^{(2\pi-\theta(\rho_1)-\theta(\rho_2))\k}\i$
\begin{align*}
s([\rho_1 + \rho_2]_0) &= -2\sin(-\theta(\rho_1))\sin(-\theta(\rho_1)-(\theta(\rho_1)-\theta(\rho_2)))\\
&= 2\sin(-\theta(\rho_1))\sin(2\theta(\rho_1)-\theta(\rho_2))\\
&= -s([\rho_1 + \rho_2]_\pi)
\end{align*}

\noindent\textbf{Case 3}: $\theta(\rho_1)\le \theta(\rho_2)$ and $\theta(\rho_1)+\theta(\rho_2)\le \pi$

$[\rho_1 + \rho_2]_0$ sends $a\mapsto \i$, $b\mapsto \pm \i$, $x\mapsto e^{-\theta(\rho_1)\k}\i$, $c\mapsto e^{(\theta(\rho_2)-\theta(\rho_1))\k}\i$

$[\rho_1 + \rho_2]_\pi$ sends $a\mapsto \i$, $b\mapsto \pm \i$, $x\mapsto e^{\theta(\rho_1)\k}\i$, $c\mapsto e^{(\theta(\rho_1)+\theta(\rho_2))\k}\i$
\begin{align*}
s([\rho_1 + \rho_2]_0) &= -2\sin(-\theta(\rho_1))\sin(-\theta(\rho_1)-(\theta(\rho_2)-\theta(\rho_1)))\\
&= 2\sin(\theta(\rho_1))\sin(-\theta(\rho_2)) \\
&= -s([\rho_1 + \rho_2]_\pi)
\end{align*}

\noindent\textbf{Case 4}: $\theta(\rho_1)\le \theta(\rho_2)$ and $\theta(\rho_1)+\theta(\rho_2)> \pi$

$[\rho_1 + \rho_2]_0$ sends $a\mapsto \i$, $b\mapsto \pm \i$, $x\mapsto e^{-\theta(\rho_1)\k}\i$, $c\mapsto e^{(\theta(\rho_2)-\theta(\rho_1))\k}\i$

$[\rho_1 + \rho_2]_\pi$ sends $a\mapsto \i$, $b\mapsto \pm \i$, $x\mapsto e^{\theta(\rho_1)\k}\i$, $c\mapsto e^{(2\pi-\theta(\rho_1)-\theta(\rho_2))\k}\i$

\begin{align*}
s([\rho_1 + \rho_2]_0) &= -2\sin(-\theta(\rho_1))\sin(-\theta(\rho_1)-(\theta(\rho_1)-\theta(\rho_2)))\\
&= 2\sin(-\theta(\rho_1))\sin(2\theta(\rho_1)-\theta(\rho_2))\\
&= -s([\rho_1 + \rho_2]_\pi)
\end{align*}

\end{proof}

\subsubsection{A Decomposition of $\RR_{\t\cup \pi_1\cup\pi_2}(T_1+T_2)$}

Because $(T_1,\pi_1)$ and $(T_2,\pi_2)$ are a good pair, Proposition \ref{prop:nicecomp} extends to the perturbed case to give

\[\RR_{\t\cup\pi_1\cup\pi_2}(T_1+T_2) \cong \RR_{\pis}(T_1\twedge T_2)\times_{P_1\times P_2} \RR_\t(C_3) = \RR_{\pi_1\cup\pi_2}(T_1\twedge T_2)\times_{P_1\times P_2} h_t(\Sols_t)\]

We will use the decomposition $\Pr =\MS_\epsilon \sqcup\NS_\epsilon$ introduced in Section \ref{sec:decomp} to induce a decomposition on $\RR_{\t\cup\pi_1\cup \pi_2}(T_1+T_2)$.
\[\J_t^\epsilon := \RR_{\pi_1\cup\pi_2}(T_1\twedge T_2)\times_{P_1\times P_2} h_t(\Sols_t\cap\NS_\epsilon)\]
\[\K_t^\epsilon := \RR_{\pi_1\cup\pi_2}(T_1\twedge T_2)\times_{P_1\times P_2} h_t(\Sols_t\cap \ol{\MS_\epsilon})\]

It is worth comparing this decomposition of $\RR_{\t\cup\pi_1\cup \pi_2}(T_1+T_2)$ when $t=0$ to the decomposition $\RR_{\pi_1\cup \pi_2}(T_1+T_2) = M\sqcup N$ defined in (\ref{eq:decompM}) and (\ref{eq:decompN}). Because $\NS_0 = \SV$, we get that $\J_0^0 = N$ and $\K_0^0 = M$. Then for $\epsilon>0$, the fact that $\Sols\cap \NS_\epsilon$ is a neighborhood of $\SV$ implies that $\J_0^\epsilon$ is a neighborhood of $N$. On the other hand, taking the complements gives us that $M$ is a neighborhood of $\K_0^\epsilon$.

Now, in an analogous result to Lemma \ref{lem:isotopy}, we want to observe that away from the singular points, the perturbation acts nicely:

\begin{lemma}
\label{lem:Kgood}
Fix $\epsilon>0$. There exists a $t_\epsilon>0$, such that there is a regular homotopy from $\K_0^\epsilon$ to $\K_{t_\epsilon}^\epsilon$.
\end{lemma}

\begin{proof}
Recall from Proposition \ref{prop:transverse} that the maps $\mu:\Sols^*\to P_1\times P_2$ and $\nu: \RR_{\pi_1}(T_1)\times \RR_{\pi_2}(T_2)\to P_1\times P_2$ are transverse. Then clearly the map $\mu|_{\ol{\MS}_\epsilon}:\Sols^*\cap \ol{\MS}_\epsilon\to P_1\times P_2$ is transverse with $\nu$. 
Let \begin{align*}
\CSS :=& \{(\gamma_1,\theta_1,\gamma_2,\theta_2)\in P_1\times P_2 \mid \gamma_1 = \gamma_2 \in\{0,\pi\}, \theta_1=\theta_2 \in (0,\pi)\}\cup\\
&\{(\gamma_1,\theta_1,\gamma_2,\theta_2)\in P_1\times P_2 \mid \gamma_1 = \gamma_2 \in\{0,\pi\}, \theta_1=\pi-\theta_2 \in (0,\pi)\}.
\end{align*}
Due to the assumption that $\RR_{\pis}(T_1+T_2)$ has no corner circles, Theorem \ref{thm:notnice} guarantees that the image of the map $\nu$ does not intersect $\CSS$ in $P_1\times P_2$.
Next, consider the space 
\begin{align*}
(\Sols\setminus\Sols^*)\cap \ol{\MS}_\epsilon 
&= ((\Sols\setminus \Sols^\dag)\cup \SV)\cap \ol{\MS}_\epsilon\\ 
&= (\Sols\setminus \Sols^\dag)\cap \ol{\MS}_\epsilon \subset (\Sols\setminus \Sols^\dag)\setminus \Sols^0\\ 
& = \{(\i,b,c,x) \in \Pr\mid b,c\in\{\pm \i\}, x\notin\{\pm \i\}\}.
\end{align*}
Equations \ref{eq:pccoords1}, \ref{eq:pccoords2}, \ref{eq:pccoords3}, \ref{eq:pccoords4} tell us that the map from the set $\{(\i,b,c,x) \in \Pr\mid b,c\in\{\pm \i\}, x\notin\{\pm \i\}\}$ to $P_1\times P_2$ has image precisely $\CSS$. Therefore, the image of the map $\mu':(\Sols\setminus\Sols^*)\cap \ol{\MS}_\epsilon\to P_1\times P_2$ is a subset of $\CSS$. Therefore, the images of $\mu'$ and $\nu$ are disjoint and so $\mu'\pitchfork\nu$ trivially. Thus, the map $\mu_0:\Sols\cap \ol{\MS}_\epsilon\to P_1\times P_2$ is transverse to $\nu$.

Now we define the family of maps $\mu_t:\Sols_t\cap \ol{\MS}_\epsilon\to P_1\times P_2$.
By Corollary \ref{cor:smoothds} and Theorem \ref{thm:lagcomp}, the map $\K_t^\epsilon \to P_3$ is an immersion if $\mu_t \pitchfork \nu$. As a consequence of Lemma \ref{lem:isotopy}, $\mu_t$ is an isotopy for small $t$. Because transversality is a stable property and $\mu_0 \pitchfork \nu$, there is a value of $t_\epsilon$ such that for $t\in[0,t_\epsilon)$, $\mu_t \pitchfork \nu$ and thus the map $\K_t^\epsilon\to P_3$ is an immersion. Because $\mu_t$ is continuous in $t$, so is $\K_t^\epsilon$, finishing the claim.  

\end{proof}

We define a fundamental domain of $\Sols^*_t$ with respect to the action of $\iota$.
\begin{align}
\label{eq:funddom}
\widetilde{\Sols}^*_t := &\{\widetilde{\Gamma}(\gamma,\theta,\alpha,\beta)\in\Sols_t \mid \theta \in (0,\pi)\} \cup\\ 
&\{\widetilde{\Gamma}(\gamma,\theta,\alpha,\beta)\in\Sols_t \mid \theta \in \{0,\pi\}, \gamma\in(0,\pi)\} \cup\nonumber\\
&\{\widetilde{\Gamma}(\gamma,\theta,\alpha,\beta)\in\Sols_t \mid \theta,\gamma \in \{0,\pi\}, \alpha = \frac{\pi}{2}, \beta\in (0,\pi)\}\nonumber
\end{align}
Because the map $h_t:\Sols^*_t\to \RR_\t(C_3)^*$ is the quotient by the action of $\iota$, every point in $\RR_\t(C_3)^*$ has a unique lift to $\widetilde{\Sols}^*_t$. Define  
\[\ol{c}_t:\RR_\t(T_1+T_2)^*\to \widetilde{\Sols}^*_t\subset \Pr\]
so that $\ol{c}_t(\rho)$ is the lift of $c_t(\rho)$ into $\widetilde{\Sols}^*_t$.

Corollary \ref{cor:singular} showed that $\J_0^0=N$ contains all of the singular points of $\RR_{\pi_1\cup\pi_2}(T_1+T_2)$. For each $\rho\in \J_0^0$, $\ol{c}(\rho)\in\SV$, so pick a neighborhood $U_\rho\subset \Pr$ of $\ol{c}(\rho)$ which satisfies Corollary \ref{cor:nbhdsmooth}. Because the points of $\J_0^0$ are isolated by Lemma \ref{lem:isolates}, there is an $\epsilon_0$ small enough so that each component of $\J_0^{\epsilon_0}$ contains at most one point of $\J_0^0$ and $\ol{c}$ maps the component containing $\rho$ into $U_\rho$. Let $V^\rho_0$ be the component of $\J_0^{\epsilon_0}$ that contains $\rho$. 

Note that $V^\rho_0\subset \RR_{\pi_1\cup\pi_2}(T_1\twedge T_2)\times_{P_1\times P_2} (\RR_{\Pert_0}(C_3) \cap U_\rho)$, but $\RR_{\pi_1\cup\pi_2}(T_1\twedge T_2)\times_{P_1\times P_2} (\RR_{\Pert_0}(C_3) \cap U_\rho)$ may also contain other disjoint components as well. Let $W_\rho$ be an open neighborhood of $\rho\mid_{\RR_{\pi_1\cup\pi_2}(T_1\twedge T_2)}$ in $\RR_{\pi_1\cup\pi_2}(T_1\twedge T_2)$ such that $V^\rho_0 = W_\rho \times_{P_1\times P_2} (\RR_{\Pert_0}(C_3) \cap U_\rho)$. Then define \[V^\rho_t := W_\rho \times_{P_1\times P_2} (\RR_\t(C_3) \cap U_\rho).\] Then for sufficiently small $t$, $\RR_{\t\cup\pi_1\cup\pi_2}(T_1+T_2) = \K^{\epsilon_0}_t\cup (\bigcup_{\rho\in \J^0_0} V^\rho_t)$. Since Lemma \ref{lem:Kgood} explains the behavior of $\K^{\epsilon_0}_t$, if $V^\rho_t$ can be understood for each $\rho\in\J^0_0$, these can be glued back together to construct $\RR_{\t\cup\pi_1\cup\pi_2}(T_1+T_2)$.

\begin{remark}
If $\mu_t|_{\RR_\t(C_3) \cap \NS_\epsilon}$ is not transverse to $\nu$ for any small value of $t$, it is possible that some $V^\rho_t$ are singular. This would mean that $\RR_{\t\cup\pi_1\cup\pi_2}(T_1+T_2) = \K^{\epsilon_0}_t\cup (\bigcup_{\rho\in \J^0_0} V^\rho_t)$ is not a manifold. In that case, Theorem \ref{thm:reg} guarantees that there is an arbitrarily small perturbation $\pi$ such that $\RR_{\t\cup\pi_1\cup\pi_2\cup \pi}(T_1+T_2)$ is a manifold and in turn that each $V^\rho_t$ is a manifold. If $\pi$ is chosen to be small enough, then all of the results of Section \ref{sec:Perturbing} so far would still hold with $\t+\pi$ in place of $\t$. Thus, we will assume that for sufficiently small $t$, $V^\rho_t$ is a 1-manifold with boundary.
\end{remark}

Each $V^\rho_0$ is homeomorphic to either a cone on one point, an interval, or a cone on four points by Lemma \ref{lem:isolates}. If $V^\rho_0$ is a cone on one point then $\rho$ is an abelian representation of $\RR_{\t\cup\pi_1\cup\pi_2}(T_1+T_2)$ which maps to the corner of the pillowcase. Thus, we are interested in the cases where $V^\rho_0$ is homeomorphic to either an interval or a cone on four points. The nature of $V^\rho_t$ will depend on which of these cases $V^\rho_0$ falls into so we will treat these two cases separately.

\begin{theorem}
\label{thm:intpert}
Suppose $\rho\in \J^0_0$ such that $V^\rho_0 \cong I$. Then $V^\rho_t$ is homeomorphic to an interval and a union of circles.
\end{theorem}

\begin{proof}
Let $\partial(V^\rho_0) = \{p^1_0, p^2_0\} \subset \K^{\epsilon_0}_0$. Let $p^1_t$ and $p^2_t$ be the corresponding boundary points of $\K^{\epsilon_0}_t$ which are well defined by Lemma \ref{lem:Kgood}. Since these are non-abelian representations, they cannot be the boundary of the perturbed character variety. Thus, for sufficiently small $t$, $\partial(V^\rho_t) = \{p^1_t, p^2_t\}$. Since $V^\rho_t$ is a 1-manifold, it must contain an interval component with boundary $p^1_t$ and $p^2_t$. Any other components of $V^\rho_t$ must be circles. As $t\to 0$, the image of these circles in $P_3$ must approach the point $p_3(\rho)$.
\end{proof}

Now take $\rho\in\J^0_0$ such that $V^\rho_0$ is a cone on four points. $\partial(V^{\rho}_0)$ consists of four points. The fact that $(T_1, \pi_1)$ and $(T_2, \pi_2)$ satisfy Condition \ref{con:edges} guarantees that $\ol{c}$ maps one of the boundary points to $A^-$ and another to $A^+$. Lemma \ref{lem:isolates} guarantees that the other two are mapped to $H^+_{\gamma(\rho)}$ and $H^-_{\gamma(\rho)}$ (henceforth abbreviated as $H^+$ and $H^-$ for simplicity). Label the four boundary points $\{p_{H^+\grho}, p_{H^-\grho}, p_{A^+}, p_{A^-}\}$ appropriately. These points also lie in $\K^{\epsilon_0}_t$. Let $\{(p_{H^+\grho})_t, (p_{H^-\grho})_t, (p_{A^+})_t, (p_{A^-})_t\}$ be the corresponding boundary points of $\K^{\epsilon_0}_t$. It is clear that $\ol{c}((p_{H^+\grho})_t)\in (H^+\grho)_t$ and similarly for the other points. 

For the same reason as the previous theorem, we know that \[\partial(V^\rho_t) = \{(p_{H^+\grho})_t, (p_{H^-\grho})_t, (p_{A^+})_t, (p_{A^-})_t\}.\] Recall that $V^\rho_t = W_\rho \times_{P_1\times P_2} (\RR_0(C_3) \cap U_\rho)$. By Corollary \ref{cor:nbhdsmooth}, we know that $\RR_0(C_3) \cap U_\rho$ has two components, $U_1$ and $U_2$, so \[V^\rho_t = W_\rho \times_{P_1\times P_2} (\RR_0(C_3) \cap U_1) \cup W_\rho \times_{P_1\times P_2} (\RR_0(C_3) \cap U_2).\] We know that $s(\rho)\ne0$, and further suppose that $s(\rho)>0$. In this case, Corollary \ref{cor:nbhdsmooth} asserts that the boundary of $U_1$ intersects $(H^-\grho)_t$ and $(A^+)_t$ while the boundary of $U_2$ intersects $(H^-\grho)_t$ and $(A^-)_t$. That means the boundary of $W_\rho \times_{P_1\times P_2} (\RR_0(C_3) \cap U_1)$ is $\{(p_{H^-\grho})_t, (p_{A^+})_t\}$ and the boundary of $W_\rho \times_{P_1\times P_2} (\RR_0(C_3) \cap U_2)$ is $\{(p_{H^+\grho})_t, (p_{A^-})_t\}$. Since $V^\rho_t$ is smooth, it must contain at least two intervals, one with boundary $\{(p_{H^-\grho})_t, (p_{A^+})_t\}$ and the other with boundary $\{(p_{H^+\grho})_t, (p_{A^-})_t\}$. As in Theorem \ref{thm:intpert}, there may be additional circle components whose image on $P_3$ approaches $p_3(\rho)$ as $t\to 0$.

The same logic shows that if instead $s(\rho)<0$, $V^\rho_t$ must contain at least two intervals with boundaries $\{(p_{H^+\grho})_t, (p_{A^+})_t\}$ and $\{(p_{H^-\grho})_t, (p_{A^-})_t\}$. The following Theorem summarizes the above results:

\begin{theorem}
\label{thm:coneres}
For $\rho\in\J^0_0$, if $V^\rho_0$ is a cone on four points, then $V^\rho_t$ is homeomorphic to the union of $S^0\times I$ and some number of circle components. The boundaries of the interval components are:

\begin{itemize}
    \item  $\{(p_{H^+})_t, (p_{A^-})_t\}$ and $\{(p_{H^-})_t, (p_{A^+})_t\}$ if $s(\rho)>0$.

    \item $\{(p_{H^+})_t, (p_{A^+})_t\}$ and $\{(p_{H^-})_t, (p_{A^-})_t\}$ if $s(\rho)<0$.
\end{itemize}
\end{theorem}

Suppose $\rho\in\J^0_0$ such that $V^\rho_0$ is a cone on four points. By Theorem \ref{thm:nice}, such representations come in pairs in the form $\rho_0:=[\rho_1+\rho_2]_0$ and $\rho_\pi:=[\rho_1+\rho_2]_\pi$ as a part of a circle of representations $\{[\rho_1+\rho_2]_\psi\in \RR_{\pi_1\cup\pi_2)(T_1+T_2)}\mid \psi\in S^1\}$. %
Let $\partial(V^{\rho_0}_0) = \{p^{\rho_0}_{H^+\grho}, p^{\rho_0}_{H^-\grho}, p^{\rho_0}_{A^+}, p^{\rho_0}_{A^-}\}$ and $\partial(V^{\rho_\pi}_0) = \{p^{\rho_\pi}_{H^+\grho}, p^{\rho_\pi}_{H^-\grho}, p^{\rho_\pi}_{A^+}, p^{\rho_\pi}_{A^-}\}$.
Let \[h_+:=\{[\rho_1+\rho_2]_\psi\in \K^{\epsilon_0}_0\mid \sin\psi>0\}\] and \[h_-:=\{[\rho_1+\rho_2]_\psi\in \K^{\epsilon_0}_0\mid \sin\psi<0\}\] be the two intervals connecting $V_0^{[\rho_1 + \rho_2]_0}$ and $V_0^{[\rho_1 + \rho_2]_\pi}$ which are a subset of the aforementioned circle of representations. Note that $\ol{c}(h_\pm) \subset H^\pm_{\gamma(\rho_0)}$ and so $\partial(h_+) = \{p^{\rho_0}_{H^+}, p^{\rho_\pi}_{H^+}\}$ and $\partial(h_-) = \{p^{\rho_0}_{H^-}, p^{\rho_\pi}_{H^-}\}$. Because $\K^{\epsilon_0}_t$ is regularly homotopic to $\K^{\epsilon_0}_0$, it makes sense to consider $(h_+)_t$ and $(h_-)_t$ whose boundaries are $\{(p^{\rho_0}_{H^+})_t, (p^{\rho_\pi}_{H^+})_t\}$ and $\{(p^{\rho_0}_{H^-})_t, (p^{\rho_\pi}_{H^-})_t\}$, respectively.

Lemma \ref{lem:endop} tells us that $s(\rho_0) = -s(\rho_\pi)$. Assume that $s(\rho_0)>0$ and $s(\rho_\pi)<0$. By Theorem \ref{thm:coneres} we know that there is an interval $I_0$ in $V^{\rho_0}_t$ with boundary $\{(p^{\rho_0}_{A^-})_t, (p^{\rho_0}_{H^+})_t\}$ and an interval $I_\pi$ in $\{(p^{\rho_\pi}_{H^+})_t, (p^{\rho_\pi}_{A^+})_t\}$. Then gluing together $I_0\cup (h_+)_t\cup I_\pi$ gives one interval connecting $(p^{\rho_0}_{A^-})_t$ to $(p^{\rho_\pi}_{A^+})_t$. Connecting the other intervals gives a similar interval connecting $(p^{\rho_0}_{A^+})_t$ to $(p^{\rho_\pi}_{A^-})_t$. Note that if the signs of $s(\rho_0)$ and $s(\rho_\pi)$ are reversed, the roles of $(h_+)_t$ and $(h_-)_t$ swap, but the end result is still two intervals with boundaries $\{(p^{\rho_0}_{A^-})_t, (p^{\rho_\pi}_{A^+})_t\}$ and $\{(p^{\rho_0}_{A^+})_t, (p^{\rho_\pi}_{A^-})_t\}$. %

\begin{remark}

Given $\rho\in\RR_{\pi}(T_1+T_2)$, if $p_3(\rho)\in \{(\gamma,\theta)\in P\mid \theta\in(0,\pi)\}$, then $\ol{c}(\rho)\in A^+$. If $p_3(\rho)\in \{(\gamma,\theta)\in P\mid \theta\in[\pi,2\pi]\}$, then $\ol{c}(\rho)\in A^-$.

\end{remark}

The following theorem summarizes the results of this section.

\begin{theorem}
\label{thm:cross}

Let $(T_1, \pi_1)$ and $(T_2, \pi_2)$ be a good pair. Furthermore, assume $\pi_1$ and $\pi_2$ are picked so that all circles components of $\RR^{\ol{\H}}_{\pi_1\cup\pi_2}(T_1+T_2)$ are internal. Then for sufficiently small $t>0$, there is a subspace of $\RR_{\t\cup\pi_1\cup\pi_2}(T_1+T_2)$ homeomorphic to the following construction:

Let $\{C_i\}$ be the circle components of $\RR^{\ol{\H}}_{\pi_1\cup\pi_2}(T_1+T_2)$. Let $\rho_1^i$ and $\rho_2^i$ be the intersection of $C_i$ and $\RR^{\ol{\A}}_{\pi_1\cup\pi_2}(T_1+T_2)$. Let $N(C_i)$ be neighborhoods of $C_i$ in $\RR_{\pi_1\cup\pi_2}(T_1+T_2)$ such that $N(C_i)$ is disjoint from $N(C_j)$ for $i\ne j$. Let $X = \RR_{\pi_1\cup\pi_2}(T+S)\setminus (\bigcup_i N(C_i))$. The endpoints (except for the abelian points) of $X$ are of the form $p^{\rho^i_j}_{A^{\pm}}$. Let $I^{i}_1$ be an interval whose endpoints are identified with $p^{\rho^i_1}_{A^{+}}$ and $p^{\rho^i_2}_{A^{-}}$. Similarly, let $I^{i}_2$ be an interval whose endpoints are identified with $p^{\rho^i_1}_{A^{-}}$ and $p^{\rho^i_2}_{A^{+}}$. Each $I^i_*$ maps into the pillowcase such that as $t\to 0$, the image of $I^i_*$ approaches the image of $C_i$.

Call the previously described subspace of $\RR_{\t\cup\pi_1\cup\pi_2}(T_1+T_2)$ the \emph{main components} of $\RR_{\t\cup\pi_1\cup\pi_2}(T_1+T_2)$. Any other components of $\RR_{\t\cup\pi_1\cup\pi_2}(T_1+T_2)$ are homeomorphic to circles which map to the pillowcase in a way so that as $t\to 0$, they approach $p_3(\J^0_0)$. Call any such components \emph{auxiliary components}.
\end{theorem}

\begin{figure}
    \centering
    \begin{subfigure}[b]{0.22\textwidth}
        \centering
        \includegraphics[width=\textwidth]{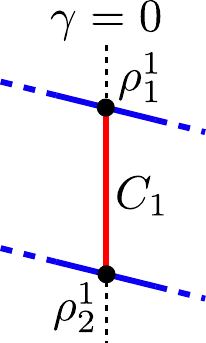}
        \caption{}
        \label{fig:Cross1}
    \end{subfigure}
    \hspace{1cm}
    \begin{subfigure}[b]{0.22\textwidth}
        \centering
        \includegraphics[width=\textwidth]{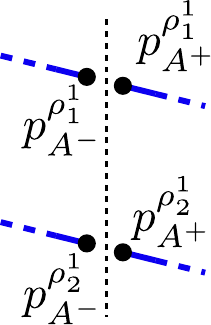}
        \caption{}
        \label{fig:Cross2}
    \end{subfigure}
    \hspace{1cm}
    \begin{subfigure}[b]{0.22\textwidth}
        \centering
        \includegraphics[width=\textwidth]{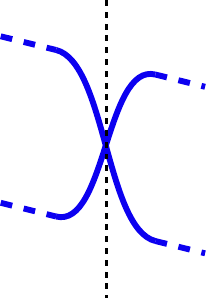}
        \caption{}
        \label{fig:Cross3}
    \end{subfigure}
    \caption{One step of the process detailed in Theorem \ref{thm:cross}.}
    \label{fig:Cross}
\end{figure}

Figure \ref{fig:Cross} shows what the process of constructing the main component looks like for one circle, $C_1$ that is mapped to the edge $\{(\gamma,\theta)\in P\mid \gamma=0\}$. In terms of the image on the pillowcase, a neighborhood of an interior circle looks like Figure \ref{fig:Cross1}. Let $\rho_1$ represent the top singularity and $\rho_2$ represent the bottom. Assume that the dotted line represents the edge $\{(\gamma,\theta)\in P\mid \gamma=0\}$. After removing a neighborhood of the circle, the endpoints on the left of the line are sent to $A_-$ and so represent the points $p^{\rho_1}_{A_-}$ and $p^{\rho_2}_{A_-}$. The points to the right of the line get sent to $A_+$ and so the points are $p^{\rho_1}_{A_+}$ and $p^{\rho_2}_{A_+}$ as in Figure \ref{fig:Cross2}. Then add an interval connecting $p^{\rho_1}_{A_+}$ and $p^{\rho_2}_{A_-}$ and another connecting $p^{\rho_2}_{A_+}$ and $p^{\rho_1}_{A_-}$, as in Figure \ref{fig:Cross3}. If instead the dotted line is the interval $\gamma=\pi$, the $A_+$'s and $A_-$'s are swapped and the final effect is the same.

\begin{remark}
As a consequence of \cite{Herald}, if there are any auxiliary components, they must be mapped into the pillowcase so that they enclose a signed area of 0 such as a figure eight.
\end{remark}

\begin{figure}
    \centering
    \begin{subfigure}[b]{0.25\textwidth}
        \centering
        \includegraphics[width=\textwidth]{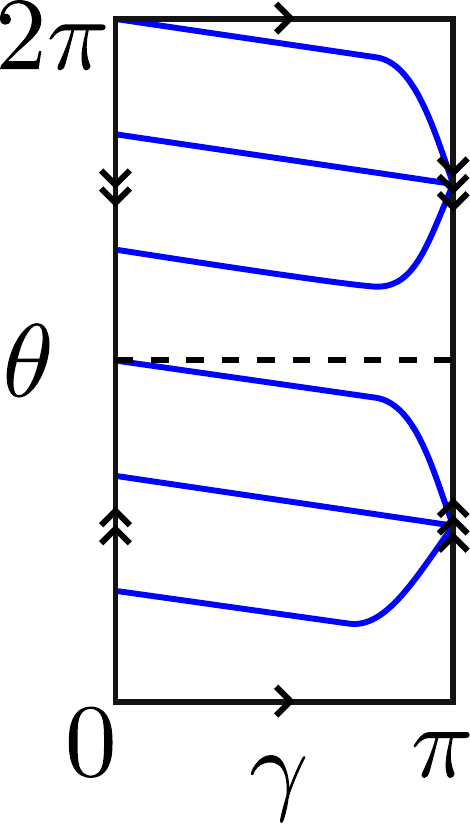}
        \caption{}
        \label{fig:Ex4}
    \end{subfigure}
    \hspace{1cm}
    \begin{subfigure}[b]{0.25\textwidth}
        \centering
        \includegraphics[width=\textwidth]{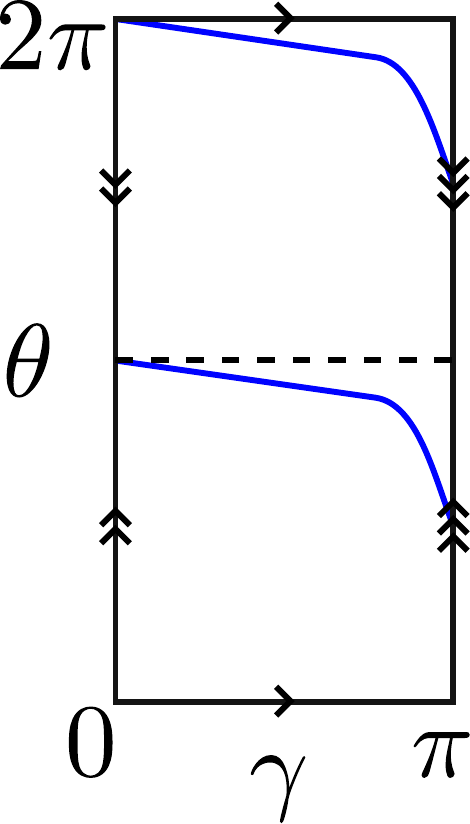}
        \caption{}
        \label{fig:ExResA}
    \end{subfigure}
    \hspace{1cm}
    \begin{subfigure}[b]{0.25\textwidth}
        \centering
        \includegraphics[width=\textwidth]{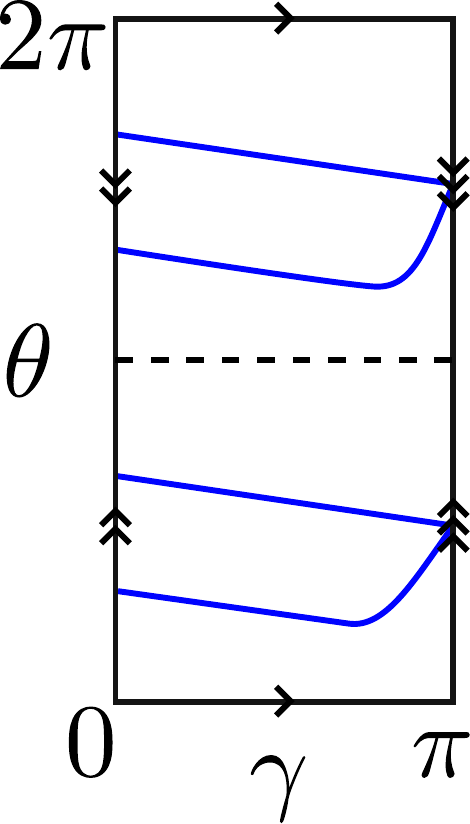}
        \caption{}
        \label{fig:ExResB}
    \end{subfigure}
    \caption{(A) shows the application of Theorem \ref{thm:cross} to $\RR(Q_\frac{1}{2}+Q_{-\frac{1}{3}})$ (whose unperturbed character variety is shown in Figure \ref{fig:Ex3}). The character variety splits into two components whose images in the pillowcase are shown in (B) and (C).}
    \label{fig:resolve}
\end{figure}

\begin{exmp}
The unperturbed character variety $\RR(Q_\frac{1}{2}+Q_{-\frac{1}{3}})$ is shown in Figure \ref{fig:Ex3}.
Figure \ref{fig:resolve} shows the application of Theorem \ref{thm:cross} to this character variety to get the image of the main components of $\RR_\t(Q_\frac{1}{2}+Q_{-\frac{1}{3}})$ in the pillowcase for small $t>0$.
\end{exmp}

\begin{figure}
    \centering
    \begin{subfigure}[b]{0.25\textwidth}
        \centering
        \includegraphics[width=\textwidth]{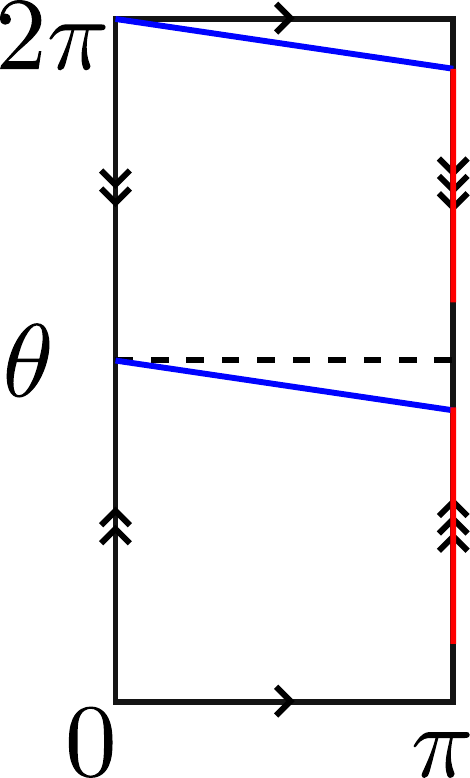}
        \caption{}
        \label{fig:ExResLimitA}
    \end{subfigure}
    \hspace{2cm}
    \begin{subfigure}[b]{0.25\textwidth}
        \centering
        \includegraphics[width=\textwidth]{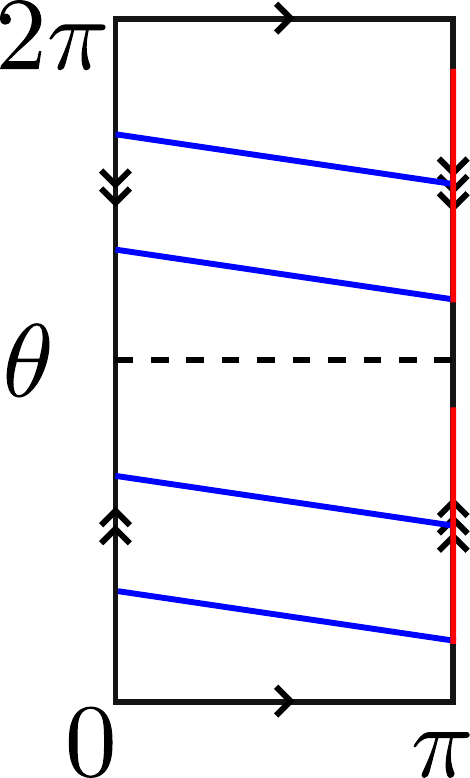}
        \caption{}
        \label{fig:ExResLimitB}
    \end{subfigure}
    \caption{The limits of the components of $\RR_t(Q_\frac{1}{2}+Q_{-\frac{1}{3}})$ as $t\to 0$. Note that here the red components are images of intervals in the character variety, compared to the unperturbed case in Figure \ref{fig:Ex3} where they were the images of circles.}
    \label{fig:limit}
\end{figure}

\begin{remark}
For computations, it is often useful to consider the limit of $\RR_{\t\cup\pis}(T_1+T_2)$ as $t\to 0$. The result is a character variety which is a manifold homeomorphic to the one described in Theorem \ref{thm:cross} but whose image in the pillowcase is the same as the unperturbed case. Figure \ref{fig:limit} shows what happens for $\RR_\t(Q_\frac{1}{2}+Q_{-\frac{1}{3}})$.
\end{remark}

\begin{remark}
The perturbed character variety of a tangle $\sum_{i=1}^n T_i$ can be identified by iterating the process described in Theorem \ref{thm:cross}. 
\end{remark}

\subsection{Related Perturbations}

\begin{figure}
    \centering
    \begin{subfigure}[b]{0.3\textwidth}
        \centering
        \includegraphics[width=\textwidth]{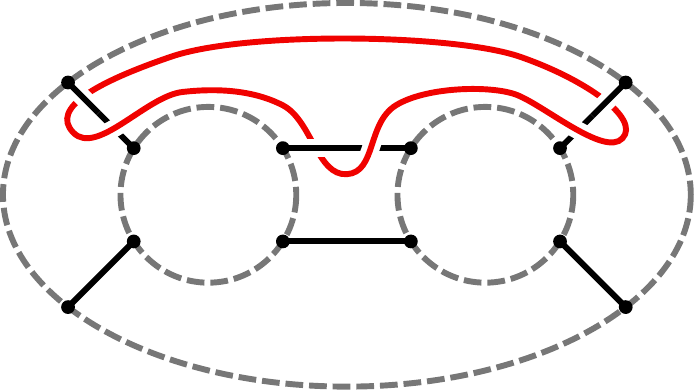}
        \caption{}
    \end{subfigure}
    \hspace{.5cm}
    \begin{subfigure}[b]{0.3\textwidth}
        \centering
        \includegraphics[width=\textwidth]{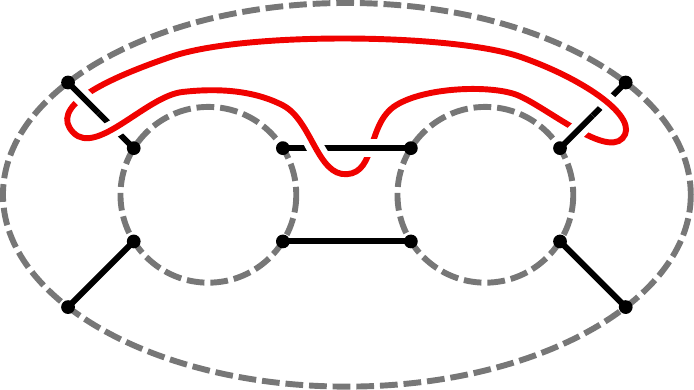}
        \caption{}
    \end{subfigure}
    \hspace{.5cm}
    \begin{subfigure}[b]{0.3\textwidth}
        \centering
        \includegraphics[width=\textwidth]{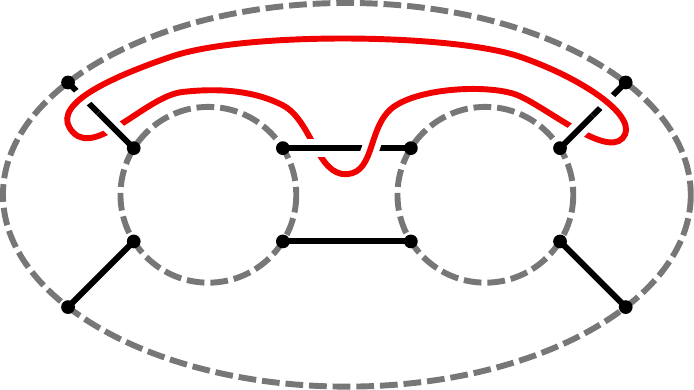}
        \caption{}
    \end{subfigure}
    \caption{The perturbation curves $D$, $D'$, and $D''$ respectively.}
    \label{fig:rotations}
\end{figure}

Because the perturbation curve is not rotationally symmetric, there are two other related perturbation curves for $C_3$ by rotating $D$, $D'$, and $D''$ as shown in Figure \ref{fig:rotations}. This section shows that $\RR_{\t\cup\pis}(T_1+T_2)$, $\RR_{D'_t\cup\pis}(T_1+T_2)$, and $\RR_{D''_t\cup\pis}(T_1+T_2)$ and their images in the pillowcase are very closely related.

Showing this is equivalent to showing that all of the results hold when $T_1+T_2$ is obtained by gluing $T_1 \sqcup T_2$ to $S_2 \sqcup S_3$ or $S_1 \sqcup S_3$ instead of $S_1 \sqcup S_2$ as previously considered. The analogous results from Section \ref{sec:unpert} hold by symmetry since there is no perturbation curve yet. All of the results from Section \ref{sec:pertC3} still hold because they only concern the character variety $\RR_\t(C_3)$ and not the tangles $T_i$. In Section \ref{sec:pertsum}, the analogous versions of Lemmas \ref{lem:sint} and  \ref{lem:endop} can be easily checked. The rest of the results in the section follow immediately from these lemmas and the results from the previous sections.

\begin{proposition}
    For a good pair with no corner circles $(T_1,\pi_1)$ and $(T_2, \pi_2)$, for sufficiently small $t>0$, the main components (defined within Theorem \ref{thm:cross}) of perturbed character varieties $\RR_{\t\cup\pis}(T_1+T_2)$, $\RR_{D'_t\cup\pis}(T_1+T_2)$, and $\RR_{D''_t\cup\pis}(T_1+T_2)$ are homeomorphic. Furthermore, the images of the respective main components are regularly homotopic in the pillowcase.
\end{proposition}

\section{Lagrangian Floer Methods for Character Varieties}
\label{sec:Hnat}

\subsection{The Earring Cobordism}
\label{sec:Earring}

\begin{figure}
    \centering
    \includegraphics[height=3in]{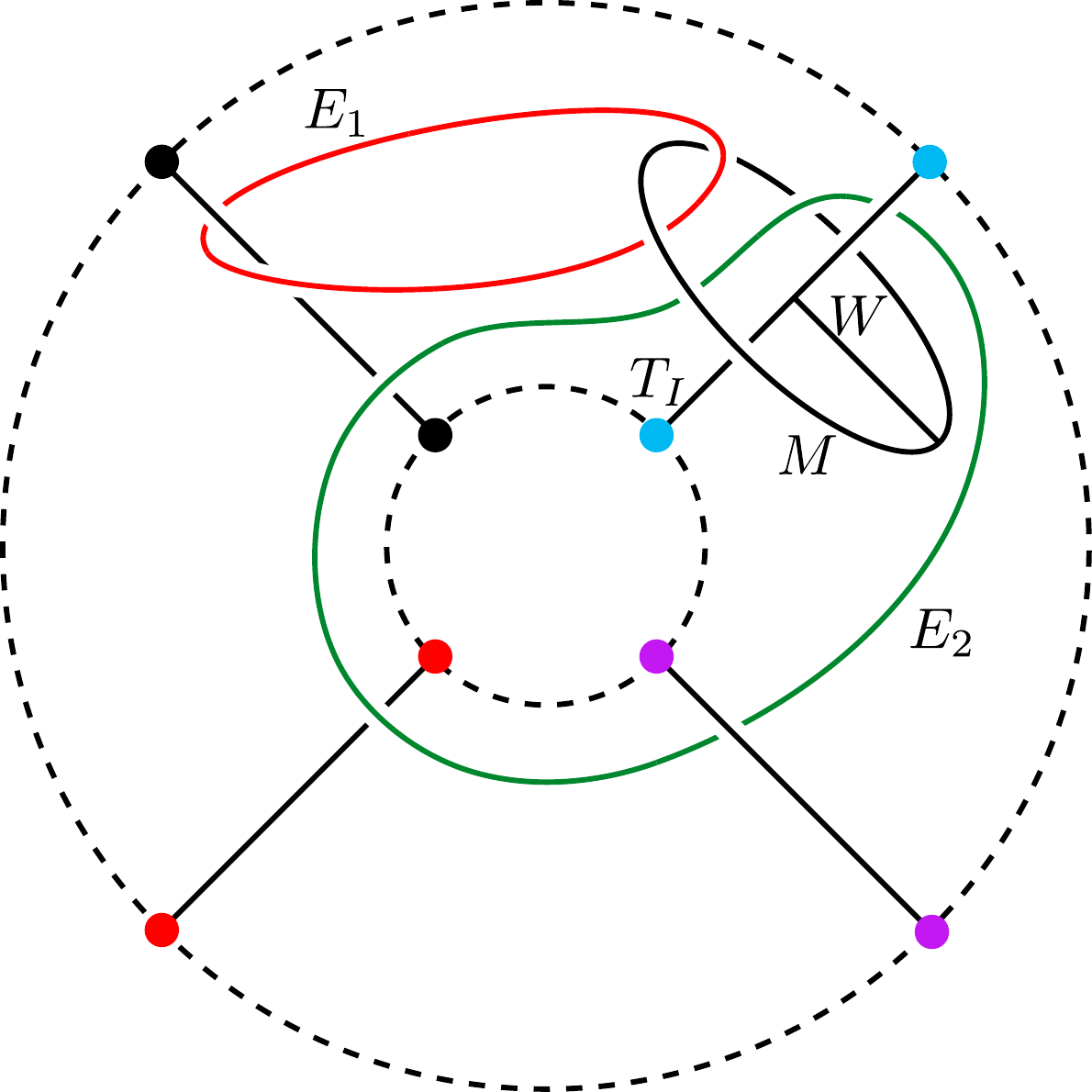}
    \caption{The earring cobordism, $\natural$.}
    \label{fig:EarringCob}
\end{figure}

Let $\natural$ denote the earring cobordism, as depicted in Figure \ref{fig:EarringCob}. It is made up from the identity cobordism $T_I = (S^2,4)\times I$. Let $x\in T_I$ and let $M$ be a meridian of $T_I$ at $x$. Then let $W$ be an arc from $M$ to $x$. Let $E = E_1\sqcup E_2$ be the perturbation curves shown in Figure \ref{fig:EarringCob}. Given a tangle $T$ with (possibly empty) perturbation data $\pi$, let $T^\natural$ be the result of gluing $T$ to the inner boundary of $\natural$.
Then we can define a new character variety, $\RR^\natural(T)$ which is the space of representations mod conjugation of $\pi_1(D^3\setminus T^\natural)$ into $\SU$ satisfying the following conditions:

\begin{itemize}
    \item Meridians of $T$ and $T_I$ get sent to traceless elements, $C(\i)$.
    
    \item Meridians of $M$ get sent to $C(\i)$.
    
    \item Meridians of $W$ get sent to $-1\in \SU$.
\end{itemize}

Similarly, for a perturbation $\pi$ in $T$, one can define the perturbed character variety of $T^\natural$, $\RR_{\pi\cup E}(T^\natural)$ which is the space of representations mod conjugation of $\pi_1(D^3\setminus (T^\natural\cup \pi\cup E))$ into $\SU$ satisfying the above conditions, and in addition, all perturbation curves in $\pi$ and $E$ satisfy the perturbation condition, Equation \ref{eq:pert}.

The effect of adding an earring to any tangle was studied in \cite{Earring}. Briefly, if $\pi$ is perturbation data so that $\RR_\pi(T)$ is a 1-manifold, then $\RR^\natural_{\pi\cup E}(T)$ is also a 1-manifold such that:

\begin{itemize}
    \item Any arc component of $\RR_\pi(T)$ with endpoints at corners of the pillowcase is turned into a figure 8 which loops around the corners.
    \item Every other component (each homeomorphic to a circle) is doubled.
\end{itemize}

$\RR^\natural_{E_1}(Q_0)$ was explicitly calculated in \cite{PCI}*{Section 7.3}. 

\subsection{Lagrangian Floer Homology}

Let $K: S^1\to S^3$ be a knot or link. Let $C \subset S^3$ be an embedded 2-sphere which intersects $K$ in 4 points. This is called a \textit{Conway sphere}. $C$ splits $S^3$ into two 3-balls, $D_1$ and $D_2$, and splits $K$ into two two-stranded tangles, $T_1$ and $T_2$, so $(S^3,K) = (D_1, T_1)\cup_{(C,4)} (D_2,T_2)$. Let $\pi_1$ be perturbation data for $(D_1,T_1)$ and let $\pi_2$ be perturbation data for $(D_2, T_2)$ and let $\pi:=\pi_1\cup \pi_2$. Let $L_1:=\RR_{\pi_1}(T_1)$, $L_2:=\RR_{\pi_2 \cup E}^\natural(T_2)$, and $p_i:L_i\to \RR(C,4)\cong P$ be the map described in Section \ref{sec:pillowcase}. Define $L^*_i$ to be the interior of $L_i$. The perturbations $\pi_i$ should be chosen to satisfy the following conditions: 

\begin{enumerate}[label={\bfseries (T\arabic*):}, ref=(T\arabic*), leftmargin=3\parindent]
    \item\label{con:lagimmerse} $L_1$ and $L_2$ are 1-manifolds with boundary. $p_i$ sends the boundary of $L_i$ to corners of $P$ and its restriction to the interior is an immersion.  
    \item\label{con:lagtransverse} $p_1(L_1) \pitchfork p_2(L_2)$.
\end{enumerate}

\noindent Theorem \ref{thm:reg} shows that perturbations can always be found which satisfy the first condition, and then shearing perturbations can be applied to ensure the second. From this point on, assume that these conditions are satisfied.

Define $p^*_i:L^*_i\to P^*$ to be the restriction of $p_i$. $p^*_i$ is a Lagrangian immersion \cite{HKRegularity}. Now we review the construction of the Lagrangian Floer homology $HF(L_1, L_2)$ following \cites{PCII, Abouz, Floer}.
Because the symplectic manifold in question, $P^*$, is 2-dimensional, the Riemann mapping theorem implies that the Lagrange Floer homology can be calculated with techniques from differential topology instead of symplectic topology.

The chain group $CF(L_1, L_2)$ is an $\mathcal{F}_2$ vector space freely generated by $\I:=L^*_1\times_{P^*} L^*_2$. Let $\pi_{L^*_i}:\I\to L^*_i$ be the projection map. Conditions \ref{con:lagimmerse} and \ref{con:lagtransverse} guarantee that this is a finite set. Because Condition \ref{con:lagimmerse} is satisfied, Lemma \ref{lem:generic} tells us that $L^*_1\times_{P^*} L^*_2 \cong \RR^\natural_{\pi\cup E}(K)$. The correspondence between flat connections and representations gives that $CF(L_1, L_2) \cong CI^\natural(K)$ (see Theorem \ref{thm:translate}). For $x,y\in \I$ such that $\pi_{L^*_i}(x)$ and $\pi_{L^*_i}(y)$ are connected by a path on $L^*_i$, the $\Z/4$ relative grading between $x$ and $y$ can be calculated by the method described in \cite{PCII}*{Section 5}. There is a distinguished generator, $x\in \I$, which does not depend on the choice of Conway sphere or perturbation. It was conjectured in \cite{PCII} and proved in \cite{PoudSav} that the relative grading on $\I$ can be promoted to an absolute $\Z/4$ grading by setting the grading of $x$ to $\sigma(K)\text{ mod } 4$ and that by doing so $CH^\natural \cong CI^\natural$ as $\Z/4$ graded groups. So $\rank(I^\natural(K))\le \rank(CF(L_1, L_2))$ in each grading. Let $\gr(y)$ denote this grading of $y\in\I$. %

\begin{figure}
    \centering
    \includegraphics[height=2in]{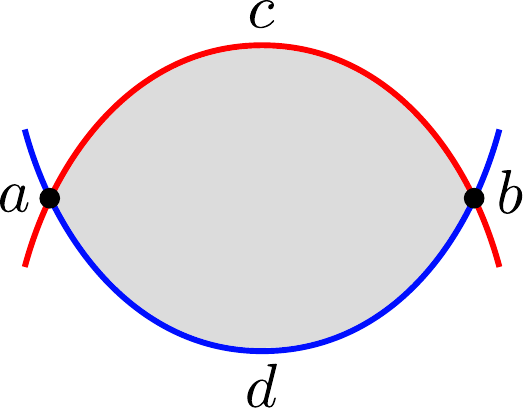}
    \caption{The standard bigon, $B$.}
    \label{fig:Bigon}
\end{figure}

Let $B$ be the \textit{standard bigon} shown in Figure \ref{fig:Bigon}. Then for $x,y\in \I$, let $\widetilde{\B}(x,y)$ be the space of orientation-preserving immersions $f:B\to P^*$ satisfying the following conditions:

\begin{itemize}
    \item $f(a) = x$
    \item $f(b) = y$
    \item $f(c) \subset p_1(L_1)$
    \item $f(d) \subset p_2(L_2)$
\end{itemize}

\noindent Let $\B(x,y) = \widetilde{\B}(x,y)/{\sim}$ where $f_1\sim f_2$ if $f_1 = f_2\circ g$ for an automorphism $g:B\to B$.
By \cite{PCII}*{Remark 3.14}, since Conditions \ref{con:lagimmerse} and \ref{con:lagtransverse} are met, $\B(x, y)$ is finite for any $x,y\in\I$. Let $\#\B(x,y)$ be the number of elements in $\B(x,y)$ mod 2. If $\gr(x)-\gr(y)\not\equiv 1 ~(\text{mod } 4)$ then $\#\B(x,y) = 0$. Define the differential $d$ so that for $x\in\I$, 
\[d(x)=\sum\limits_{y\in \I} \#\B(x,y) y.\]

By \cite{Abouz}, $d^2 = 0$. Then $HF(L_1, L_2)$ is the homology of $CF(L_1, L_2)$ with respect to the differential $d$.
As explained in \cite{PCII}, a similar count of immersed $(n+1)$-gons can be used to define higher composition maps $\mu^n$, where $\mu^1 = d$.%

In \cite{Earring}, an example was found of a knot $(S^3,K) = (D_1,T_1)\cup_{(C,4)} (D_2,T_2)$ with $T_1$ trivial such that $I^\natural(K) = HF(\RR^\natural_\pi(T_1), \RR_\pi(T_2)) \ne HF(\RR_\pi(T_1),\RR^\natural_\pi(T_2))$. The reason for this is that composing with the earring tangle can give rise to a phenomenon known as \emph{figure-eight bubbling}.
Inspired by work of Bottman and Wehrheim, Cazassus, Herald, Kirk, and Kotelskiy conjecture in \cite{Earring} that a link invariant can be defined via a modified Lagrangian Floer homology using bounding cochains. The precise statement is given shortly in Conjecture \ref{con:boundingcc}.

\subsection{Bounding Cochains}

\begin{definition}
If $L$ is an immersed Lagrangian of some symplectic manifold, then a \emph{bounding} cochain of $L$ is an element $b\in CF(L,L)$ satisfying the Maurer-Cartan equation

\begin{equation}
\label{eq:mc}
\sum\limits_{k\ge 0} \mu^k(b,\dots, b) = 0.
\end{equation}
\end{definition}

If $L_1$ and $L_2$ are immersed Lagrangians of the pillowcase and $b_i\in CF(L_i,L_i)$ are bounding cochains, then it is possible to define a Lagrangian Floer homology of the pairs $(L_1,b_1)$ and $(L_2, b_2)$. As groups $CF((L_1, b_1), (L_2, b_2))\cong CF(L_1, L_2)$. Just as the differential for $CF(L_1, L_2)$ could be obtained by counting smooth immersions of bigons to the pillowcase, the differential of $CF((L_1, b_1), (L_2, b_2))$ can be computed by a count of smooth immersions of polygons in the pillowcase. Two of the vertices of such a polygon must map to elements of $CF(L_1, L_2)$ while every other vertex must map to $b_1$ or $b_2$. For details, see \cite{FOOO}. Let $HF((L_1, b_1), (L_2, b_2))$ be the homology of the chain complex $CF((L_1, b_1), (L_2, b_2))$. Note that $CF(L_1, L_2) = CF((L_1, 0), (L_2, 0))$

\begin{conjecture}[\cite{Earring}*{Conjecture D}]
\label{con:boundingcc}
There exists an assignment that associates to every 2-tangle $T$ and its holonomy perturbed traceless character variety $\RR_\pi(T)\looparrowright P^*$ a bounding cochain $b\in CF(\RR_\pi(T), \RR_\pi(T))$, satisfying

\begin{itemize}
    \item $(R_\pi(T), b)$ is a well defined tangle invariant, as an object of the wrapped Fukaya category of $P^*$;
    \item This assignment of bounding cochains extends to tangles modified by the earring, which results in a tangle invariant $(\RR^\natural_\pi(T), b)$;
    \item Given a decomposition of a link $(S^3, L) = (D^3, T_1)\cup_{(S^2, 4)}(D^3, T_2)$, the corresponding Lagrangian Floer homology recovers the reduced singular instanton homology: \[HF((\RR_\pi(T_1), b_1), (\RR^\natural_\pi(T_2), b_2))\cong I^\natural(L).\]
\end{itemize}
\end{conjecture}

If the conjecture holds, $HF((\RR_\pi(T_1), b_1), (\RR^\natural_\pi(T_2), b_2))$ is called the \emph{pillowcase homology} of $L$.

For the torus knot $(S^3, T(4,5))$, a decomposition $(S^3, T(4,5)) = (D^3, T_1)\cup_{(S^2, 4)} (D^3,T_2)$ is identified in \cite{Earring} where $(D^3, T_2)$ is a rational tangle. A computation shows that \[\rank(HF(\RR_\pi(T_1),\RR^\natural_\pi(T_2)) = 7 \text{ and } \rank(HF(\RR^\natural_\pi(T_1), \RR_\pi(T_2)) = 9.\] It is known that $\rank(I^\natural(S^3,T(4,5))) = 7$. Because $(D^3, T_2)$ is a rational tangle, by Proposition \ref{prop:rational}, the image of $\RR_\pi(T_1)$ in the pillowcase has no self intersections and so $b_2 = 0$. A bounding cochain $b\in CF(\RR^\natural_\pi(T_1), \RR^\natural_\pi(T_1))$ is identified such that \[\rank(HF((\RR^\natural_\pi(T_1), b), (\RR_\pi(T_2), 0)) = 7.\] This proves that if Conjecture \ref{con:boundingcc} holds, the associated bounding cochain to a tangle with earring can be nontrivial.

\begin{figure}
    \centering
    \begin{subfigure}[b]{0.25\textwidth}
        \centering
        \includegraphics[width=\textwidth]{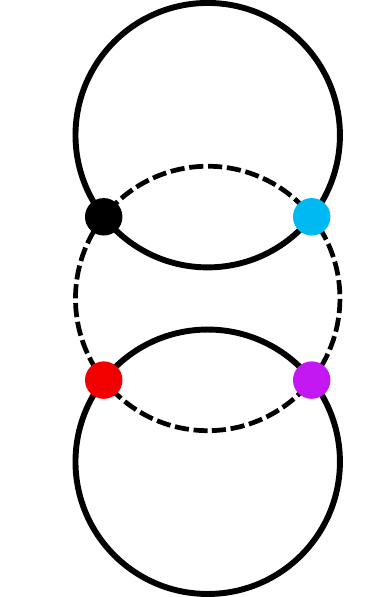}
        \caption{}
        \label{fig:Unlink}
    \end{subfigure}
    \hspace{.5cm}
    \begin{subfigure}[b]{0.2\textwidth}
        \centering
        \includegraphics[width=\textwidth]{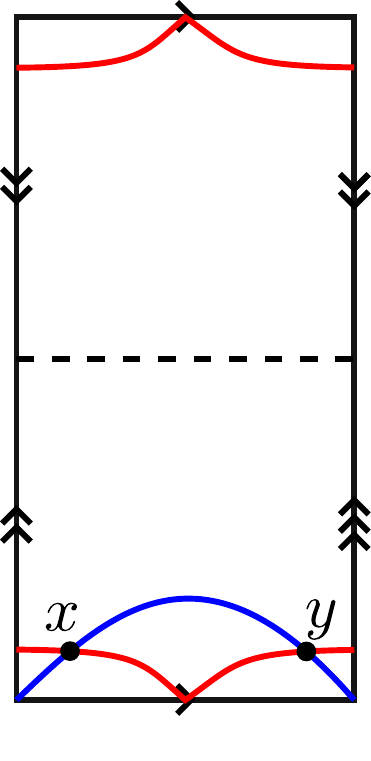}
        \caption{}
        \label{fig:UnlinkA}
    \end{subfigure}
    \hspace{.5cm}
    \begin{subfigure}[b]{0.2\textwidth}
        \centering
        \includegraphics[width=\textwidth]{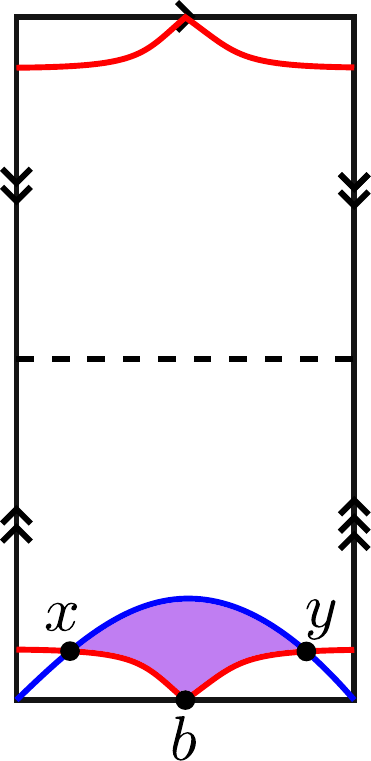}
        \caption{}
        \label{fig:UnlinkB}
    \end{subfigure}
    \caption{(A) shows a decomposition of the unlink $(S^3, U_2)$ into two copies of $Q_0$. (B) shows the maps of $\RR_\pi(Q_0)$ and $\RR^\natural(Q_0)$ into the pillowcase where $\pi$ is a shearing map. (C) shows the differential that appears when allowing for the bounding cochain $b$.}
    \label{fig:Unlinks}
\end{figure}

\begin{lemma}
\label{lem:trivialbc}
If Conjecture \ref{con:boundingcc} is true, the bounding cochain associated to a rational tangle with an earring (equivalent to $Q_0^\natural$) must be 0.
\end{lemma}

\begin{proof}
Consider the unlink $(S^3, U_2)$ with the tangle decomposition into two copies of $Q_0$ shown in Figure \ref{fig:Unlink}. Let $\pi$ be a shearing perturbation on one copy of $Q_0$. Figure \ref{fig:UnlinkA} shows the images of $\RR_\pi(Q_0)$ and $\RR^\natural(Q_0)$ in the pillowcase. $CF(\RR_\pi(Q_0), \RR^\natural(Q_0))$ is generated by the intersections $x$ and $y$ and there are no bigons so
$\rank(HF(\RR_\pi(Q_0), \RR^\natural(Q_0))) = 2$.

The image of $\RR_\pi(Q_0)$ in the pillowcase has no self intersections, so it must be assigned the trivial bounding cochain.
The image of $\RR^\natural(Q_0)$ in the pillowcase has a single self intersection. This means that there is one potential nontrivial bounding cochain (if it satisfies Equation \ref{eq:mc}) for $\RR^\natural(Q_0)$. Call this bounding cochain $b$.
$CF((\RR_\pi(Q_0), 0), (\RR^\natural(Q_0), b))$ is again generated by $x$ and $y$ but now contains exactly one differential shown in Figure \ref{fig:UnlinkB}. Thus $\rank(HF((\RR_\pi(Q_0),0), (\RR^\natural(Q_0), b))) = 0$.
It is known that $\rank(I^\natural(U_2)) = 2$ \cite{XieZhang}*{Lemma 4.3}. Therefore, for Conjecture \ref{con:boundingcc} to hold, the associated bounding cochain for $Q_0^\natural$ is trivial.
\end{proof}

For all examples computed in \cites{PCI, PCII, Earring}, if $(S^3, L) = (D^3, T_1)\cup_{(S^2, 4)}(D^3, T_2)$ where $(D^3, T_2)$ is a rational tangle, $HF((\RR_\pi(T_1), 0), (\RR^\natural_\pi(T_2),0))$ is isomorphic to the known or conjectured value of $I^\natural(S^3, L)$. This leaves open the possibility that the assignment of Conjecture \ref{con:boundingcc} assigns a trivial bounding cochain to each tangle $(D^3, T_1)$ with no earring. Theorem \ref{thm:main} will show that this cannot be the case.

\subsection{Some Simple Classes of Knots}

We will briefly review the definitions for some classes of knots for which the results of Sections \ref{sec:unpert} and \ref{sec:Perturbing} can be easily applied.

\begin{figure}
    \centering
    \begin{subfigure}[b]{0.22\textwidth}
        \centering
        \includegraphics[width=\textwidth]{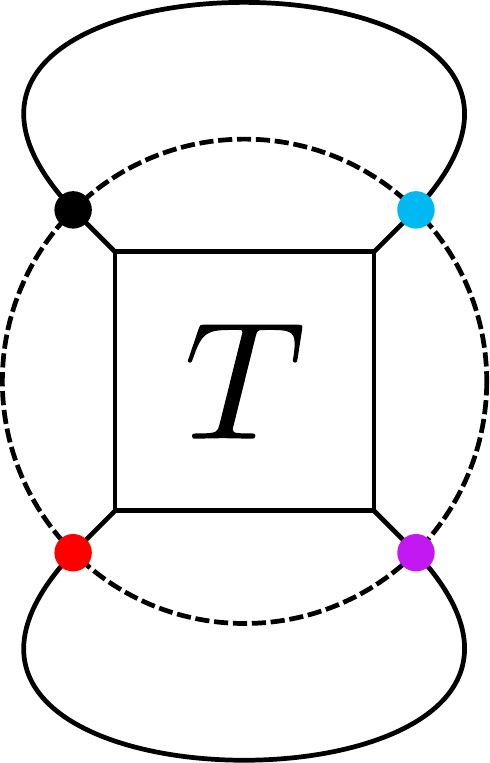}
        \caption{}
        \label{fig:num}
    \end{subfigure}
    \hspace{2cm}
    \begin{subfigure}[b]{0.38\textwidth}
        \centering
        \includegraphics[width=\textwidth]{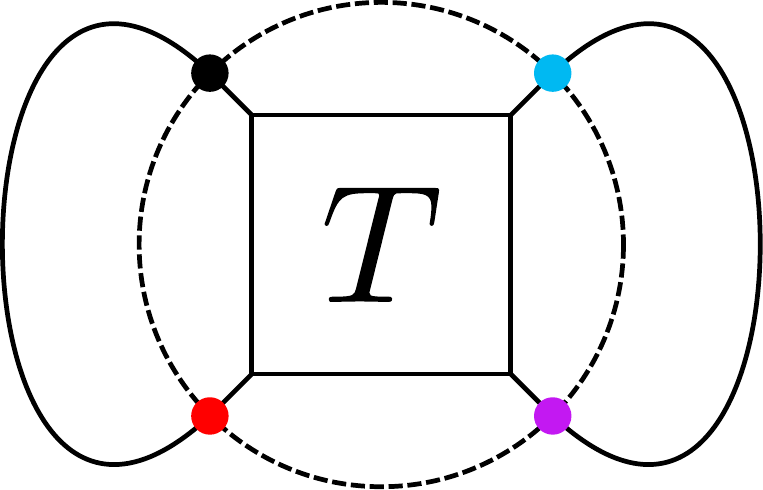}
        \caption{}
        \label{fig:den}
    \end{subfigure}
    \caption{The numerator and denominator closures of a tangle.}
\end{figure}

\begin{definition}
If $(X, T)$ is a tangle such that $\partial(X, T) = (S^2,4)$, then define the \emph{numerator closure} of a tangle $T$, $\operatorname{num}(T)$, to be the knot in Figure \ref{fig:num} and let the \emph{denominator closure} of $T$, $\operatorname{den}(T)$, be the knot in Figure \ref{fig:den}.
\end{definition}

\begin{definition}
An \emph{arborescent link} or \emph{arborescent tangle} is a link or tangle that can be decomposed into rational tangles and tangles homeomorphic to $C_3$ \cite{BonSieb}. 
\end{definition}

Similarly, the set of arborescent tangles are precisely those which can be obtained by starting with a collection of rational tangles and iteratively summing and composing with the tangle cobordism in Figure \ref{fig:rotate}. If $T$ is an arborescent tangle, then $\operatorname{num}(T)$ and $\operatorname{den}(T)$ are arborescent knots. Every arborescent link can be realized as the closure of an arborescent tangle.

\begin{remark}
Conway originally called such links and tangles \emph{algebraic}, but this terminology was changed in order to avoid confusion with links of singularities of algebraic curves.
\end{remark}

\begin{definition}
A special class of links within arborescent knots are \textit{Montesinos links}. Define the Montesinos knot $M(r_1, \dots, r_n)$ to be $\operatorname{num}(Q_{r_1} + \dots + Q_{r_n})$ for $r_i\in\Q\cup\{\infty\}$. Similarly, let a \emph{Montesinos tangle} be a tangle of the form $Q_{r_1} + \dots + Q_{r_n}$.
\end{definition}

A further subset of Montesinos links are \emph{pretzel links}. Define the pretzel link \[P(q_1, \dots, q_n) := M(\frac{1}{q_1}, \dots, \frac{1}{q_n}).\] As expected, a \emph{pretzel tangle} is a tangle of the form $Q_\frac{1}{q_1} + \dots + Q_\frac{1}{q_n}$.

\subsection{Computations}

At first the result of Theorem \ref{thm:cross} may seem less than helpful because it does not determine the number of auxiliary components. However, the following proposition shows that for computations of Lagrangian Floer homology, the auxiliary components may be ignored.

\begin{proposition}
Let $(T_1, \pi_1)$ and $(T_2, \pi_2)$ be a good pair such that $\RR_\pis(T_1+T_2)$ has no corner circles. By Theorem \ref{thm:cross} $\RR_{\t\cup\pis}(T_1+T_2)$ can be separated into main components and auxiliary components. Let $\RR^\dag_{\t\cup\pis}(T_1+T_2)$ be the main components of $\RR_{\t\cup\pis}(T_1+T_2)$. If $b$ is a bounding cochain on $\RR_{\t\cup\pis}(T_1+T_2)$, let $b^\dag$ be the restriction of $b$ to $\RR^\dag_{\t\cup\pis}(T_1+T_2)$. Let $(T_3, \pi_3)$ be another tangle with perturbation data such that $\RR_{\pi_3}(T_3)$ is a 1-manifold. Then there exists an arbitrarily small perturbation $\pi$ for $T_3$ such that \[HF((\RR_{\t\cup\pis}(T_1+T_2),b), (\RR_{\pi_3\cup \pi}(T_3),b')) = HF((\RR^\dag_{\t\cup\pis}(T_1+T_2),b), (\RR_{\pi_3}(T_3),b')).\]
This also applies if either of the tangles $T_1+T_2$ or $T_3$ has an earring.
\end{proposition}

\begin{proof}
By Theorem \ref{thm:cross} there are a finite number of points $\{p_1,\dots, p_n\}$ in the smooth stratum of the pillowcase, $P^*$, such that as $t\to 0$ all auxiliary components of $\RR_{\t\cup\pis}(T_1+T_2)$ limit to the set $\{p_1,\dots, p_n\}$. It is possible to find an arbitrarily small perturbation $\pi$ (in particular, one can construct a composition of shearing perturbations similar to the proof of Lemma \ref{lem:internal}) such that the image of $\RR_{\pi_3\cup \pi}(T_3)$ in the pillowcase misses the points $\{p_1,\dots, p_n\}$. Then there exists a small $t>0$ such that $\RR_{\pi_3\cup \pi}(T_3)$ misses the auxiliary components of $\RR_{\t\cup\pis}(T_1+T_2)$. The same idea works if one of the tangles has an earring.
\end{proof}

Call a 2-stranded tangle $T$ \textit{linear} if the image of $\RR(T)$ in the pillowcase is piecewise-linear. It was shown in \cite{FKP} that not all tangles are linear. However, they showed that certain tangles associated to pretzel knots are linear. This can be generalized by the following proposition:

\begin{proposition}
\label{lem:arbPL}
Arborescent tangles are linear with rational slopes. The endpoints of each linear segment lie in $(\pi\Q)^2$.
\end{proposition}

\begin{proof}
Arborescent tangles are constructed by starting with rational tangles and allowing the operations of rotating (composing with Figure \ref{fig:rotate}) and tangle addition. Proposition \ref{prop:rational} shows that rational tangles are linear with endpoints in $(\pi\Q)^2$. Lemma \ref{lem:rotate} shows that rotating preserves linearity and any endpoint will be sent from $(\gamma,\theta)$ to $(\theta,\gamma)$ and thus will stay in $(\pi\Q)^2$. Finally, Theorem \ref{thm:notnice} show that linearity is preserved under tangle addition and if the linear segments of the summands are in $(\pi\Q)^2$ then all linear segments of the sum have endpoints in $(\pi\Q)^2$.
\end{proof}

The fact that arborescent tangles are linear means that their character varieties are readily computed by computer. Recall that any pair of rational tangles $Q_\frac{p}{q}$ and $Q_\frac{p'}{q'}$ without perturbations form a good pair (unless $\frac{p}{q} = \frac{p'}{q'} = Q_\infty$). Furthermore, if $q$ and $q'$ are coprime $\RR(T_1+T_2)$ has no corner circles. Thus the limit of the perturbed character variety $\RR_\t(T_1+T_2)$ as $t\to 0$ has piecewise-linear image in the pillowcase (since the image is the same as the unperturbed character variety) and can also be computed by computer. The author's program \texttt{pcase} \cite{pcase} can do these computations as well as compute $CF(L_1, L_2)$ and $HF((L_1,0), (L_2,0))$ for piecewise-linear curves $L_1$ and $L_2$ in the pillowcase.
Thus $CF$ and $HF$ can be computed for certain decompositions of many arborescent knots.

{%

}

First we remark on an observation from computing the chain complexes $CF$ for a number of small knots. Let $(S^3, K)$ be a knot or link with decomposition $(B^3, T_1)\cup_{(S^2,4)}(B^3, T_2)$.
Recall that \[CF(\RR_{\pi_1}(T_1), \RR^\natural_{\pi_2}(T_2))\cong CI^\natural(K).\] Thus $\rank(I^\natural(K))\le \rank(CF(\RR_{\pi_1}(T_1), \RR^\natural_{\pi_2}(T_2)))$. This fact was applied to torus knots in \cite{PCI} to find new bounds for $I^\natural$. In several cases, combining this bound with the lower bound $l(K)$, $I^\natural$ was identified exactly.

It is observed in that \cite{LZ} that for computed examples of pretzel knots, the upper bound coming from $\rank(CF(\RR_{\pi_1}(T_1), \RR_{\pi_2}(T_2)))$ is no better than the bound coming from Khovanov homology. Using \texttt{pcase} and \texttt{khoca} \cite{khoca}, the two upper bounds have been computed for all 3-bridge pretzel knots $P(p,q,r)$ with $\abs{p}$, $\abs{q}$, $\abs{r}\le 13$ and $p$, $q$, and $r$ pairwise coprime. For all such knots, the computations agreed with the observation of \cite{LZ}. Additionally, these computations were computed for other small 3-bridge Montesinos knots and small arborescent knots of the form $\operatorname{num}(Q_a+Q_b+\frac{1}{-(Q_c+Q_d)})$ for $a,b,c,d\in\Q$. There were no cases where the bound from $\rank (CF(\RR_{\pi_1}(T_1), \RR_{\pi_2}(T_2)))$ was strictly better than the bound coming from Khovanov homology, suggesting that the pattern noticed in \cite{LZ} for pretzel knots extends to arborescent knots more generally.

Finally, we use a computation of $HF$ to show that the assignment of bounding cochain $b$ for tangles without earring from Conjecture \ref{con:boundingcc} must be nontrivial if the conjecture is to hold, which was speculated to be the case in \cite{Earring}*{Section 11.1}.

\begin{figure}
    \centering
    \begin{subfigure}[b]{0.45\textwidth}
        \centering
        \includegraphics[width=\textwidth]{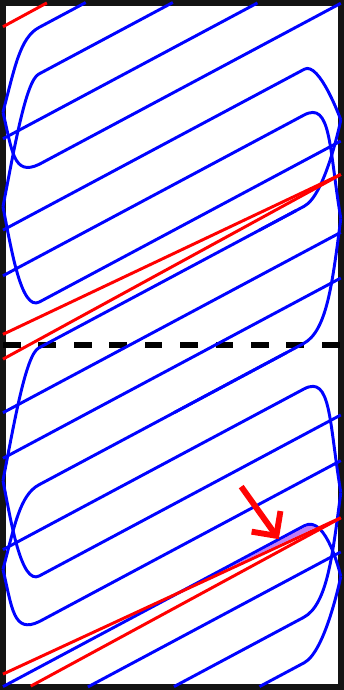}
        \caption{}
        \label{fig:2bigonsA}
    \end{subfigure}
    \hspace{1cm}
    \begin{subfigure}[b]{0.45\textwidth}
        \centering
        \includegraphics[width=\textwidth]{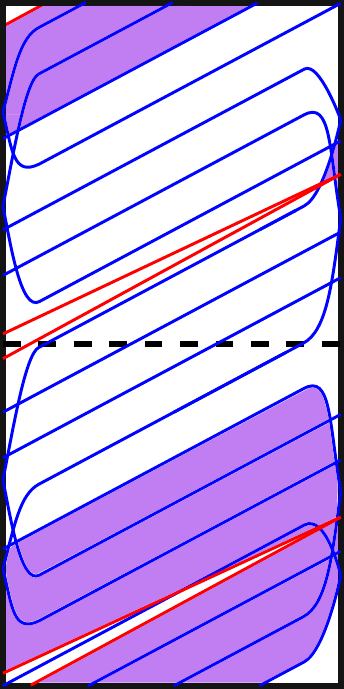}
        \caption{}
        \label{fig:2bigonsB}
    \end{subfigure}
    \caption{This depicts the calculation of the Lagrangian Floer homology associated to a particular decomposition of $P(-2,3,5)$. The Lagrangian intersections of $\RR_\pi^\natural(\widehat{Q_{-\frac{1}{2}}})$ (red) and $\RR_t(Q_\frac{1}{3}+Q_\frac{1}{5})$ (blue). Both figures show the same Lagrangians, but highlight a different bigons in purple. In total there are nine intersection points and two immersed bigons whose vertices are sent to distinct intersections, giving $\rank(HF(\RR_\pi^\natural(\widehat{Q_{-\frac{1}{2}}}), \RR_t(Q_\frac{1}{3}+Q_\frac{1}{5}))) = 5$.}
    \label{fig:2bigons}
\end{figure}

\begin{theorem}
\label{thm:main}
If Conjecture \ref{con:boundingcc} holds, then the assignment of bounding cochains to tangles cannot be trivial on every tangle without an earring. 
\end{theorem}

\begin{proof}
Let $K = P(-2,3,5)$. Consider the decomposition of $K$  into $T_1=\widehat{Q_{-\frac{1}{2}}}$ and $T_2 = Q_\frac{1}{3}+Q_\frac{1}{5}$ (see Figure \ref{fig:ExDecomp}). Assume $b_1$ and $b_2$ are the bounding cochains of $T_1$ and $T_2$ satisfying Conjecture \ref{con:boundingcc}. Since we are going to add the earring to $T_1$, Lemma \ref{lem:trivialbc} implies that $b_1 = 0$. Thus, $HF((\RR^\natural(T_1), 0), (\RR_\Pert(T_2), b_2)) = I^\natural(K)$. Calculating the reduced Khovanov homology and Alexander polynomials for $K$ gives the bounds $l(K) = u(K) = 7$ and thus $\rank(I^\natural(K)) = 7$. So to prove the theorem we need to show that $b_2\ne 0$, which can be done by showing $HF((\RR^\natural(T_1), 0), (\RR_\t(T_2), 0)) \ne 7$. To do this, we first compute the images of $\RR^\natural(T_1)$ and $\RR_\t(T_2)$ in the pillowcase using Theorems \ref{thm:nice} and \ref{thm:cross}, which is shown in Figure \ref{fig:ExDecomp}. These Lagrangians have nine intersection points, so $\rank(CF((\RR^\natural(T_1), 0), (\RR_\Pert(T_2), 0))) = 9$. Next we must identify all immersed bigons which give the differentials of the chain complex. There are two such bigons which are shown in Figure \ref{fig:2bigons}. Note that the vertices of these bigons are distinct, so taking the homology of our chain complex gives $\rank(HF_1(K)) = HF((\RR^\natural(T_1), 0), (\RR_\Pert(T_2), 0)) = 5$. Thus, $b_2\ne 0$.
\end{proof}

Computations of $HF$ including the one in the preceding proof give specific examples of tangles for which the associated bounding cochain must be nontrivial. Finding which choices of bounding cochain can correct the rank of $HF$ to match that of $I^\natural$ should give insight into Conjecture \ref{con:boundingcc}.

\renewcommand{\bibname}{References}

\bibliographystyle{alpha}
\bibliography{bib}

\end{document}

%% file: preamble.tex
\usepackage[utf8]{inputenc}
\usepackage{graphicx}
\usepackage{amsopn}
\usepackage{amsthm}
\usepackage{amsmath}
\usepackage{amsfonts}
\usepackage[alphabetic]{amsrefs}
\usepackage{amssymb}
\usepackage{tikz}
\usepackage{caption}
\usepackage{subcaption}
\usepackage{enumerate}
\usepackage[shortlabels]{enumitem}
\usepackage[title]{appendix}
\usepackage{todonotes}
\usepackage{tikz}
\usepackage{tikz-cd}
\usepackage{hyperref}

\newcommand{\SU}{\mathrm{SU(2)}}
\newcommand{\U}{\mathrm{U(1)}}
\newcommand{\su}{\mathrm{su(2)}}
\renewcommand{\Re}{\operatorname{Re}}
\renewcommand{\Im}{\operatorname{Im}}
\newcommand{\tr}{\operatorname{tr}}
\newcommand{\Hom}{\operatorname{Hom}}
\newcommand{\abs}[1]{\vert #1\vert}
\newcommand{\ol}[1]{\overline{#1}}
\newcommand{\R}{\mathbb{R}}
\newcommand{\Z}{\mathbb{Z}}

\newcommand{\Q}{\mathbb{Q}}
\newcommand{\inv}{^{-1}}

\renewcommand{\P}{\widetilde{P}} %
\renewcommand{\Pr}{\mathcal{P}} %
\newcommand{\SV}{\mathcal{S}} %
\newcommand{\NS}{\mathcal{N}} %
\newcommand{\MS}{\mathcal{M}} %
\newcommand{\CSS}{\mathcal{C}} %
\newcommand{\I}{\mathcal{I}} %
\newcommand{\B}{\mathcal{B}} %
\newcommand{\RR}{R} %
\newcommand{\Ro}{\widetilde{R}} %
\newcommand{\F}{\Phi} %
\newcommand{\pcs}{w} %
\newcommand{\Pert}{D}
\renewcommand{\t}{{\Pert_t}}
\newcommand{\SOne}{S_1}
\newcommand{\STwo}{S_2}
\newcommand{\SThree}{S_3}
\newcommand{\POne}{P_1}
\newcommand{\PTwo}{P_2}
\newcommand{\PThree}{P_3}
\renewcommand{\H}{\mathcal{H}}
\newcommand{\A}{\mathcal{A}}
\newcommand{\J}{\mathcal{J}}
\newcommand{\K}{\mathcal{K}}
\newcommand{\grho}{}%
\newcommand{\twedge}{\sqcup}
\newcommand{\Sols}{\mathcal{V}}
\newcommand{\bmu}{{\boldsymbol{\mu}}}
\newcommand{\pis}{{\pi_1\cup\pi_2}}
\newcommand{\pt}{\text{a point}} %

\newcommand{\bd}{^\text{bd}}
\newcommand{\nbd}{^\text{nbd}}
\newcommand{\bmat}[1]{\begin{bmatrix}#1\end{bmatrix}}
\newcommand{\ang}[2]{\angle #1 #2}
\newcommand{\rang}[3]{\angle_{#3} #1 #2}
\DeclareMathOperator{\rank}{rank}
\DeclareMathOperator{\sinc}{sinc}
\DeclareMathOperator{\sign}{sign}
\DeclareMathOperator{\gr}{gr}
\DeclareMathOperator{\Kh}{Khr}
\DeclareMathOperator{\CS}{CS}
\DeclareMathOperator{\Stab}{Stab}
\newcommand{\wStab}{\widetilde{\Stab}}
\newcommand{\Sk}{S^1_{\k^\perp}}

\newcommand{\fcirc}[1]{\tikz\draw[#1,fill=#1] (0,0) circle (.5ex);}

\renewcommand{\i}{\textbf{i}}
\renewcommand{\j}{\textbf{j}}
\renewcommand{\k}{\textbf{k}}

\numberwithin{equation}{subsection}
\newtheorem{theorem}{Theorem}[section]
\newtheorem*{theorem*}{Theorem}
\newtheorem{lemma}[theorem]{Lemma}
\newtheorem{corollary}[theorem]{Corollary}
\newtheorem*{corollary*}{Corollary}
\newtheorem{conjecture}[theorem]{Conjecture}
\newtheorem{proposition}[theorem]{Proposition}
\newtheorem{innercustomthm}{Theorem}
\newtheorem{innercustomthmb}{Rough Idea of Theorem}
\newtheorem{innercustomthmc}{Conjecture}
\newenvironment{mythm}[1]
  {\renewcommand\theinnercustomthm{#1}\innercustomthm}
  {\endinnercustomthm}
\newenvironment{idea}[1]
  {\renewcommand\theinnercustomthmb{#1}\innercustomthmb}
  {\endinnercustomthm}
  \newenvironment{myconjecture}[1]
  {\renewcommand\theinnercustomthmc{#1}\innercustomthmc}
  {\endinnercustomthm}
\theoremstyle{definition}
\newtheorem{exmp}[theorem]{Example}
\newtheorem*{remark}{Remark}
\newtheorem{definition}[theorem]{Definition}